\documentclass[11pt]{amsart}
\usepackage{amssymb}
\usepackage{amsmath}
\usepackage{amsthm}
\usepackage{enumerate}
\usepackage{mathrsfs}
\newtheorem{theorem}{Theorem}[section]
\newtheorem{lemma}[theorem]{Lemma}
\newtheorem{corollary}[theorem]{Corollary}
\theoremstyle{definition}
\newtheorem*{example}{Example}
\newtheorem*{remark}{Remark}

\numberwithin{equation}{section}
\mathchardef\hyphen="2D
\begin{document}
\allowdisplaybreaks
\title{Similarity of operators on $l^{p}$}
\author{March T.~Boedihardjo}
\address{Department of Mathematics, University of California, Los Angeles, CA 90095-1555}
\email{march@math.ucla.edu}
\begin{abstract}
For $1<p<\infty$, we prove (i) a version of Voiculescu's absorption theorem for operators on $l^{p}$, (ii) that $\mathrm{Ext}_{\sim,s}(\mathcal{A},K(l^{p}))$ is a group for certain Banach algebra $\mathcal{A}$, and (iii) homotopy invariance of $\mathrm{Ext}_{\sim,s}(\mathcal{A},K(l^{p}))^{-1}$ in $\mathcal{A}$ for separable Banach algebra $\mathcal{A}$ that is isomorphic to a subalgebra of $B(l^{p})$.
\end{abstract}
\maketitle
\tableofcontents
\section{Introduction}
Over the past 80 years, certain classical results about operators on Hilbert spaces have been extended to operators on more general Banach spaces. Mean ergodic theorem for reflexive Banach spaces is established \cite{Lorch} (for $p=2$, see \cite{Riesz}); commutators on $c_{0}$, $l^{p}$ and $L^{p}$, for $1\leq p\leq\infty$, are characterized \cite{Apostol1}, \cite{Apostol2}, \cite{Dosev1}, \cite{Dosev2}, \cite{Dosev3} (for $p=2$, see \cite{BrownPearcy}); quasitriagularity of operators on $c_{0}$ and $l^{p}$, for $1\leq p<\infty$, is characterized \cite{Apostol4} (for $p=2$, see \cite{Apostol3}); West's result \cite{West} on Riesz operators on Hilbert spaces is generalized to $c_{0}$ and $l^{p}$, for $1\leq p<\infty$ \cite{Davidson}. This list is by no means complete. In fact, there is a recent development of $L^{p}$-operator algebras since \cite{PhillipsCuntz}, \cite{PhillipsCrossProduct}, \cite{PhillipsIsomorphism}.

The purpose of this paper is to study the extent to which results concerning unitary equivalence, up to a small perturbation, of operators on Hilbert spaces hold for operators on $l^{p}$ for $1<p<\infty$. We obtain an $l^{p}$ version of Voiculescu's absorption theorem \cite{Voiculescu}. We introduce a notion of $\mathrm{Ext}_{\sim,s}(\mathcal{A},K(l^{p}))$ for Banach algebra $\mathcal{A}$ that is isomorphic to a subalgebra of $B(l^{p})$, where $K(l^{p})$ is the algebra of compact operators on $l^{p}$. We show that $\mathrm{Ext}_{\sim,s}(\mathcal{A},K(l^{p}))$ is a group when $\mathcal{A}=C(M)$, for any compact metric space $M$ that is homeomorphic to a subset of a Euclidean space, or when $\mathcal{A}$ is the algebra generated by the range of the left regular representation of a countable amenable group on $l^{p}$. Hence, a lifting theorem for homomorphisms from $\mathcal{A}$ into $B(l^{p})/K(l^{p})$ is obtained for these algebras $\mathcal{A}$. We also prove homotopy invariance of $\mathrm{Ext}_{\sim,s}(\mathcal{A},K(l^{p}))^{-1}$ for separable Banach algebra $\mathcal{A}$ that is isomorphic to a subalgebra of $B(l^{p})$. The $p=2$ case of these results are proved in \cite{Brown}, \cite{Arveson}, \cite{Kasparov}.

In the study of single operators, the following consequences are obtained for $1<p<\infty$: (1) the unilateral shift $U$ on $l^{p}$ is approximately similar to the direct sum of the unilateral shift and the bilateral shift on $l^{p}$; (2) the bilateral shift $B$ on $l^{p}$ is approximately similar to the direct sum of circular shifts on finite dimensional $l^{p}$ spaces; (3) there exist $u_{1},u_{2},\ldots\in[0,1]$ such that $U\oplus(\oplus_{k=1}^{\infty}u_{k}B)$ is similar to a compact perturbation of $\oplus_{k=1}^{\infty}u_{k}B$. (For $p=2$, this is a consequence of the Brown-Douglas-Fillmore theorem.)

Note that we have relaxed unitary equivalence to similarity. This is due to the lack of invertible isometries on $l^{p}$: every invertible isometry on $l^{p}$, for $p\in[1,\infty]\backslash\{2\}$, is the composition of a diagonal operator with entries on the unit circle and a permutation operator \cite{Lamperti}.

In the negative direction, we point that some basic results concerning unitary equivalence of operators on Hilbert spaces do not hold for some more general Banach spaces even if we relax unitary equivalence to similarity. The bilateral shift on $l^{2}$ is unitarily equivalent to the sum of a diagonal operator on $l^{2}$ and a compact operator. However, the bilateral shift on $l^{p}$, for $p\in[1,\infty]\backslash\{2\}$, does not have functional calculus for functions in $C(S^{1})$ \cite{Fixman} unlike unitary operators on Hilbert spaces. Here $S^{1}$ is the unit circle on $\mathbb{C}$. As a consequence, the bilateral shift on $l^{p}$, for $p\in(1,\infty)\backslash\{2\}$, is not similar to the sum of a diagonal operator on $l^{p}$ and a compact operator (Corollary \ref{bilateralnondiag}).

Normal operators on a separable complex Hilbert space that have the same essential spectrum are unitarily equivalent modulo compact operators \cite{Berg}. When $p\in(1,\infty)\backslash\{2\}$, if $\mu_{1}$ and $\mu_{2}$ are purely nonatomic mutually singular measures on $[0,1]$, the multiplication operators $M_{\mu_{1}}\in B(L^{p}(\mu_{1}))$ and $M_{\mu_{2}}\in B(L^{p}(\mu_{2}))$ defined by $(M_{\mu_{1}}f)(z)=zf(z)$, for $f\in L^{p}(\mu_{1})$, and $(M_{\mu_{2}}f)(z)=zf(z)$, for $f\in L^{p}(\mu_{2})$, are not similar modulo compact operators \cite{Boedihardjo} even if $\mu_{1}$ and $\mu_{2}$ have the same support.

All trivial extensions of $K(l^{2})$ by a separable $C^{*}$-algebra $\mathcal{A}$ are equivalent \cite{Voiculescu} (when $\mathcal{A}$ is commutative, this is proved in \cite{Brown}). In this paper, we show that when $p\in(1,\infty)\backslash\{2\}$, there are bounded below homomorphisms $\phi_{1}:C[0,1]\to B(l^{p})$ and $\phi_{2}:C[0,1]\to B(l^{p})$ such that $\phi_{1}(z)$ and $\phi_{2}(z)$ are not similar modulo compact operators, where $z\in C[0,1]$ is the identity function.

In Sections 2-6, we prove some preliminary results and introduce some new notions that are needed for Sections 7-10. In Section 7, we state and prove an $l^{p}$ version of Voiculescu's absorption theorem (Theorem \ref{lpVoic}) and obtain various consequences. In Section 8, we define $\mathrm{Ext}_{\sim,s}(\mathcal{A},K(l^{p}))$ and provide nontrivial examples of isomorphic extensions of $K(l^{p})$ by some Banach algebras. In Section 9, we show that $\mathrm{Ext}_{\sim,s}(\mathcal{A},K(l^{p}))$ is a group for certain Banach algebra $\mathcal{A}$ (Theorems \ref{main21} and \ref{main22}). In Section 10, we prove homotopy invariance of $\mathrm{Ext}_{\sim,s}(\mathcal{A},K(l^{p}))^{-1}$ for separable Banach algebra $\mathcal{A}$ that is isomorphic to a subalgebra of $B(l^{p})$ (Corollary \ref{homotopyinv2}).

Throughout this paper, unless stated otherwise, the scalar field is $\mathbb{C}$ and $1<p<\infty$. If $\mathcal{X}$ is a Banach space, $B(\mathcal{X})$ denotes the algebra of operators on $\mathcal{X}$ and $K(\mathcal{X})$ denotes the ideal of compact operators on $\mathcal{X}$. The symmetric difference between two sets $F_{1}$ and $F_{2}$ is denoted by $F_{1}\Delta F_{2}$.

If $F$ is a set then $l^{p}(F)=\{x:F\to\mathbb{C}:\sum_{i\in F}|x(i)|^{p}<\infty\}$ is the $l^{p}$ space on $F$. When $F$ is an interval, $l^{p}(F)$ is understood as $l^{p}(F\cap\mathbb{Z})$. For instance, $l^{p}([1,3])$ is a 3-dimensional $l^{p}$ space. If $F$ is the empty set, $l^{p}(F)=\{0\}$. The canonical basis for $l^{p}(F)$ is denoted by $(e_{j})_{j\in F}$ and $(e_{j}^{*})_{j\in F}$ are the coordinate functionals on $l^{p}(F)$, i.e., $e_{j}^{*}(e_{i})=1$ if $i=j$ and 0 if $i\neq j$. If $M$ is a compact metric space, then $C(M)$ is the algebra of scalar valued continuous functions on $M$.

If $\mathcal{X}_{1},\mathcal{X}_{2},\ldots$ are Banach spaces, then $(\oplus_{n=1}^{\infty}\mathcal{X}_{n})_{l^{p}}$ is the Banach space \[\left\{(x_{1},x_{2},\ldots):x_{n}\in\mathcal{X}_{n},\text{ for }n\in\mathbb{N},\text{ and }\sum_{n=1}^{\infty}\|x_{n}\|^{p}<\infty\right\}\]
with norm
\[\|(x_{1},x_{2},\ldots)\|=\left(\sum_{n=1}^{\infty}\|x_{n}\|^{p}\right)^{\frac{1}{p}}.\]
For each $n\in\mathbb{N}$, let $T_{n}$ be an operator on $\mathcal{X}_{n}$. Assume that $\sup_{n\in\mathbb{N}}\|T_{n}\|<\infty$. Then $T_{1}\oplus T_{2}\oplus\ldots$ is the operator on $(\oplus_{n\in\mathbb{N}}\mathcal{X}_{n})_{l^{p}}$ defined by $(T_{1}\oplus T_{2}\oplus\ldots)(x_{1},x_{2},\ldots)=(T_{1}x_{1},T_{2}x_{2},\ldots)$.

If $\mathcal{X}_{1}$ and $\mathcal{X}_{2}$ are Banach spaces and $T:\mathcal{X}_{1}\to\mathcal{X}_{2}$ is an operator, then $\|T\|_{e}$ is the infimum of $\|T+K\|$ over all compact operator $K:\mathcal{X}_{1}\to\mathcal{X}_{2}$.

A map is a function where no continuity or algebraic property is assumed. Let $\mathcal{A}_{1}$ and $\mathcal{A}_{2}$ be Banach algebras. A {\it homomorphism} $\psi:\mathcal{A}_{1}\to\mathcal{A}_{2}$ is a bounded linear map such that $\psi(ab)=\psi(a)\psi(b)$ for all $a,b\in\mathcal{A}_{1}$. The Banach algebras $\mathcal{A}_{1}$ and $\mathcal{A}_{2}$ are {\it isomorphic} if there exist a bijective homomorphism $\psi:\mathcal{A}_{1}\to\mathcal{A}_{2}$ and $C\geq 1$ such that
\[\frac{1}{C}\|a\|\leq\|\psi(a)\|\leq C\|a\|,\]
for all $a\in\mathcal{A}_{1}$.

Let $\mathcal{X}_{1}$ and $\mathcal{X}_{2}$ be Banach spaces. Let $\lambda\geq 1$. The spaces $\mathcal{X}_{1}$ and $\mathcal{X}_{2}$ are $\lambda$-{\it isomorphic} if there is an invertible operator $S:\mathcal{X}_{1}\to\mathcal{X}_{2}$ such that $\|S\|\|S^{-1}\|\leq\lambda$. The spaces $\mathcal{X}_{1}$ and $\mathcal{X}_{2}$ are {\it isomorphic} if they are $\lambda$-isomorphic for some $\lambda\geq 1$. Let $\Lambda$ be a set. Let $\psi_{1}:\Lambda\to B(\mathcal{X}_{1})$ and $\psi_{2}:\Lambda\to B(\mathcal{X}_{2})$ be maps. The maps $\psi_{1}$ and $\psi_{2}$ are $\lambda$-{\it similar} if there is an invertible operator $S:\mathcal{X}_{1}\to\mathcal{X}_{2}$ such that $\|S\|\|S^{-1}\|\leq\lambda$ and $\psi_{2}(\alpha)=S\psi_{1}(\alpha)S^{-1}$ for all $\alpha\in\Lambda$. The maps $\psi_{1}$ and $\psi_{2}$ are $\lambda$-{\it approximately similar} if there are invertible operators $S_{n}:\mathcal{X}_{1}\to\mathcal{X}_{2}$, for $n\in\mathbb{N}$, such that
\begin{enumerate}[(i)]
\item $\sup_{n\in\mathbb{N}}\|S_{n}\|\|S_{n}^{-1}\|<\infty$;
\item $\psi_{2}(\alpha)-S_{n}\psi_{1}(\alpha)S_{n}^{-1}$ is compact for all $n\in\mathbb{N}$ and $\alpha\in\Lambda$; and
\item $\lim_{n\to\infty}\|\psi_{2}(\alpha)-S_{n}\psi_{1}(\alpha)S_{n}^{-1}\|=0$ for all $\alpha\in\Lambda$.
\end{enumerate}
(Hadwin \cite{Hadwin} defines approximate similarity without condition (ii) but the author \cite{Boedihardjo} shows that these two notions coincide for single operators.) The maps $\psi_{1}$ and $\psi_{2}$ are $\lambda$-{\it similar modulo compact operators} if there is an invertible operator $S:\mathcal{X}_{1}\to\mathcal{X}_{2}$ such that $\|S\|\|S^{-1}\|\leq\lambda$ and $\psi_{2}(\alpha)-S\psi_{1}(\alpha)S^{-1}$ is compact for all $\alpha\in\Lambda$. The maps $\psi_{1}$ and $\psi_{2}$ are {\it similar/approximately similar/similar modulo compact operators} if they are $\lambda$-similar/$\lambda$-approximately similar/$\lambda$-similar modulo compact operators for some $\lambda\geq 1$.

For each $n\in\mathbb{N}$, let $\mathcal{X}_{n}$ be a Banach space and let $\psi_{n}:\Lambda\to B(\mathcal{X}_{n})$ be a map. The {\it direct sum} $\oplus_{n=1}^{\infty}\psi_{n}:\Lambda\to B((\oplus_{n=1}^{\infty}\mathcal{X}_{n})_{l^{p}})$ is defined by $(\oplus_{n=1}^{\infty}\psi_{n})(\alpha)=\oplus_{n=1}^{\infty}\psi_{n}(\alpha)$ for $\alpha\in\Lambda$.

Except for Section 10, the symbol $\pi$ always denotes the quotient map from $B(\mathcal{X})$ onto $B(\mathcal{X})/K(\mathcal{X})$ for some Banach space $\mathcal{X}$. Except for Sections 8 and 10, by abuse of notation, we use the same $\pi$ even for two different Banach spaces $\mathcal{X}_{1}$ and $\mathcal{X}_{2}$. Moreover, if $T_{1}\in B(\mathcal{X}_{1})$ and $T_{2}\in B(\mathcal{X}_{2})$ then we define $\pi(T_{1})\oplus\pi(T_{2})=\pi(T_{1}\oplus T_{2})\in B(\mathcal{X}_{1}\oplus\mathcal{X}_{2})/K(\mathcal{X}_{1}\oplus\mathcal{X}_{2})$. In Section 8, we use different notation for quotient maps onto the Calkin algebras of different Banach spaces. In Section 10, the symbol $\pi$ denotes the quotient map from $B(\mathcal{X}_{1},\mathcal{X}_{2})$ onto $B(\mathcal{X}_{1},\mathcal{X}_{2})/K(\mathcal{X}_{1},\mathcal{X}_{2})$ for some Banach spaces $\mathcal{X}_{1},\mathcal{X}_{2}$, where $B(\mathcal{X}_{1},\mathcal{X}_{2})$ is the Banach space of operators from $\mathcal{X}_{1}$ to $\mathcal{X}_{2}$ and $K(\mathcal{X}_{1},\mathcal{X}_{2})$ is the subspace of compact operators. Moreover, if $T_{1}\in B(\mathcal{X}_{1},\mathcal{X}_{2})$ and $T_{2}\in B(\mathcal{X}_{2},\mathcal{X}_{3})$ then we define $\pi(T_{2})\pi(T_{1})=\pi(T_{2}T_{1})\in B(\mathcal{X}_{1},\mathcal{X}_{3})$.

For a topological vector space $\mathcal{Y}$, a countable infinite set $\Lambda$ and elements $x_{\alpha}\in\mathcal{Y}$, for $\alpha\in\Lambda$, a series $\sum_{\alpha\in\Lambda}x_{\alpha}$ {\it converges unconditionally} if the series $\sum_{i=1}^{\infty}x_{g(i)}$ converges to the same element $x\in\mathcal{Y}$ for all bijection $g:\mathbb{N}\to\Lambda$.

If $\mathcal{X}$ is a Banach space then its dual space is denoted by $\mathcal{X}^{*}$. If $T\in B(\mathcal{X})$ then $T^{*}\in B(\mathcal{X}^{*})$ is defined by $(T^{*}x^{*})(x)=x^{*}(Tx)$ for $x^{*}\in\mathcal{X}^{*}$ and $x\in\mathcal{X}$. The strong operator topology on $B(\mathcal{X})$ is denoted by SOT and the weak operator topology is denoted by WOT.

For $x^{*}\in\mathcal{X}^{*}$, we define a seminorm $|\;|_{x^{*}}$ on $B(\mathcal{X})$ by
\[|T|_{x^{*}}=\|T^{*}x^{*}\|=\sup_{x\in\mathcal{X},\,\|x\|=1}|x^{*}(Tx)|,\]
for $T\in B(\mathcal{X})$. It is easy to see that $|T_{1}T_{2}|_{x^{*}}=|T_{2}|_{T_{1}^{*}x^{*}}$ and \[|T_{1}T_{2}|_{x^{*}}\leq|T_{1}|_{x^{*}}\|T_{2}\|\leq\|x^{*}\|\|T_{1}\|\|T_{2}\|,\]
for all $T_{1},T_{2}\in B(\mathcal{X})$.

More terminologies and notation will be introduced at the beginning of some later sections.
\begin{center}
\begin{tabular}{ll}
$\pi$&Section 1\\
$|\;|_{x^{*}}$&Section 1\\
$\mathrm{supp}(A)$&Section 2\\
refined diagonal approximate identity&Section 3\\
$h(\pi(A))$&Section 4\\
$\mathcal{Y}_{u}^{\oplus k}$&Section 5\\
$(T_{1}\oplus\ldots\oplus T_{k})_{u}$&Section 5\\
$\phi\stackrel{\lambda}{\ll}\rho$&Section 6\\
$\mathrm{Ext}_{\sim,s}(\mathcal{A},K(l^{p}))$&Section 8\\
$\psi\stackrel{\lambda}{\prec}(\eta_{n})_{n\in\mathbb{N}}$&Section 9\\
$(l^{p})^{(\infty)}$ and $T^{(\infty)}$&Section 10
\end{tabular}
\end{center}
\section{Estimates and partitions of unity}
A {\it diagonal operator} $A$ on $l^{p}$ is an operator of the form
\[A(x_{1},x_{2},\ldots)=(w_{1}x_{1},w_{2}x_{2},\ldots),\]
for some $w_{1},w_{2},\ldots\in\mathbb{C}$, which are the {\it diagonal entries} of $A$. The {\it support} of $A$ is the set $\mathrm{supp}(A)=\{i\in\mathbb{N}:a_{i}\neq 0\}$.

\begin{lemma}\label{Diagest}
Let $r\in\mathbb{N}$. Let $(D_{n})_{n\in\mathbb{N}}$ be a sequence of diagonal operators on $l^{p}$ of norms at most $1$ such that every $j\in\mathbb{N}$ is contained in at most $r$ of the sets $\mathrm{supp}\,D_{1},\mathrm{supp}\,D_{2},\ldots$. Then
\begin{enumerate}[(i)]
\item \[\left\|\sum_{i=1}^{\infty}D_{i}x_{i}\right\|\leq r\left(\sum_{i=1}^{\infty}\|D_{i}x_{i}\|^{p}\right)^{\frac{1}{p}},\]
for all bounded sequence $(x_{n})_{n\in\mathbb{N}}$ in $l^{p}$ such that $\sum_{i=1}^{\infty}\|D_{i}x_{i}\|^{p}$ is finite;
\item \[\sum_{i=1}^{\infty}\|D_{i}x\|^{p}\leq r\|x\|^{p},\]
for all $x\in l^{p}$;
\item \[\left\|\sum_{i=1}^{k}D_{i}T_{i}D_{i}\right\|\leq r^{2}\max_{1\leq i\leq k}\|D_{i}T_{i}\|,\]
for all $T_{1},\ldots,T_{k}\in B(l^{p})$ and $k\in\mathbb{N}$; and
\item \[\left\|\sum_{i=1}^{k}D_{i}T_{i}D_{i}\right\|_{e}\leq r^{2}\max_{1\leq i\leq k}\|D_{i}T_{i}\|_{e},\]
for all $T_{1},\ldots,T_{k}\in B(l^{p})$ and $k\in\mathbb{N}$.
\end{enumerate}
\end{lemma}
\begin{proof}
We have
\[\left\|\sum_{i=1}^{\infty}D_{i}x_{i}\right\|^{p}=\sum_{j=1}^{\infty}\left|\sum_{i=1}^{\infty}e_{j}^{*}(D_{i}x_{i})\right|^{p}\leq
\sum_{j=1}^{\infty}\left(\sum_{i=1}^{\infty}|e_{j}^{*}(D_{i}x_{i})|\right)^{p}.\]
For all $j\in\mathbb{N}$ and $(x_{n})_{n\in\mathbb{N}}$ in $l^{p}$, since $e_{j}^{*}(D_{i}x_{i})\neq 0$ for at most $r$ values of $i$,
\[\left(\sum_{i=1}^{\infty}|e_{j}^{*}(D_{i}x_{i})|\right)^{p}\leq r^{p}\sum_{i=1}^{\infty}|e_{j}^{*}(D_{i}x_{i})|^{p}.\]
Therefore,
\[\left\|\sum_{i=1}^{\infty}D_{i}x_{i}\right\|^{p}\leq r^{p}\sum_{j=1}^{\infty}\sum_{i=1}^{\infty}|e_{j}^{*}(D_{i}x_{i})|^{p}=r^{p}\sum_{i=1}^{\infty}\|D_{i}x_{i}\|^{p}.\]
This proves (i). Let $x\in l^{p}$. For each $j\in\mathbb{N}$, since $e_{j}^{*}(D_{i}x)\neq 0$ for at most $r$ values of $i$,
\[\sum_{i=1}^{\infty}\|D_{i}x\|^{p}=\sum_{j=1}^{\infty}\sum_{i=1}^{\infty}|e_{j}^{*}(D_{i}x)|^{p}\leq\sum_{j=1}^{\infty}r\max_{1\leq i\leq n}|e_{j}^{*}(D_{i}x)|^{p}.\]
Since each $D_{i}$ is a diagonal operator of norm at most 1, we have $\displaystyle\max_{1\leq i\leq n}|e_{j}^{*}(D_{i}x)|\leq|e_{j}^{*}(x)|$. Therefore,
\[\sum_{i=1}^{k}\|D_{i}x\|^{p}\leq r\sum_{j=1}^{\infty}|e_{j}^{*}(x)|^{p}=r\|x\|^{p}.\]
This proves (ii). For every $x\in l^{p}$,
\begin{eqnarray*}
\left\|\sum_{i=1}^{k}D_{i}T_{i}D_{i}x\right\|^{p}&\leq&r^{p}\sum_{i=1}^{k}\|D_{i}T_{i}D_{i}x\|^{p}\text{ by (i)}\\&\leq&
r^{p}\max_{1\leq i\leq k}\|D_{i}T_{i}\|^{p}\sum_{i=1}^{k}\|D_{i}x\|^{p}\\&\leq&
r^{p+1}\max_{1\leq i\leq k}\|D_{i}T_{i}\|^{p}\|x\|^{p}\text{ by (ii)}.
\end{eqnarray*}
This proves (iii). Let $n\in\mathbb{N}$. Let $P_{n}$ be the projection from $l^{p}$ onto $l^{p}([1,n])$. By (iii),
\[\left\|\sum_{i=1}^{k}D_{i}T_{i}D_{i}\right\|_{e}\leq\left\|\sum_{i=1}^{k}D_{i}T_{i}(I-P_{n})D_{i}\right\|\leq r^{2}\max_{1\leq i\leq k}\|D_{i}T_{i}(I-P_{n})\|.\]
Taking $n\to\infty$, we obtain
\[\left\|\sum_{i=1}^{k}D_{i}T_{i}D_{i}\right\|_{e}\leq r^{2}\max_{1\leq i\leq k}\|D_{i}T_{i}\|_{e}.\]
This proves (iv).
\end{proof}
\begin{lemma}\label{Embedding}
Let $r\in\mathbb{N}$. Let $D_{1},\ldots,D_{k}$ be diagonal operators on $l^{p}$ such that every $j\in\mathbb{N}$ is contained in at most $r$ of $\mathrm{supp}\,D_{1},\ldots,\mathrm{supp}\,D_{k}$. Define operators $V:l^{p}\to\underbrace{l^{p}\oplus\ldots\oplus l^{p}}_{k}$ and $E:\underbrace{l^{p}\oplus\ldots\oplus l^{p}}_{k}\to l^{p}$ by
\[Vx=(D_{1}x,\ldots,D_{k}x)\quad\text{and}\quad E(y_{1},\ldots,y_{k})=\sum_{i=1}^{k}D_{i}y_{i},\]
for $x\in l^{p}$ and $y_{1},\ldots,y_{k}\in l^{p}$. Then $\displaystyle\|V\|\leq r\max_{1\leq i\leq k}\|D_{i}\|$ and $\displaystyle\|E\|\leq r\max_{1\leq i\leq k}\|D_{i}\|$.
\end{lemma}
\begin{proof}
By homogeneity, we may assume that $D_{1},\ldots,D_{k}$ have norm $1$. By Lemma \ref{Diagest}(ii), we have $\|V\|\leq r$. By Lemma \ref{Diagest}(i), we have $\|E\|\leq r$.
\end{proof}
\begin{lemma}\label{Multidiagest}
Let $r\in\mathbb{N}\cup\{0\}$. Let $\mathcal{X}_{1},\mathcal{X}_{2},\ldots$ be finite dimensional Banach spaces. For each $k\in\mathbb{N}$, let $J_{k}:\mathcal{X}_{k}\to(\oplus_{n\in\mathbb{N}}\mathcal{X}_{n})_{l^{p}}$ and $Q_{k}:(\oplus_{n\in\mathbb{N}}\mathcal{X}_{n})_{l^{p}}\to\mathcal{X}_{k}$ be the canonical embedding and projection, respectively. For $i,j\in\mathbb{N}$, let $T_{i,j}\in B(\mathcal{X}_{j},\mathcal{X}_{i})$. Assume that $\sup_{i,j\in\mathbb{N}}\|T_{i,j}\|<\infty$ and $T_{i,j}=0$ for all $|i-j|>r$. Then the sum $\sum_{i,j\in\mathbb{N}}J_{i}T_{i,j}Q_{j}$ in $B((\oplus_{n\in\mathbb{N}}\mathcal{X}_{n})_{l^{p}})$ converges in SOT unconditionally and
\[\limsup_{j\to\infty}\sup_{i\in\mathbb{N}}\|T_{i,j}\|\leq\left\|\sum_{i,j\in\mathbb{N}}J_{i}T_{i,j}Q_{j}\right\|_{e}\leq
(2r+1)\limsup_{j\to\infty}\sup_{i\in\mathbb{N}}\|T_{i,j}\|\]
\end{lemma}
\begin{proof}
For notational convenience, let $\mathcal{X}_{k}=\{0\}$ for $k<1$ and let $J_{i}=0$ and $T_{i,j}=0$ for $i<1$.  For $s\in\mathbb{Z}$, finite subset $\Omega\subset\mathbb{N}$ and $x\in(\oplus_{n\in\mathbb{N}}\mathcal{X}_{n})_{l^{p}}$,
\begin{eqnarray}\label{diagJPT}
\left\|\sum_{j\in\Omega}J_{j+s}T_{j+s,j}Q_{j}x\right\|&=&\left(\sum_{j\in\Omega}\|T_{j+s,j}Q_{j}x\|^{p}\right)^{\frac{1}{p}}\\&\leq&
\sup_{j\in\Omega}\|T_{j+s,j}\|\left(\sum_{j\in\Omega}\|Q_{j}x\|^{p}\right)^{\frac{1}{p}}.\nonumber
\end{eqnarray}
So $\sum_{j\in\mathbb{N}}J_{j+s}T_{j+s,j}Q_{j}$ converges in SOT unconditionally for every $s\in\mathbb{Z}$. Since $T_{i,j}=0$ for all $|i-j|>r$, it follows that the series $\sum_{i,j\in\mathbb{N}}J_{i}T_{i,j}Q_{j}=\sum_{s=-r}^{r}\sum_{j\in\mathbb{N}}J_{j+s}T_{j+s,j}Q_{j}$ converges in SOT unconditionally.

By (\ref{diagJPT}),
\[\left\|\sum_{j>n}J_{j+s}T_{j+s,j}Q_{j}\right\|\leq\sup_{j>n}\|T_{j+s,j}\|,\]
for all $s\in\mathbb{Z}$ and $n\in\mathbb{N}$. Since $T_{i,j}=0$ for all $|i-j|>r$,
\begin{eqnarray*}
\left\|\sum_{i\in\mathbb{N},\,j>n}J_{i}T_{i,j}Q_{j}\right\|&=&\left\|\sum_{s=-r}^{r}\sum_{j>n}J_{j+s}T_{j+s,j}Q_{j}\right\|\\&\leq&
(2r+1)\sup_{i\in\mathbb{N},\,j>n}\|T_{i,j}\|,
\end{eqnarray*}
for all $n\in\mathbb{N}$. Since $\mathcal{X}_{1},\mathcal{X}_{2},\ldots$ are finite dimensional, $\displaystyle\|T\|_{e}=\lim_{n\to\infty}\|T(I-(Q_{1}+\ldots+Q_{n}))\|$ for every operator $T$ on $(\oplus_{n=1}^{\infty}\mathcal{X}_{n})_{l^{p}}$. It follows that
\begin{eqnarray*}
\left\|\sum_{i\in\mathbb{N}}J_{i}T_{i,j}Q_{j}\right\|_{e}&=&
\lim_{n\to\infty}\left\|\left(\sum_{i,j\in\mathbb{N}}J_{i}T_{i,j}Q_{j}\right)(I-(Q_{1}+\ldots+Q_{n}))\right\|\\&=&
\lim_{n\to\infty}\left\|\sum_{i\in\mathbb{N},\,j>n}J_{i}T_{i,j}Q_{j}\right\|\\&\leq&
\lim_{n\to\infty}(2r+1)\sup_{i\in\mathbb{N},\,j>n}\|T_{i,j}\|=(2r+1)\limsup_{j\to\infty}\sup_{i\in\mathbb{N}}\|T_{i,j}\|.
\end{eqnarray*}
Also,
\begin{eqnarray*}
\left\|\sum_{i,j\in\mathbb{N}}J_{i}T_{i,j}Q_{j}\right\|_{e}&\geq&
\limsup_{n\to\infty}\left\|\left(\sum_{i,j\in\mathbb{N}}J_{i}T_{i,j}Q_{j}\right)Q_{n}\right\|\\&=&
\limsup_{n\to\infty}\left\|\sum_{i\in\mathbb{N}}J_{i}T_{i,n}Q_{n}\right\|\\&\geq&
\limsup_{n\to\infty}\sup_{m\in\mathbb{N}}\left\|Q_{m}\left(\sum_{i\in\mathbb{N}}J_{i}T_{i,n}Q_{n}\right)\right\|\\&=&
\limsup_{n\to\infty}\sup_{m\in\mathbb{N}}\|T_{m,n}\|=\limsup_{j\to\infty}\sup_{i\in\mathbb{N}}\|T_{i,j}\|.
\end{eqnarray*}
Therefore,
\[\limsup_{j\to\infty}\sup_{i\in\mathbb{N}}\|T_{i,j}\|\leq\left\|\sum_{i,j\in\mathbb{N}}J_{i}T_{i,j}Q_{j}\right\|_{e}\leq
(2r+1)\limsup_{j\to\infty}\sup_{i\in\mathbb{N}}\|T_{i,j}\|.\]
\end{proof}
\begin{lemma}\label{specialopcov}
Let $d\in\mathbb{N}$. Let $M$ be a compact subset of $\mathbb{R}^{d}$. Let $\epsilon>0$. Then there exists an open cover $(U_{i})_{1\leq i\leq k}$ of $M$ such that each $U_{i}$ has diameter at most $\epsilon$ and every $v\in M$ is contained in at most $2^{d}$ of the sets $U_{1},\ldots,U_{k}$.
\end{lemma}
\begin{proof}
Without loss of generality, we may assume that $M$ is a subset of $[0,1]^{d}$. Let $n\geq\frac{2d}{\epsilon}$ be a natural number. Note that the intervals $(\frac{j-1}{n},\frac{j+1}{n})$, for $j=0,\ldots,n$ form an open cover of $[0,1]$. Every $c\in[0,1]$ is contained in at most two of these intervals. For $j_{1},\ldots,j_{d}\in\{0,\ldots,n\}$, let $U_{j_{1},\ldots,j_{k}}=(\frac{j_{1}-1}{n},\frac{j_{1}+1}{n})\times\ldots\times(\frac{j_{d}-1}{n},\frac{j_{d}+1}{n})$. Then the open sets $U_{j_{1},\ldots,j_{d}}$, for $j_{1},\ldots,j_{d}\in\{0,\ldots,n\}$, form an open cover of $[0,1]^{d}$. Each $U_{j_{1},\ldots,j_{d}}$ has diameter at most $\epsilon$ with respect to the Euclidean metric. Fix $(v_{1},\ldots,v_{d})\in[0,1]^{d}$. For $1\leq i\leq d$, the number $v_{i}$ is contained in at most two of the intervals $(\frac{j-1}{n},\frac{j+1}{n})$, for $j=0,\ldots,n$. So there are $j_{i}^{(1)},j_{i}^{(2)}\in\{0,\ldots,n\}$ (which may or may not be the same) such that if $v_{i}\in (\frac{j-1}{n},\frac{j+1}{n})$ then $j$ must be $j_{i}^{(1)}$ or $j_{i}^{(2)}$. Thus, if $(v_{1},\ldots,v_{d})\in U_{j_{1},\ldots,j_{d}}$ then for each $1\leq i\leq d$, the number $j_{i}$ must be $j_{i}^{(1)}$ or $j_{i}^{(2)}$. So $(v_{1},\ldots,v_{d})$ is contained in at most $2^{d}$ of the sets $U_{j_{1},\ldots,j_{d}}$, for $j_{1},\ldots,j_{d}\in\{0,\ldots,n\}$.
\end{proof}
\begin{lemma}\label{partitionunity}
Let $d\in\mathbb{N}$. Let $M$ be a nonempty compact subset of $\mathbb{R}^{d}$. Let $\epsilon>0$. Then there exist a partition of unity $(f_{i})_{1\leq i\leq k}$ on $M$ and continuous functions $g_{i}:M\to[0,1]$, for $1\leq i\leq k$, such that
\begin{enumerate}[(1)]
\item the diameter of the support of $g_{i}$ is at most $\epsilon$ for every $i=1,\ldots,k$;
\item every $v\in M$ is contained in at most $2^{d}$ of the sets $\mathrm{supp}\,g_{1},\ldots,\mathrm{supp}\,g_{k}$; and
\item $g_{i}=1$ on the support of $f_{i}$ for every $i=1,\ldots,k$.
\end{enumerate}
\end{lemma}
\begin{proof}
By Lemma \ref{specialopcov}, there exists an open cover $(U_{i})_{1\leq i\leq k}$ of $M$ such that each $U_{i}$ has diameter at most $\epsilon$ and every $v\in M$ is contained in at most $2^{d}$ of the sets $U_{1},\ldots,U_{k}$. Take $(f_{i})_{1\leq i\leq k}$ to be a partition of unity on $M$ subordinate to $(U_{i})_{1\leq i\leq k}$. For each $1\leq i\leq k$, let $g_{i}:M\to[0,1]$ be a continuous function such that $g_{i}=1$ on the support of $f_{i}$ and $\mathrm{supp}\,g_{i}\subset U_{i}$. Since each $v\in M$ is contained in at most $2^{d}$ of the sets $U_{1},\ldots,U_{k}$, it is contained in at most $2^{d}$ of $\mathrm{supp}\,g_{1},\ldots,\mathrm{supp}\,g_{k}$.
\end{proof}
\section{Approximate identity}
A sequence $(A_{n})_{n\in\mathbb{N}}$ of operators on $l^{p}$ is a {\it refined diagonal approximate identity on} $l^{p}$ if
\begin{enumerate}[(1)]
\item each $A_{n}$ is a diagonal operator on $l^{p}$ with diagonal entries in $[0,1]$ and $\mathrm{supp}(A_{n})$ is finite;
\item $A_{n}\to I$ in SOT, as $n\to\infty$; and
\item $(I-A_{n+1})A_{n}=0$ for all $n\in\mathbb{N}$.
\end{enumerate}
It is easy to see that if $(A_{n})_{n\in\mathbb{N}}$ is a refined diagonal approximate identity, then $\mathrm{supp}\,A_{1}\subset\mathrm{supp}\,A_{2}\subset\ldots$ and $(I-A_{n})A_{m}=0$ for all $m<n$ in $\mathbb{N}$.

The following two lemmas are well known results but we include their proofs for convenience.
\begin{lemma}\label{weak}
Let $\mathcal{X}$ be a separable reflexive Banach space. Let $(K_{n})_{n\in\mathbb{N}}$ be a bounded sequence of compact operators on $\mathcal{X}$ converging to 0 in WOT. Then $K_{n}\to 0$ weakly in the sense of Banach space.
\end{lemma}
\begin{proof}
Let $B_{\mathcal{X}}(0,1)$ and $B_{\mathcal{X}^{*}}(0,1)$ be the unit balls of $\mathcal{X}$ and $\mathcal{X}^{*}$, respectively, equipped with the weak topologies. Consider the compact Hausdorff space $M=B_{\mathcal{X}}(0,1)\times B_{\mathcal{X}^{*}}(0,1)$ with the product topology. For each $K\in K(\mathcal{X})$, define a continuous function $f_{K}:M\to\mathbb{C}$ by $f_{K}(x,x^{*})=x^{*}(Kx)$, for $(x,x^{*})\in M$. Since $(K_{n})_{n\in\mathbb{N}}$ is uniformly bounded and converges to 0 in WOT, $(f_{K_{n}})_{n\in\mathbb{N}}$ is uniformly bounded and converges to 0 pointwise. So by dominated convergence theorem, $\int f_{K_{n}}\,d\mu\to 0$ for every finite measure $\mu$ on $M$. So $f_{K_{n}}$ converges to 0 weakly in the space $C(M)$ of continuous functions from $M$ into $\mathbb{C}$ equipped with $\|\,\|_{\infty}$.

Note that the map $K\mapsto f_{K}$ defines an isometry from $K(\mathcal{X})$ into $C(M)$. Thus, by Hahn-Banach Theorem, it follows that $K_{n}$ converges to 0 weakly.
\end{proof}
\begin{lemma}\label{qcau}
Let $\Lambda$ be a countable set. Let $\psi:\Lambda\to B(l^{p})$. Then for all finite subsets $\Omega_{1}\subset\Omega_{2}\subset\ldots$ of $\Lambda$ and $\epsilon_{1},\epsilon_{2},\ldots>0$, there exists a refined diagonal approximate identity $(A_{n})_{n\in\mathbb{N}}$ on $l^{p}$ such that $\|A_{n}\psi(\alpha)-\psi(\alpha)A_{n}\|\leq\epsilon_{n}$ for all $\alpha\in\Omega_{n}$ and $n\in\mathbb{N}$.
\end{lemma}
\begin{proof}
For each $n\in\mathbb{N}$, let $P_{n}$ be the projection from $l^{p}$ onto $l^{p}([1,n])$. Then for each $\alpha\in\Lambda$, we have $P_{n}\psi(\alpha)-\psi(\alpha)P_{n}\to 0$ in SOT as $n\to\infty$. So by Lemma \ref{weak}, for each $\alpha\in\Lambda$, we have $P_{n}\psi(\alpha)-\psi(\alpha)P_{n}\to 0$ weakly in the sense of Banach space as $n\to\infty$.

Choose finite rank diagonal operators $A_{1},A_{2},\ldots$ as follows: Take $A_{1}=0$. Suppose that $A_{1},\ldots,A_{i-1}$ have been chosen.  Let $j\in\mathbb{N}$ be large enough so that $j\geq i$ and $(I-P_{j})A_{i-1}=0$. From the previous paragraph, $P_{n}\psi(\alpha)-\psi(\alpha)P_{n}\to 0$ weakly in the sense of Banach space, as $n\to\infty$, for every $\alpha\in\Omega_{j}$. Thus, there exists $A_{i}$ that is a convex combination of $P_{j},P_{j+1},\ldots$ such that $\|A_{i}\psi(\alpha)-\psi(\alpha)A_{i}\|\leq\epsilon_{i}$ for every $\alpha\in\Omega_{j}$.

Since $j\geq i$, it follows that $\|A_{i}\psi(\alpha)-\psi(\alpha)A_{i}\|\leq\epsilon_{i}$ for all $\alpha\in\Omega_{i}$ and $i\in\mathbb{N}$. Since $(I-P_{j})A_{i-1}=0$ and $A_{i}$ is a convex combination of $P_{j},P_{j+1},\ldots$, it is easy to see that $(I-A_{i})A_{i-1}=0$ for all $i\geq 2$. Also, since $j\geq i$ and $A_{i}$ is a convex combination of $P_{j},P_{j+1},\ldots$, we have $A_{n}\to I$ in SOT as $n\to\infty$. Therefore, $(A_{n})_{n\in\mathbb{N}}$ is a refined diagonal approximate identity on $l^{p}$.
\end{proof}
In the sequel, $A_{-1}=A_{0}=0$.
\begin{lemma}\label{qcausum}
Suppose that $(A_{n})_{n\in\mathbb{N}}$ is a refined diagonal approximate identity on $l^{p}$. Then
\begin{equation}\label{qcausum1}
\left(\sum_{n=1}^{\infty}\|(A_{n}-A_{n-1})x\|^{p}\right)^{\frac{1}{p}}\leq 2\|x\|,
\end{equation}
for all $x\in l^{p}$. Moreover, if $(x_{n})_{n\in\mathbb{N}}$ is a sequence in $l^{p}$ such that $\displaystyle\sum_{n=1}^{\infty}\|x_{n}\|^{p}<\infty$, then
\begin{equation}\label{qcausum2}
\left\|\sum_{n=1}^{\infty}(A_{n+1}-A_{n-2})x_{n}\right\|^{p}\leq 4^{p}\sum_{n=1}^{\infty}\|x_{n}\|^{p}.
\end{equation}
\end{lemma}
\begin{proof}
For $n\in\mathbb{N}$, since $(I-A_{n-1})A_{n-2}=(I-A_{n})A_{n-2}=0$, we have $A_{n}-A_{n-1}=0$ on $l^{p}(\mathrm{supp}\,A_{n-2})$ and so $\mathrm{supp}(A_{n}-A_{n-1})\subset\mathrm{supp}\,A_{n}\backslash\mathrm{supp}\,A_{n-2}$. Since $\mathrm{supp}\,A_{1}\subset\mathrm{supp}\,A_{2}\subset\ldots$, it follows that every $j\in\mathbb{N}$ is contained in at most $2$ of $\mathrm{supp}(A_{n}-A_{n-1})$, for $n\in\mathbb{N}$. By Lemma \ref{Diagest}(ii) with $r=2$, we obtain (\ref{qcausum1}).

Since $(I-A_{n+1})A_{n-3}=(I-A_{n-2})A_{n-3}=0$, we have $A_{n+1}-A_{n-2}=0$ on $l^{p}(\mathrm{supp}\,A_{n-3})$ and so $\mathrm{supp}(A_{n+1}-A_{n-2})\subset\mathrm{supp}\,A_{n+1}\backslash\mathrm{supp}\,A_{n-3}$. For each $n\in\mathbb{N}$, let $D_{n}$ be the projection from $l^{p}$ onto $l^{p}(\mathrm{supp}(A_{n+1}-A_{n-2}))$. By Lemma \ref{Diagest}(i) with $r=4$, if $x_{n}$ is in the range of $A_{n+1}-A_{n-2}$ for $n\in\mathbb{N}$, then
\[\left\|\sum_{n=1}^{\infty}x_{n}\right\|^{p}=\left\|\sum_{n=1}^{\infty}D_{n}x_{n}\right\|^{p}\leq 4^{p}\sum_{n=1}^{\infty}\|D_{n}x_{n}\|^{p}\leq
4^{p}\sum_{n=1}^{\infty}\|x_{n}\|^{p}.\]
This proves (\ref{qcausum2}).
\end{proof}
Suppose that $(A_{n})_{n\in\mathbb{N}}$ is a refined diagonal approximate identity on $l^{p}$. For every $n\in\mathbb{N}$,
\[A_{n+1}A_{n}=A_{n}\text{ and }A_{n+1}A_{n-1}=A_{n-1}\]
so
\[A_{n+1}(A_{n}-A_{n-1})=A_{n}-A_{n-1}.\]
On the other hand, for every $n\in\mathbb{N}$,
\[A_{n-2}A_{n}=A_{n-2}=A_{n-2}A_{n-1}\]
so
\[A_{n-2}(A_{n}-A_{n-1})=0.\]
Therefore,
\[(A_{n+1}-A_{n-2})(A_{n}-A_{n-1})=A_{n}-A_{n-1}.\]
So combining Lemma \ref{qcau} and Lemma \ref{qcausum}, we obtain the following result.
\begin{lemma}\label{qcaum}
Let $\Lambda$ be a countable set. Let $\lambda\geq 1$. Suppose that $\mathcal{X}$ is a Banach space that is either finite dimensional or $\lambda$-isomorphic to $l^{p}$. Let $\psi:\Lambda\to B(l^{p})$. Let $\Omega_{1}\subset\Omega_{2}\subset\ldots$ be finite subsets of $\Lambda$. Let $\epsilon_{1},\epsilon_{2},\ldots>0$. Then there exist finite rank operators $A_{1},A_{2},\ldots$ on $\mathcal{X}$ such that
\begin{enumerate}[(i)]
\item $\displaystyle\sup_{n\in\mathbb{N}}\|A_{n}\|\leq\lambda$;
\item $A_{n}\to I$ in SOT as $n\to\infty$;
\item $\displaystyle\|A_{n}\psi(\alpha)-\psi(\alpha)A_{n}\|\leq\epsilon_{n}$ for all $\alpha\in\Omega_{n}$ and $n\in\mathbb{N}$;
\item $(A_{n+1}-A_{n-2})(A_{n}-A_{n-1})=A_{n}-A_{n-1}$ for all $n\in\mathbb{N}$;
\item $\displaystyle\left(\sum_{n=1}^{\infty}\|(A_{n}-A_{n-1})x\|^{p}\right)^{\frac{1}{p}}\leq2\lambda\|x\|$ for $x\in\mathcal{X}$; and
\item  if $(x_{n})_{n\in\mathbb{N}}$ is a sequence in $\mathcal{X}$ such that $\displaystyle\sum_{n=1}^{\infty}\|x_{n}\|^{p}<\infty$, then
\[\left(\left\|\sum_{n=1}^{\infty}(A_{n+1}-A_{n-2})x_{n}\right\|^{p}\right)^{\frac{1}{p}}\leq
4\lambda\left(\sum_{n=1}^{\infty}\|x_{n}\|^{p}\right)^{\frac{1}{p}}.\]
\end{enumerate}
\end{lemma}
\section{Functional calculus}
Let $\mathcal{B}$ be a Banach algebra. A function $h:[0,1]\to\mathcal{B}$ is {\it $\mathcal{B}$-simple} if there are $b_{1},\ldots,b_{n}\in\mathcal{B}$ and $f_{1},\ldots,f_{n}\in C[0,1]$ such that $h(t)=\sum_{i=1}^{n}b_{i}f_{i}(t)$ for all $t\in M$.

Suppose that $A$ is a diagonal operator on $l^{p}$ with entries $a_{1},a_{2},\ldots$ in $[0,1]$. For $f\in C[0,1]$, let $f(A)$ be the diagonal operator on $l^{p}$ with entries $f(a_{1}),f(a_{2}),\ldots$ and let $f(\pi(A))=\pi(f(A))\in B(l^{p})/K(l^{p})$.

Let $\mathcal{B}$ be the commutant of $\pi(A)$ in $B(l^{p})/K(l^{p})$. We define $h(\pi(A))$, for every $\mathcal{B}$-simple function $h:[0,1]\to\mathcal{B}$, as follows: if we write $h(t)=\sum_{i=1}^{n}b_{i}f_{i}(t)$, where $b_{1},\ldots,b_{n}\in\mathcal{B}$, we set
\[h(\pi(A))=\sum_{i=1}^{n}b_{i}f_{i}(\pi(A))\in\mathcal{B}.\]
For example, if $h(t)=b_{0}+b_{1}t+b_{2}t^{2}$ then $h(\pi(A))=b_{0}+b_{1}\pi(A)+b_{2}\pi(A^{2})$.

Next we can define $h(\pi(A))$ for all continuous function $h:[0,1]\to\mathcal{B}$. Indeed, we show that every continuous function from $[0,1]$ to $\mathcal{B}$ can be approximated by $\mathcal{B}$-simple functions on $[0,1]$ for all unital Banach algebra $\mathcal{B}$ (Lemma \ref{simpleapprox}). We also prove that
\[\|h(\pi(A))\|\leq4\sup_{0\leq t\leq 1}\|h(t)\|,\]
for all $\mathcal{B}$-simple function $h$ on $[0,1]$ (Corollary \ref{Simpleest}). We define $h(\pi(A))$ as follows: if $(h_{n})_{n\in\mathbb{N}}$ is any sequence of $\mathcal{B}$-simple functions on $[0,1]$ such that $\sup_{0\leq t\leq 1}\|h_{n}(t)-h(t)\|\to 0$, as $n\to\infty$, then we set
\[h(\pi(A))=\lim_{n\to\infty}h_{n}(\pi(A)),\]
where the limit exists and does not depend of the choice of $(h_{n})_{n\in\mathbb{N}}$.

Let $\mathcal{B}$ be a Banach algebra. Let $a\in\mathcal{B}$. A function $h:[0,1]\to\mathcal{B}$ is {\it $a$-continuous} if there exists a continuous function $\widetilde{h}:[0,1]\to\mathcal{B}$ such that $(h(t)-\widetilde{h}(t))a=0$, for all $t\in[0,1]$, and $h(0)=\widetilde{h}(0)$.

Finally, if $A$ is a diagonal operator on $l^{p}$ with entries in $[0,1]$ and $\mathcal{B}$ is the commutant of $\pi(A)$ in $B(l^{p})/K(l^{p})$, we define $h(\pi(A))$, for every $\pi(A)$-continuous function $h:[0,1]\to\mathcal{B}$, as follows: \[h(\pi(A))=\widetilde{h}(\pi(A)),\]
where $\widetilde{h}:[0,1]\to\mathcal{B}$ is any continuous function such that $(h(t)-\widetilde{h}(t))\pi(A)=0$, for all $t\in[0,1]$, and $h(0)=\widetilde{h}(0)$. By Lemma \ref{samecalculus} below, the element $\widetilde{h}(\pi(A))$ of $\mathcal{B}$ does not depend on the choice of $\widetilde{h}$.
\begin{lemma}\label{simpleapprox}
Let $\mathcal{B}$ be a unital Banach algebra. Let $h:[0,1]\to\mathcal{B}$ be a continuous function. Then there is a sequence $(h_{n})_{n\in\mathbb{N}}$ of $\mathcal{B}$-simple functions on $[0,1]$ such that $\displaystyle\sup_{0\leq t\leq 1}\|h_{n}(t)-h(t)\|\to 0$ as $n\to\infty$.
\end{lemma}
\begin{proof}
Let $\epsilon>0$. We need to find a $\mathcal{B}$-simple function $h_{1}$ on $[0,1]$ such that $\sup_{0\leq t\leq 1}\|h_{1}(t)-h(t)\|\leq\epsilon$. There exists $\gamma>0$ such that
\[\|h(t)-h(s)\|\leq\epsilon,\]
for all $t,s\in[0,1]$ such that $|t-s|\leq\gamma$. Let $(U_{i})_{1\leq i\leq k}$ be an open cover of $[0,1]$ such that each $U_{i}$ has diameter at most $\gamma$. Let $(f_{i})_{1\leq i\leq k}$ be a partition of unity on $[0,1]$ subordinate to $(U_{i})_{1\leq i\leq k}$. For each $1\leq i\leq k$, pick $t_{i}\in U_{i}$. Take
\[h_{1}(t)=\sum_{i=1}^{k}h(t_{i})f_{i}(t),\]
for $t\in[0,1]$. For $t\in[0,1]$, we have
\begin{eqnarray*}
\|h_{1}(t)-h(t)\|&=&\left\|\sum_{i=1}^{k}h(t_{i})f_{i}(t)-\sum_{i=1}^{k}h(t)f_{i}(t)\right\|\\&\leq&
\sum_{i=1}^{k}\|h(t_{i})-h(t)\|f_{i}(t)\leq\sup_{i:f_{i}(t)\neq 0}\|h(t_{i})-h(t)\|,
\end{eqnarray*}
where the last inequality follows from the fact that $\sum_{i=1}^{k}f_{i}(t)=1$. However, if $f_{i}(t)\neq 0$ then $t\in U_{i}$. Since $U_{i}$ has diameter at most $\gamma$, this implies that $|t-t_{i}|\leq\gamma$ and so $\|h(t_{i})-h(t)\|\leq\epsilon$. Therefore, $\|h_{1}(t)-h(t)\|\leq\epsilon$ for all $t\in[0,1]$.
\end{proof}
\begin{lemma}\label{Simplecont}
Let $A$ be a diagonal operator on $l^{p}$ with entries in $[0,1]$. Let $\mathcal{B}$ be the commutant of $\pi(A)$ in $B(l^{p})/K(l^{p})$. Let $h$ be a $\mathcal{B}$-simple function. Then for every $\epsilon>0$, there exists $\gamma>0$ such that $\|(h(\pi(A))-h(t))\pi(f(A))\|\leq\epsilon$ for all $t\in[0,1]$ and $f\in C[0,1]$ with $\mathrm{supp}\,f\subset[t-\gamma,t+\gamma]$.
\end{lemma}
\begin{proof}
We write $h(t)=\sum_{i=1}^{n}b_{i}f_{i}(t)$, for $t\in[0,1]$, where $b_{1},\ldots,b_{n}\in\mathcal{B}$ and $f_{1},\ldots,f_{n}\in C[0,1]$. Take $\gamma>0$ to be small enough so that $|f_{i}(s)-f_{i}(t)|\leq\frac{\epsilon}{\|b_{1}\|+\ldots+\|b_{n}\|}$ for all $1\leq i\leq n$ and $s,t\in[0,1]$ such that $|s-t|\leq\gamma$. Then $\|(f_{i}(A)-f_{i}(t))f(A)\|\leq\frac{\epsilon}{\|b_{1}\|+\ldots+\|b_{n}\|}$ for all $t\in[0,1]$ and $f\in C[0,1]$ with $\mathrm{supp}\,f\subset[t-\gamma,t+\gamma]$. Hence,
\begin{eqnarray*}
\|(h(\pi(A))-h(t))\pi(f(A))\|&\leq&\sum_{i=1}^{n}\|b_{i}\|\|(f_{i}(A)-f_{i}(t))f(A)\|\\&\leq&
\sum_{i=1}^{n}\|b_{i}\|\frac{\epsilon}{\|b_{1}\|+\ldots+\|b_{n}\|}=\epsilon.
\end{eqnarray*}
\end{proof}
\begin{lemma}\label{Simpleest}
Let $A$ be a diagonal operator on $l^{p}$ with entries in $[0,1]$. Let $\mathcal{B}$ be the commutant of $\pi(A)$ in $B(l^{p})/K(l^{p})$. Let $h$ be a $\mathcal{B}$-simple function. Then
\[\|h(\pi(A))\|\leq4\sup_{0\leq t\leq 1}\|h(t)\|.\]
\end{lemma}
\begin{proof}
Let $\epsilon>0$. By Lemma \ref{Simplecont}, there exists $\gamma>0$ such that
\begin{equation}\label{Simpleesteq1}
\|(h(\pi(A))-h(t))\pi(f(A))\|\leq\epsilon,
\end{equation}
for all $t\in[0,1]$ and $f\in C[0,1]$ with $\mathrm{supp}\,f\subset[t-\gamma,t+\gamma]$. Let $(U_{i})_{1\leq i\leq k}$ be an open cover of $[0,1]$ such that each $U_{i}$ has diameter at most $\gamma$ and every point on $[0,1]$ is contained in at most $2$ of $U_{1},\ldots,U_{k}$. Let $(f_{i})_{1\leq i\leq k}$ be a partition of unity on $[0,1]$ subordinate to $(U_{i})_{1\leq i\leq k}$. We have $\sum_{i=1}^{k}f_{i}(A)=I$ so
\[h(\pi(A))=h(\pi(A))\sum_{i=1}^{k}\pi(f_{i}(A))=\sum_{i=1}^{k}h(\pi(A))\pi(f_{i}(A)^{\frac{1}{2}})\pi(f_{i}(A)^{\frac{1}{2}}).\]
Since $h(\pi(A))\in\mathcal{B}$, the elements $h(\pi(A))$ and $\pi(A)$ commute. Thus, $h(\pi(A))$ commutes with any polynomial of $\pi(A)$ and so $h(\pi(A))$ commutes with $\pi(f_{i}(A)^{\frac{1}{2}})$. Therefore,
\begin{equation}\label{Simpleesteq2}
h(\pi(A))=\sum_{i=1}^{k}\pi(f_{i}(A)^{\frac{1}{2}})h(\pi(A))\pi(f_{i}(A)^{\frac{1}{2}}).
\end{equation}
Since every point on $[0,1]$ is contained in at most $2$ of $U_{1},\ldots,U_{k}$, every $j\in\mathbb{N}$ is contained in at most $2$ of $\mathrm{supp}\,f_{1}(A)^{\frac{1}{2}},\mathrm{supp}\,f_{2}(A)^{\frac{1}{2}},\ldots$. By Lemma \ref{Diagest}(iv),
\begin{equation}\label{Simpleesteq3}
\left\|\sum_{i=1}^{k}\pi(f_{i}(A)^{\frac{1}{2}})h(\pi(A))\pi(f_{i}(A)^{\frac{1}{2}})\right\|\leq4\max_{1\leq i\leq k}\|\pi(f_{i}(A)^{\frac{1}{2}})h(\pi(A))\|.
\end{equation}
Since each $U_{i}$ has diameter at most $\gamma$, we have $U_{i}\subset[t_{i}-\gamma,t_{i}+\gamma]$ for some $t_{i}\in[0,1]$. So $\mathrm{supp}\,f_{i}\subset[t_{i}-\gamma,t_{i}+\gamma]$ for all $1\leq i\leq k$. Thus, by (\ref{Simpleesteq1}),
\begin{eqnarray*}
\|\pi(f_{i}(A)^{\frac{1}{2}})h(\pi(A))\|&=&\|h(\pi(A))\pi(f_{i}(A)^{\frac{1}{2}})\|\\&\leq&
\|h(t_{i})\pi(f_{i}(A)^{\frac{1}{2}})\|+\|h(\pi(A))-h(t_{i}))\pi(f_{i}(A)^{\frac{1}{2}})\|\\&\leq&\|h(t_{i})\|+\epsilon.
\end{eqnarray*}
So by (\ref{Simpleesteq2}) and (\ref{Simpleesteq3}), we have $\displaystyle\|h(\pi(A))\|\leq 4\sup_{0\leq t\leq 1}\|h(t)\|$.
\end{proof}
As explained at the beginning of this section, having proved Lemmas \ref{simpleapprox}-\ref{Simpleest}, we can define $h(\pi(A))\in\mathcal{B}$ for all diagonal operator $A$ on $l^{p}$ with entries in $[0,1]$ and continuous function $h:[0,1]\to\mathcal{B}$ where $\mathcal{B}$ is the commutant of $\pi(A)$ in $B(l^{p})/K(l^{p})$. Moreover,
\begin{equation}\label{Simpleest2}
\|h(\pi(A))\|\leq4\sup_{0\leq t\leq 1}\|h(t)\|.
\end{equation}
\begin{lemma}\label{samecalculus}
Let $A$ be a diagonal operator on $l^{p}$ with entries in $[0,1]$. Let $\mathcal{B}$ be the commutant of $\pi(A)$ in $B(l^{p})/K(l^{p})$. If $h_{1}:[0,1]\to\mathcal{B}$ and $h_{2}:[0,1]\to\mathcal{B}$ are continuous functions such that $(h_{1}(t)-h_{2}(t))\pi(A)=0$, for all $t\in[0,1]$, and $h_{1}(0)=h_{2}(0)$, then $h_{1}(\pi(A))=h_{2}(\pi(A))$.
\end{lemma}
\begin{proof}
Let $\epsilon>0$. Since $h_{1}(0)=h_{2}(0)$, there exists $\gamma>0$ such that $|h_{1}(t)-h_{2}(t)|\leq\epsilon$ for all $0\leq t\leq\gamma$. Let $f:[0,1]\to[0,1]$ be a continuous function such that $f(0)=1$ and $f(t)=0$ for all $\gamma\leq t\leq 1$. Then $\|(h_{1}-h_{2})f\|\leq\epsilon$ and so $\|(h_{1}(\pi(A))-h_{2}(\pi(A)))f(\pi(A))\|\leq\epsilon$.

Since $(h_{1}(t)-h_{2}(t))\pi(A)=0$ and $1-f(0)=0$, we have $(h_{1}(t)-h_{2}(t))(1-f(\pi(A)))=0$ by polynomial approximation. Hence, by (\ref{Simpleest2}), we have $(h_{1}(\pi(A))-h_{2}(\pi(A)))(1-f(\pi(A)))=0$. Therefore, $\|h_{1}(\pi(A))-h_{2}(\pi(A))\|=\|(h_{1}(\pi(A))-h_{2}(\pi(A)))f(\pi(A))\|\leq\epsilon$. Choose $\epsilon>0$ to be arbitrarily small. The result follows.
\end{proof}
As explained at the beginning of this section, having proved Lemma \ref{samecalculus}, we can define $h(\pi(A))\in\mathcal{B}$ for all diagonal operator $A$ on $l^{p}$ with entries in $[0,1]$ and $\pi(A)$-continuous function $h:[0,1]\to\mathcal{B}$ where $\mathcal{B}$ is the commutant of $\pi(A)$ in $B(l^{p})/K(l^{p})$. Recall that $h$ is $\pi(A)$-continuous if there is a continuous function $\widetilde{h}:[0,1]\to\mathcal{B}$ such that $(h(t)-\widetilde{h}(t))\pi(A)=0$, for all $t\in[0,1]$, and $h(0)=\widetilde{h}(0)$. The following results say that this ``functional calculus" preserves direct sum with $1$.
\begin{lemma}\label{calculusdirectsum}
Let $A$ be a diagonal operator on $l^{p}$ with entries in $[0,1]$. Let $\mathcal{B}$ be the commutant of $\pi(A)$ in $B(l^{p})/K(l^{p})$. Let $h_{1}:[0,1]\to B(l^{p})/K(l^{p})$ and $h_{2}:[0,1]\to\mathcal{B}$ be continuous functions. Let $\mathcal{B}_{2}$ be the commutant of $\pi(I)\oplus\pi(A)$ in $B(l^{p}\oplus l^{p})/K(l^{p}\oplus l^{p})$. Define $h_{3}:[0,1]\to\mathcal{B}_{2}$ by $h_{3}(t)=h_{1}(t)\oplus h_{2}(t)$ for $t\in[0,1]$. Then $h_{3}(\pi(I)\oplus\pi(A))=h_{1}(1)\oplus h_{2}(\pi(A))$.
\end{lemma}
\begin{proof}
Let $\epsilon>0$. There exist $b_{1},\ldots,b_{n}\in B(l^{p})/K(l^{p})$, $c_{1},\ldots,c_{k}\in\mathcal{B}$, $f_{1},\ldots,f_{n},g_{1},\ldots,g_{k}\in C[0,1]$ such that $\|h_{1}(t)-\sum_{i=1}^{n}b_{i}f_{i}(t)\|\leq\epsilon$ and $\|h_{2}(t)-\sum_{j=1}^{k}c_{j}g_{j}(t)\|\leq\epsilon$ for all $t\in[0,1]$. Define $h_{4}:[0,1]\to\mathcal{B}_{2}$ by
\[h_{4}(t)=\left(\sum_{i=1}^{n}b_{i}f_{i}(t)\right)\oplus\left(\sum_{j=1}^{k}c_{j}g_{j}(t)\right)=
\sum_{i=1}^{n}(b_{i}\oplus\pi(0))f_{i}(t)+\sum_{j=1}^{k}(\pi(0)\oplus c_{j})g_{j}(t),\]
for $t\in [0,1]$. Then $h_{4}$ is a $\mathcal{B}_{2}$-simple function and $\|h_{3}(t)-h_{4}(t)\|\leq\epsilon$ for all $t\in[0,1]$. So
\begin{eqnarray*}
h_{4}(\pi(I)\oplus\pi(A))&=&\sum_{i=1}^{n}(b_{i}\oplus\pi(0))f_{i}(\pi(I)\oplus\pi(A))+\sum_{j=1}^{k}(\pi(0)\oplus c_{j})g_{j}(\pi(I)\oplus\pi(A))\\&=&
\left(\sum_{i=1}^{n}b_{i}f_{i}(1)\right)\oplus\left(\sum_{j=1}^{k}c_{j}\pi(g_{j}(A))\right).
\end{eqnarray*}
Thus, $\|h_{4}(\pi(I)\oplus\pi(A))-h_{1}(1)\oplus h_{2}(\pi(A))\|\leq\epsilon$. Therefore, $h_{3}(\pi(I)\oplus\pi(A))=h_{1}(1)\oplus h_{2}(\pi(A))$.
\end{proof}
The following result says that Lemma \ref{calculusdirectsum} still holds if $h_{2}$ is only assumed to be $\pi(A)$-continuous.
\begin{lemma}\label{calculusdirectsum2}
Let $A$ be a diagonal operator on $l^{p}$ with entries in $[0,1]$. Let $\mathcal{B}$ be the commutant of $\pi(A)$ in $B(l^{p})/K(l^{p})$. Suppose that $h_{1}:[0,1]\to B(l^{p})/K(l^{p})$ is continuous and $h_{2}:[0,1]\to\mathcal{B}$ is $\pi(A)$-continuous. Let $\mathcal{B}_{2}$ be the commutant of $\pi(I)\oplus\pi(A)$ in $B(l^{p}\oplus l^{p})/K(l^{p}\oplus l^{p})$. Define $h_{3}:[0,1]\to\mathcal{B}_{2}$ by $h_{3}(t)=h_{1}(t)\oplus h_{2}(t)$ for $t\in[0,1]$. Then $h_{3}$ is $(\pi(I)\oplus\pi(A))$-continuous and $h_{3}(\pi(I)\oplus\pi(A))=h_{1}(1)\oplus h_{2}(\pi(A))$.
\end{lemma}
\begin{proof}
There is a continuous function $\widetilde{h}_{2}:[0,1]\to\mathcal{B}$ such that $(h_{2}(t)-\widetilde{h}_{2}(t))\pi(A)=0$, for all $t\in[0,1]$, and $h_{2}(0)=\widetilde{h}_{2}(0)$. Then $\pi(I)\oplus\pi(A)$ and $h_{1}(t)\oplus\widetilde{h}_{2}(t)$ commute for all $t\in[0,1]$. Define $\widetilde{h}_{3}:[0,1]\to\mathcal{B}_{2}$ by $\widetilde{h}_{3}(t)=h_{1}(t)\oplus\widetilde{h}_{2}(t)$ for $t\in[0,1]$. Then $\widetilde{h}_{3}$ is continuous, $(h_{3}(t)-\widetilde{h}_{3}(t))(\pi(I)\oplus\pi(A))=0$, for all $t\in[0,1]$, and $h_{3}(0)=\widetilde{h}_{3}(0)$. Therefore, $h_{3}$ is $(\pi(I)\oplus\pi(A))$-continuous and $h_{3}(\pi(I)\oplus\pi(A))=\widetilde{h}_{3}(\pi(I)\oplus\pi(A))$. But by Lemma \ref{calculusdirectsum}, we have $\widetilde{h}_{3}(\pi(I)\oplus\pi(A))=h_{1}(1)\oplus\widetilde{h}_{2}(\pi(A))$. Since $\widetilde{h}_{2}(\pi(A))=h_{2}(\pi(A))$, the result follows.
\end{proof}
\section{Unconditional direct sum}
Let $\mathcal{Y}$ be a Banach space. Let $k\in\mathbb{N}$. Define the Banach space
\[\mathcal{Y}_{u}^{\oplus k}=\{(y_{1},\ldots,y_{k}):y_{1},\ldots,y_{k}\in\mathcal{Y}\}\]
with norm
\[\|(y_{1},\ldots,y_{k})\|=\mathbb{E}\left\|\sum_{i=1}^{k}\delta_{i}y_{i}\right\|,\]
where $\delta=(\delta_{1},\ldots,\delta_{k})$ is a random vector uniformly distributed on $\{-1,1\}^{k}$ and $\mathbb{E}$ denotes expectation with respect to $\delta$. For example,
\[\|(y_{1},y_{2})\|_{\mathcal{Y}_{u}^{\oplus 2}}=\frac{1}{2}(\|y_{1}+y_{2}\|+\|y_{1}-y_{2}\|).\]

Let $T_{1},\ldots,T_{k}\in B(\mathcal{Y})$. Define an operator $(T_{1}\oplus\ldots\oplus T_{k})_{u}\in B(\mathcal{Y}_{u}^{\oplus k})$ by
\[(T_{1}\oplus\ldots\oplus T_{k})_{u}(y_{1},\ldots,y_{k})=(T_{1}y_{1},\ldots,T_{k}y_{k}),\]
for $(y_{1},\ldots,y_{k})\in\mathcal{Y}_{u}^{\oplus k}$. It is easy to see that $\displaystyle\|(T_{1}\oplus\ldots\oplus T_{k})_{u}\|\geq\max_{1\leq i\leq k}\|T_{i}\|$. Equality does not necessarily hold. However, when $T_{1},\ldots,T_{k}$ are scalars, $\|(T_{1}\oplus\ldots\oplus T_{k})_{u}\|$ is, up to a constant, bounded by $\max_{1\leq i\leq k}\|T_{i}\|$.
\begin{lemma}\label{unconditional1}
Let $\mathcal{Y}$ be a Banach space. Let $c_{1},\ldots,c_{k}\in\mathbb{C}$. Then the operator $(c_{1}I\oplus\ldots\oplus c_{k}I)_{u}$ on $\mathcal{Y}_{u}^{\oplus k}$ has norm at most $2\max_{1\leq i\leq k}|c_{i}|$.
\end{lemma}
\begin{proof}
If $c_{1},\ldots,c_{k}\in\{1,-1\}$ then for all $y_{1},\ldots,y_{k}\in\mathcal{Y}$,
\begin{eqnarray*}
\|(c_{1}I\oplus\ldots\oplus c_{k}I)_{u}(y_{1},\ldots,y_{k})\|_{\mathcal{Y}_{u}^{\oplus k}}=\mathbb{E}\left\|\sum_{i=1}^{k}\delta_{i}c_{i}y_{i}\right\|=
\mathbb{E}\left\|\sum_{i=1}^{k}\delta_{i}y_{i}\right\|=\|(y_{1},\ldots,y_{k})\|_{\mathcal{Y}_{u}^{\oplus k}}.
\end{eqnarray*}
So $\|(c_{1}I\oplus\ldots\oplus c_{k}I)_{u}\|=1$.

The set $\{(c_{1},\ldots,c_{k})\in\mathbb{C}^{k}:\|(c_{1}I\oplus\ldots\oplus c_{k}I)_{u}\|\leq 1\}$ is convex and contains $\{-1,1\}^{k}$. So it contains $[-1,1]^{k}$. Thus, $\|(c_{1}I\oplus\ldots\oplus c_{k}I)_{u}\|\leq 1$ for all $c_{1},\ldots,c_{k}\in[-1,1]$. If $c_{1},\ldots,c_{k}\in\mathbb{C}$ and $|c_{i}|\leq 1$ for all $1\leq i\leq k$ then writing $c_{i}$ as the sum of its real and imaginary parts, we have $\|(c_{1}I\oplus\ldots\oplus c_{k}I)_{u}\|\leq 2$.
\end{proof}
From Lemma \ref{unconditional1}, we have the following result.
\begin{lemma}\label{normleq2}
Let $\mathcal{Y}$ be a Banach space. Let $M$ be a compact metric space. Let $w_{1},\ldots,w_{k}\in M$. Let $\eta:C(M)\to B(\mathcal{Y}_{u}^{\oplus k})$ be the homomorphism defined by
\[\eta(h)=(h(w_{1})I\oplus\ldots\oplus h(w_{k})I)_{u},\]
for $h\in C(M)$. Then $\|\eta\|\leq 2$.
\end{lemma}
\begin{lemma}\label{unconditional2}
Let $\mathcal{Y}$ be a Banach space. Let $T_{1},\ldots,T_{k}\in B(\mathcal{Y})$. Then
\[\left\|\sum_{i=1}^{k}T_{i}y_{i}\right\|\leq
\max_{\delta\in\{-1,1\}^{k}}\left\|\sum_{i=1}^{k}\delta_{i}T_{i}\right\|\|(y_{1},\ldots,y_{k})\|_{\mathcal{Y}_{u}^{\oplus k}},\]
for all $y_{1},\ldots,y_{k}\in\mathcal{Y}$.
\end{lemma}
\begin{proof}
Since
\[\sum_{i=1}^{k}T_{i}y_{i}=\mathbb{E}\left[\left(\sum_{i=1}^{k}\delta_{i}T_{i}\right)\left(\sum_{j=1}^{k}\delta_{j}y_{j}\right)\right],\]
we have
\[\left\|\sum_{i=1}^{k}T_{i}y_{i}\right\|\leq
\mathbb{E}\left[\left\|\sum_{i=1}^{k}\delta_{i}T_{i}\right\|\left\|\sum_{j=1}^{k}\delta_{j}y_{j}\right\|\right]\leq
\max_{\delta\in\{-1,1\}^{k}}\left\|\sum_{i=1}^{k}\delta_{i}T_{i}\right\|\|(y_{1},\ldots,y_{k})\|_{\mathcal{Y}_{u}^{\oplus k}}.\]
\end{proof}
\begin{lemma}\label{unconditional3}
Let $T_{1},\ldots,T_{k}\in B(l^{p})$. Define operators $W_{1}:l^{p}\to(l^{p})_{u}^{\oplus k}$ and $W_{2}:(l^{p})_{u}^{\oplus k}\to l^{p}$ by
\[W_{1}x=(T_{1}x,\ldots,T_{k}x)\quad\text{and}\quad W_{2}(y_{1},\ldots,y_{k})=\sum_{i=1}^{k}T_{i}y_{i},\]
for $x\in l^{p}$ and $y_{1},\ldots,y_{k}\in l^{p}$. Then $\|W_{1}\|_{e}$ and $\|W_{2}\|_{e}$ are bounded by $\displaystyle\max_{\delta\in\{-1,1\}^{k}}\left\|\sum_{i=1}^{k}\delta_{i}T_{i}\right\|_{e}$.
\end{lemma}
\begin{proof}
For each $n\in\mathbb{N}$, let $P_{n}$ be the projection from $l^{p}$ onto $l^{p}([1,n])$. For $n\in\mathbb{N}$ and $x\in l^{p}$,
\[\|W_{1}(I-P_{n})x\|=\mathbb{E}\left\|\sum_{i=1}^{k}\delta_{i}T_{i}(I-P_{n})x\right\|\leq
\max_{\delta\in\{-1,1\}^{k}}\left\|\sum_{i=1}^{k}\delta_{i}T_{i}(I-P_{n})\right\|\|x\|.\]
Hence,
\[\|W_{1}(I-P_{n})\|\leq\max_{\delta\in\{-1,1\}^{k}}\left\|\sum_{i=1}^{k}\delta_{i}T_{i}(I-P_{n})\right\|,\]
for all $n\in\mathbb{N}$. Taking $n\to\infty$, we have
\[\|W_{1}\|_{e}\leq\max_{\delta\in\{-1,1\}^{k}}\lim_{n\to\infty}\left\|\sum_{i=1}^{k}\delta_{i}T_{i}(I-P_{n})\right\|\leq
\mathbb{E}\left\|\sum_{i=1}^{k}\delta_{i}T_{i}\right\|_{e}.\]

By Lemma \ref{unconditional2}, for $n\in\mathbb{N}$ and $y_{1},\ldots,y_{k}\in l^{p}$,
\begin{eqnarray*}
\|(I-P_{n})W_{2}(y_{1},\ldots,y_{k})\|&\leq&\left\|\sum_{i=1}^{k}(I-P_{n})T_{i}y_{i}\right\|\\&\leq&
\max_{\delta\in\{-1,1\}^{k}}\left\|\sum_{i=1}^{k}\delta_{i}(I-P_{n})T_{i}\right\|\|(y_{1},\ldots,y_{k})\|_{\mathcal{Y}_{u}^{\oplus k}}.
\end{eqnarray*}
So
\[\|(I-P_{n})W_{2}\|\leq\max_{\delta\in\{-1,1\}^{k}}\left\|\sum_{i=1}^{k}\delta_{i}(I-P_{n})T_{i}\right\|,\]
for all $n\in\mathbb{N}$. Taking $n\to\infty$, we have
\[\|W_{2}\|_{e}\leq\max_{\delta\in\{-1,1\}^{k}}\left\|\sum_{i=1}^{k}\delta_{i}T_{i}\right\|_{e}.\]
\end{proof}
In the rest of this section, we prove Lemma \ref{isomorphicspaceu} below.
\begin{lemma}[\cite{Khintchine}]\label{Khintchine}
Let $1<p<\infty$. Then there exists $C_{p}\geq 1$ such that
\[\frac{1}{C_{p}}\left(\sum_{i=1}^{k}|v_{i}|^{2}\right)^{\frac{1}{2}}\leq\left(\mathbb{E}\left|
\sum_{i=1}^{k}\delta_{i}v_{i}\right|^{p}\right)^{\frac{1}{p}}\leq C_{p}\left(\sum_{i=1}^{k}|v_{i}|^{2}\right)^{\frac{1}{2}}\]
for all $k\in\mathbb{N}$ and $v_{1},\ldots,v_{k}\in\mathbb{C}$.
\end{lemma}
\begin{lemma}[\cite{Kahane}]\label{Kahane}
Let $1<p<\infty$. Then there exists $C_{p}\geq 1$ such that
\[\frac{1}{C_{p}}\mathbb{E}\left\|\sum_{i=1}^{k}\delta_{i}y_{i}\right\|\leq\left(\mathbb{E}\left\|
\sum_{i=1}^{k}\delta_{i}y_{i}\right\|^{p}\right)^{\frac{1}{p}}\leq C_{p}\mathbb{E}\left\|\sum_{i=1}^{k}\delta_{i}y_{i}\right\|,\]
for all $k\in\mathbb{N}$ and $y_{1},\ldots,y_{k}$ in a Banach space $\mathcal{Y}$.
\end{lemma}
\begin{lemma}[\cite{Pelczynski}]\label{isomorphicspace}
Let $1<p<\infty$. Then there exists $\lambda_{p}\geq 1$ such that $(\oplus_{n\in\mathbb{N}}\mathcal{H}_{n})_{l^{p}}$ is $\lambda_{p}$-isomorphic to $l^{p}$ for all nonzero finite dimensional Hilbert spaces $\mathcal{H}_{1},\mathcal{H}_{2},\ldots$.
\end{lemma}
\begin{lemma}\label{isomorphicspaceu}
Let $1<p<\infty$. Then there exists $\lambda_{p}\geq 1$ such that $(l^{p})_{u}^{\oplus k}$ is $\lambda_{p}$-isomorphic to $l^{p}$ for every $k\in\mathbb{N}$.
\end{lemma}
\begin{proof}
For each $n\in\mathbb{N}$, let $e_{n}^{*}\in(l^{p})^{*}$ be the $n$th coordinate functional, i.e., $e_{n}^{*}(x_{1},x_{2},\ldots)=x_{n}$ for $(x_{1},x_{2},\ldots)\in l^{p}$ and $n\in\mathbb{N}$. For each $n\in\mathbb{N}$, define $S_{n}:(l^{p})_{u}^{\oplus k}\to l^{2}([1,k])$ by
\[S_{n}(y_{1},\ldots,y_{k})=(e_{n}^{*}(y_{1}),\ldots,e_{n}^{*}(y_{k})),\]
for $(y_{1},\ldots,y_{k})\in(l^{p})_{u}^{\oplus k}$. Then
\[\|S_{n}(y_{1},\ldots,y_{k})\|=\left(\sum_{i=1}^{k}|e_{n}^{*}(y_{i})|^{2}\right)^{\frac{1}{2}}.\]
By Lemma \ref{Khintchine},
\begin{equation}\label{Khintchinesn}
\frac{1}{C_{p}}\left(\mathbb{E}\left|\sum_{i=1}^{k}\delta_{i}e_{n}^{*}(y_{i})\right|^{p}\right)^{\frac{1}{p}}\leq\|S_{n}(y_{1},\ldots,y_{k})\|\leq C_{p}\left(\mathbb{E}\left|\sum_{i=1}^{k}\delta_{i}e_{n}^{*}(y_{i})\right|^{p}\right)^{\frac{1}{p}}.
\end{equation}
Define $S:(l^{p})_{u}^{\oplus k}\to(l^{2}([1,k])\oplus l^{2}([1,k])\oplus\ldots)_{l^{p}}$ by
$Sy=(S_{1}y,S_{2}y,\ldots)$ for $y\in (l^{p})_{u}^{\oplus k}$. Then
\[\|Sy\|=\left(\sum_{n=1}^{\infty}\|S_{n}y\|^{p}\right)^{\frac{1}{p}}\]
so by (\ref{Khintchinesn}),
\[\frac{1}{C_{p}}\left(\mathbb{E}\sum_{n=1}^{\infty}\left|\sum_{i=1}^{k}\delta_{i}e_{n}^{*}(y_{i})\right|^{p}\right)^{\frac{1}{p}}\leq
\|S(y_{1},\ldots,y_{k})\|\leq C_{p}\left(\mathbb{E}\sum_{n=1}^{\infty}\left|\sum_{i=1}^{k}\delta_{i}e_{n}^{*}(y_{i})\right|^{p}\right)^{\frac{1}{p}}.\]
But
\[\mathbb{E}\sum_{n=1}^{\infty}\left|\sum_{i=1}^{k}\delta_{i}e_{n}^{*}(y_{i})\right|^{p}=\mathbb{E}\sum_{n=1}^{\infty}\left|e_{n}^{*}\left(
\sum_{i=1}^{k}\delta_{i}y_{i}\right)\right|^{p}=\mathbb{E}\left(\left\|\sum_{i=1}^{k}\delta_{i}y_{i}\right\|^{p}\right).\]
Thus,
\[\frac{1}{C_{p}^{p}}\mathbb{E}\left(\left\|\sum_{i=1}^{k}\delta_{i}y_{i}\right\|^{p}\right)\leq\|S(y_{1},\ldots,y_{k})\|^{p}\leq C_{p}^{p}\mathbb{E}\left(\left\|\sum_{i=1}^{k}\delta_{i}y_{i}\right\|^{p}\right).\]
So by Lemma \ref{Kahane}, we have $\frac{1}{C_{p}^{2}}\|y\|\leq\|Sy\|\leq C_{p}^{2}\|y\|$ for all $y\in(l^{p})_{u}^{\oplus k}$. It is easy to see that $S$ is surjective. Therefore, $(l^{p})_{u}^{\oplus k}$ is $C_{p}^{2}$-isomorphic to $(l^{2}([1,k])\oplus l^{2}([1,k])\oplus\ldots)_{l^{p}}$. By Lemma \ref{isomorphicspace}, the result follows.
\end{proof}
\section{Domination}
Recall that $|\;|_{x^{*}}$ is defined at the end of Section 1 for a bounded linear functional $x^{*}$ on a Banach space $\mathcal{X}$. Let $\Lambda$ be a set. Let $\mathcal{X}$ and $\mathcal{Y}$ be Banach spaces. Let $\psi:\Lambda\to B(\mathcal{X})$ and $\rho:\Lambda\to B(\mathcal{Y})$ be maps. Let $\lambda\geq 1$. We write $\psi\stackrel{\lambda}{\ll}\rho$ if there exist operators $V_{n}:\mathcal{X}\to\mathcal{Y}$ and $E_{n}:\mathcal{Y}\to\mathcal{X}$, for $n\in\mathbb{N}$, such that
\begin{enumerate}[(i)]
\item $\|V_{n}\|\leq\lambda$ and $\|E_{n}\|\leq 1$ for all $n\in\mathbb{N}$;
\item $E_{n}V_{n}\to I$ in SOT as $n\to\infty$;
\item $V_{n}\psi(\alpha)-\rho(\alpha)V_{n}\to 0$ in SOT, as $n\to\infty$, for all $\alpha\in\Lambda$;
\item $|E_{n}\rho(\alpha)-\psi(\alpha)E_{n}|_{x^{*}}\to 0$, as $n\to\infty$, for all $\alpha\in\Lambda$ and $x^{*}\in\mathcal{X}^{*}$; and
\item $V_{n}\to 0$ in WOT as $n\to\infty$.
\end{enumerate}
We write $\psi\ll\rho$ if $\psi\stackrel{\lambda}{\ll}\rho$ for some $\lambda\geq 1$. The motivation of this notion is the first lemma in \cite{Voiculescu} where it is shown that if $\mathcal{A}$ is a separable $C^{*}$-subalgebra of $B(l^{2})$ with $I\in\mathcal{A}$ and $\phi:\pi(\mathcal{A})\to B(l^{2})$ is a cyclic $*$-representation, then $\phi\circ\pi\ll\mathrm{id}$ where $\mathrm{id}:\mathcal{A}\to B(l^{2})$ is the identity representation.

The following result says that condition (v) above can be strengthened.
\begin{lemma}\label{llimprove}
Let $\Lambda$ be a countable set. Let $\lambda\geq 1$. Let $\mathcal{X}$ be a Banach space. Let $\psi:\Lambda\to B(\mathcal{X})$ and $\rho:\Lambda\to B(l^{p})$ be such that $\psi\stackrel{\lambda}{\ll}\rho$. Then for all numbers $m_{1}\in\mathbb{N}$ and $\epsilon>0$, finite subset $\Omega\subset\Lambda$ and finite dimensional subspaces $\mathcal{F}\subset\mathcal{X}$ and $\mathcal{G}\subset\mathcal{X}^{*}$, there exist $m_{2}>m_{1}$ and operators $V:\mathcal{X}\to l^{p}$ and $E:l^{p}\to\mathcal{X}$ such that
\begin{enumerate}[(i)]
\item $\|V\|\leq\lambda$ and $\|E\|\leq 1$;
\item $\|EVx-x\|\leq\epsilon\|x\|$ for all $x\in\mathcal{F}$;
\item $\|V\psi(\alpha)x-\rho(\alpha)Vx\|\leq\epsilon\|x\|$ for all $\alpha\in\Omega$ and $x\in\mathcal{F}$;
\item $|E\rho(\alpha)-\psi(\alpha)E|_{x^{*}}\leq\epsilon\|x^{*}\|$ for all $\alpha\in\Omega$ and $x^{*}\in\mathcal{G}$; and
\item $Vx\in l^{p}([m_{1}+1,m_{2}])$ for all $x\in\mathcal{X}$ and $Ey=0$ for all $y\in l^{p}([1,m_{1}]\cup(m_{2},\infty))$.
\end{enumerate}
\end{lemma}
\begin{proof}
Fix numbers $m_{1}\in\mathbb{N}$ and $\epsilon>0$, a finite subset $\Omega\subset\Lambda$ and a finite dimensional subspace $\mathcal{F}\subset\mathcal{X}$. By Lemma \ref{qcau}, there exists a finite rank diagonal operator $A_{1}$ on $l^{p}$ with diagonal entries in $[0,1]$ such that $A_{1}x=x$ for all $x\in l^{p}([1,m_{1}])$ and
\begin{equation}\label{llimproveeq1}
\|\rho(\alpha)A_{1}-A_{1}\rho(\alpha)\|\leq\epsilon,
\end{equation}
for all $\alpha\in\Omega$. Since $A_{1}$ is a finite rank diagonal operator, there exists $m_{1}'>m_{1}$ such that $A_{1}x=0$ for all $x\in l^{p}([m_{1}',\infty))$.

For each $m\in\mathbb{N}$, let $P_{m}$ be the projection from $l^{p}$ onto $l^{p}([1,m])$. Since $\psi\stackrel{\lambda}{\ll}\rho$, there exist operators $V_{0}:\mathcal{X}\to l^{p}$ and $E_{0}:l^{p}\to\mathcal{X}$ such that
\begin{enumerate}[(a)]
\item $\|V_{0}\|\leq\lambda$ and $\|E_{0}\|\leq 1$;
\item $\|E_{0}V_{0}x-x\|\leq\epsilon\|x\|$ for all $x\in\mathcal{F}$;
\item $\|V_{0}\psi(\alpha)x-\rho(\alpha)V_{0}x\|\leq\epsilon\|x\|$ for all $\alpha\in\Omega$ and $x\in\mathcal{F}$;
\item $|E_{0}\rho(\alpha)-\psi(\alpha)E_{0}|_{x^{*}}\leq\epsilon\|x^{*}\|$ for all $\alpha\in\Omega$ and $x\in\mathcal{G}$; and
\item $\displaystyle\|P_{m_{1}'}V_{0}x\|\leq\frac{\epsilon}{1+\max_{\alpha\in\Omega}\|\rho(\alpha)\|}\|x\|$ and $\|P_{m_{1}'}V_{0}\psi(\alpha)x\|\leq\epsilon\|x\|$ for all $x\in\mathcal{F}$ and $\alpha\in\Omega$.
\end{enumerate}
Let $N>m_{1}'$ be large enough so that
\begin{equation}\label{llimproveeq2}
\|(I-P_{N})V_{0}x\|\leq\frac{\epsilon}{1+\max_{\alpha\in\Omega}\|\rho(\alpha)\|}\|x\|\quad\text{and}\quad
\|(I-P_{N})V_{0}\psi(\alpha)x\|\leq\epsilon\|x\|,
\end{equation}
for all $x\in\mathcal{F}$ and $\alpha\in\Omega$. By Lemma \ref{qcau}, there exists a finite rank diagonal operator $A_{2}$ on $l^{p}$ with diagonal entries in $[0,1]$ such that $A_{2}x=x$ for all $x\in l^{p}([1,N])$ and
\begin{equation}\label{llimproveeq3}
\|\rho(\alpha)A_{2}-A_{2}\rho(\alpha)\|\leq\epsilon,
\end{equation}
for all $\alpha\in\Omega$. There exists $m_{2}>N$ such that $A_{2}x=0$ for all $x\in l^{p}((m_{2},\infty))$.

Take $V=(P_{N}-P_{m_{1}'})V_{0}$ and $E=E_{0}(A_{2}-A_{1})$. We have $\|V\|\leq\lambda$ and $\|E\|\leq 1$. Since $m_{1}<m_{1}'<N<m_{2}$, the range of $V$ is in $l^{p}([m_{1}+1,m_{2}])$ and $Ey=0$ for all $y\in l^{p}([1,m_{1}]\cup(m_{2},\infty))$. So we obtain (i) and (v).

Since $A_{1}=0$ on $l^{p}([m_{1}',\infty))$ and $N>m_{1}'$, we have $A_{1}(P_{N}-P_{m_{1}'})=0$. Since $A_{2}=I$ on $l^{p}([1,N])$ and $N>m_{1}'$, we have
$A_{2}(P_{N}-P_{m_{1}'})=P_{N}-P_{m_{1}'}$. Therefore,
\[(A_{2}-A_{1})(P_{N}-P_{m_{1}'})=P_{N}-P_{m_{1}'}.\]
Thus, for all $x\in\mathcal{F}$, we have
\begin{eqnarray*}
\|EVx-x\|&=&\|E_{0}(A_{2}-A_{1})(P_{N}-P_{m_{1}'})V_{0}x-x\|\\&=&\|E_{0}(P_{N}-P_{m_{1}'})V_{0}x-x\|\\&\leq&
\|E_{0}V_{0}x-x\|+\|(I-P_{N}+P_{m_{1}'})V_{0}x\|\\&\leq&
\epsilon\|x\|+\|(I-P_{N})V_{0}x\|+\|P_{m_{1}'}V_{0}x\|\text{ by (b)}\\&\leq&3\epsilon\|x\|\text{ by }(\ref{llimproveeq2})\text{ and (e)}.
\end{eqnarray*}
This proves (ii). For all $\alpha\in\Omega$ and $x\in\mathcal{F}$, we have
\begin{align*}
&\|V\psi(\alpha)x-\rho(\alpha)Vx\|\\=&
\|(P_{N}-P_{m_{1}'})V_{0}\psi(\alpha)x-\rho(\alpha)(P_{N}-P_{m_{1}'})V_{0}x\|\\\leq&\|V_{0}\psi(\alpha)x-\rho(\alpha)V_{0}x\|+
\|(I-P_{N}+P_{m_{1}'})V_{0}\psi(\alpha)x\|+\|\rho(\alpha)(I-P_{N}+P_{m_{1}'})V_{0}x\|\\\leq&
\epsilon\|x\|+\|(I-P_{N})V_{0}\psi(\alpha)x\|+\|P_{m_{1}'}V_{0}\psi(\alpha)x\|+\|\rho(\alpha)(I-P_{N})V_{0}x\|+
\|\rho(\alpha)P_{m_{1}'}V_{0}x\|\\&\text{ by (c)}\\\leq&5\epsilon\|x\|\text{ by }(\ref{llimproveeq2})\text{ and (e)}.
\end{align*}
This proves (iii). For all $\alpha\in\Omega$ and $x^{*}\in\mathcal{G}$, we have
\begin{align*}
&|E\rho(\alpha)-\psi(\alpha)E|_{x^{*}}\\=&
|E_{0}(A_{2}-A_{1})\rho(\alpha)-\psi(\alpha)E_{0}(A_{2}-A_{1})|_{x^{*}}\\\leq&
|E_{0}\rho(\alpha)(A_{2}-A_{1})-\psi(\alpha)E_{0}(A_{2}-A_{1})|_{x^{*}}+
\|E_{0}\|\|(A_{2}-A_{1})\rho(\alpha)-\rho(\alpha)(A_{2}-A_{1})\|\|x^{*}\|
\\\leq&|E_{0}\rho(\alpha)-\psi(\alpha)E_{0}|_{x^{*}}\|A_{2}-A_{1}\|+2\|E_{0}\|\epsilon\|x^{*}\|\text{ by }(\ref{llimproveeq1})\text{ and }(\ref{llimproveeq3})\\\leq&
3\epsilon\|x^{*}\|\text{ by (d) and (a)}.
\end{align*}
Thus, (iv) is proved.
\end{proof}
\begin{lemma}\label{EVI}
Let $0<\epsilon<1$. Suppose that $T$ is an operator on a Banach space $\mathcal{X}$ and $\mathcal{F}$ is a finite dimensional subspace of $\mathcal{X}$ such that
\[\|Tx-x\|\leq\frac{\epsilon}{\mathrm{dim}\,\mathcal{F}}\|x\|,\]
for all $x\in\mathcal{F}$. Then there exists an invertible operator $S$ on $\mathcal{X}$ such that $\|I-S^{-1}\|\leq\frac{\epsilon}{1-\epsilon}$ and $S^{-1}Tx=x$ for all $x\in\mathcal{F}$.
\end{lemma}
\begin{proof}
Define an operator $K:\mathcal{F}\to\mathcal{X}$ by $Kx=Tx-x$ for $x\in\mathcal{F}$. By Auerbach's lemma, we can extend $K$ to an operator $\widetilde{K}\in B(\mathcal{X})$ so that
\[\|\widetilde{K}\|\leq(\mathrm{dim}\,\mathcal{F})\|K\|\leq\epsilon.\]
Let $S=I+\widetilde{K}$. Then $S$ is an invertible operator on $\mathcal{X}$ such that $\|I-S^{-1}\|\leq\frac{\epsilon}{1-\epsilon}$ and
\[S^{-1}Tx=S^{-1}(I+K)x=S^{-1}(I+\widetilde{K})x=x,\]
for all $x\in\mathcal{F}$.
\end{proof}
The following result says that in Lemma \ref{llimprove}, we can have equality in (ii) but with (i) being slightly weakened. This can be proved by replacing $E$ in Lemma \ref{llimprove} by $S^{-1}E$ where $S$ is obtained from Lemma \ref{EVI} (for a different $\epsilon$).
\begin{lemma}\label{llimprove2}
Let $\Lambda$ be a countable set. Let $\lambda\geq 1$. Let $\mathcal{X}$ be a Banach space. Let $\psi:\Lambda\to B(\mathcal{X})$ and $\rho:\Lambda\to B(l^{p})$ be such that $\psi\stackrel{\lambda}{\ll}\rho$. Then for all numbers $m_{1}\in\mathbb{N}$ and $\epsilon>0$, finite subset $\Omega\subset\Lambda$ and finite dimensional subspaces $\mathcal{F}\subset\mathcal{X}$ and $\mathcal{G}\subset\mathcal{X}^{*}$, there exist $m_{2}>m_{1}$ and operators $V:\mathcal{X}\to l^{p}$ and $E:l^{p}\to\mathcal{X}$ such that
\begin{enumerate}[(i)]
\item $\|V\|\leq\lambda$ and $\|E\|\leq 1+\epsilon$;
\item $EVx=x$ for all $x\in\mathcal{F}$;
\item $\|V\psi(\alpha)x-\rho(\alpha)Vx\|\leq\epsilon\|x\|$ for all $\alpha\in\Omega$ and $x\in\mathcal{F}$;
\item $|E\rho(\alpha)-\psi(\alpha)E|_{x^{*}}\leq\epsilon$ for all $\alpha\in\Omega$ and $x^{*}\in\mathcal{G}$; and
\item $Vx\in l^{p}([m_{1}+1,m_{2}])$ for all $x\in\mathcal{X}$ and $Ey=0$ for all $y\in l^{p}([1,m_{1}]\cup(m_{2},\infty))$.
\end{enumerate}
\end{lemma}
\begin{lemma}\label{dsl}
Let $\Lambda$ be a countable set. Let $\lambda\geq 1$. Let $\mathcal{X}_{1},\mathcal{X}_{2},\ldots$ be Banach spaces. Let $\rho:\Lambda\to B(l^{p})$. For each $i\in\mathbb{N}$, let $\psi_{i}:\Lambda\to B(\mathcal{X}_{i})$ be such that $\psi_{i}\stackrel{\lambda}{\ll}\rho$. Then $\psi_{1}\oplus\psi_{2}\oplus\ldots\stackrel{\lambda}{\ll}\rho$.
\end{lemma}
\begin{proof}
Fix numbers $m_{1},j_{0}\in\mathbb{N}$ and $\epsilon>0$ and a finite subset $\Omega\subset\Lambda$. For each $1\leq i\leq j_{0}$, fix a finite dimensional subspace $\mathcal{F}_{i}\subset\mathcal{X}_{i}$. Consider the finite dimensional subspace
\[(\mathcal{F}_{1}\oplus\ldots\oplus\mathcal{F}_{j_{0}})_{l^{p}}=\{(x_{1},\ldots,x_{j_{0}},0,0,\ldots):x_{i}\in\mathcal{F}_{i}\text{ for }1\leq i\leq j_{0}\}\]
of $(\mathcal{X}_{1}\oplus\mathcal{X}_{2}\oplus\ldots)_{l^{p}}$ and the finite dimensional subspace
\[(\mathcal{G}_{1}\oplus\ldots\oplus\mathcal{G}_{j_{0}})_{l^{q}}=\{(x_{1}^{*},\ldots,x_{j_{0}}^{*},0,0,\ldots):x_{i}^{*}\in\mathcal{G}_{i}
\text{ for }1\leq i\leq j_{0}\}\]
of the dual $(\mathcal{X}_{1}^{*}\oplus\mathcal{X}_{2}^{*}\oplus\ldots)_{l^{q}}$ of $(\mathcal{X}_{1}\oplus\mathcal{X}_{2}\oplus\ldots)_{l^{p}}$, where $\frac{1}{p}+\frac{1}{q}=1$. It suffices to find $V:(\mathcal{X}_{1}\oplus\mathcal{X}_{2}\oplus\ldots)_{l^{p}}\to l^{p}$ and $E:l^{p}\to(\mathcal{X}_{1}\oplus\mathcal{X}_{2}\oplus\ldots)_{l^{p}}$ such that
\begin{enumerate}[(i)]
\item $\|V\|\leq\lambda$ and $\|E\|\leq 1$;
\item $\|EVx-x\|\leq\epsilon\|x\|$ for all $x\in(\mathcal{F}_{1}\oplus\ldots\oplus\mathcal{F}_{j_{0}})_{l^{p}}$;
\item $\|V(\psi_{1}(\alpha)\oplus\psi_{2}(\alpha)\oplus\ldots)x-\rho(\alpha)Vx\|\leq\epsilon\|x\|$ for all $\alpha\in\Omega$ and $x\in(\mathcal{F}_{1}\oplus\ldots\oplus\mathcal{F}_{j_{0}})_{l^{p}}$;
\item $|E\rho(\alpha)-(\psi_{1}(\alpha)\oplus\psi_{2}(\alpha)\oplus\ldots)E|_{x^{*}}\leq\epsilon\|x^{*}\|$ for all $\alpha\in\Omega$ and $x^{*}\in(\mathcal{G}_{1}\oplus\ldots\oplus\mathcal{G}_{j_{0}})_{l^{q}}$; and
\item $Vx\in l^{p}([m_{1},\infty))$ for all $x\in(\mathcal{X}_{1}\oplus\mathcal{X}_{2}\oplus\ldots)_{l^{p}}$.
\end{enumerate}
By using Lemma \ref{llimprove} recursively, we obtain integers $m_{2}<\ldots<m_{j_{0}+1}$ and operators $V_{i}:\mathcal{X}_{i}\to l^{p}$ and $E_{i}:l^{p}\to\mathcal{X}_{i}$, for $1\leq i\leq j_{0}$, such that $m_{2}>m_{1}$ and for every $1\leq i\leq j_{0}$, we have
\begin{enumerate}[(a)]
\item $\|V_{i}\|\leq\lambda$ and $\|E_{i}\|\leq1$;
\item $\|E_{i}V_{i}x_{i}-x_{i}\|\leq\frac{1}{j_{0}}\epsilon\|x_{i}\|$ for all $x_{i}\in\mathcal{F}_{i}$;
\item $\|V_{i}\psi_{i}(\alpha)x_{i}-\rho(\alpha)V_{i}x_{i}\|\leq\frac{1}{j_{0}}\epsilon\|x_{i}\|$ for all $\alpha\in\Omega$ and $x_{i}\in\mathcal{F}_{i}$;
\item $|E_{i}\rho(\alpha)-\psi_{i}(\alpha)E_{i}|_{x_{i}^{*}}\leq\frac{1}{j_{0}}\epsilon\|x_{i}^{*}\|$ for all $\alpha\in\Omega$ and $x_{i}^{*}\in\mathcal{G}_{i}$; and
\item $V_{i}x_{i}\in l^{p}([m_{i}+1,m_{i+1}])$ for all $x_{i}\in\mathcal{X}_{i}$ and $E_{i}y=0$ for all $y\in l^{p}([1,m_{i}]\cup(m_{i+1},\infty))$.
\end{enumerate}
Take
\[V(x_{1},x_{2},\ldots)=V_{1}x_{1}+\ldots+V_{j_{0}}x_{j_{0}},\]
for $(x_{1},x_{2},\ldots)\in(\mathcal{X}_{1}\oplus\mathcal{X}_{2}\oplus\ldots)_{l^{p}}$, and
\[Ey=(E_{1}y,\ldots,E_{j_{0}}y,0,0,\ldots),\]
for $y\in l^{p}$. For $(x_{1},x_{2},\ldots)\in(\mathcal{X}_{1}\oplus\mathcal{X}_{2}\oplus\ldots)_{l^{p}}$, since $V_{i}x_{i}\in l^{p}([m_{i}+1,m_{i+1}])$ by (e),
\[\|V(x_{1},x_{2},\ldots)\|^{p}=\sum_{i=1}^{j_{0}}\|V_{i}x_{i}\|^{p}\leq\lambda^{p}\sum_{i=1}^{j_{0}}\|x_{i}\|^{p}.\]
Thus, $\|V\|\leq\lambda$. For each $m\in\mathbb{N}$, let $P_{m}$ be the projection from $l^{p}$ onto $l^{p}([1,m])$. Since $E_{i}=E_{i}(P_{m_{i+1}}-P_{m_{i}})$ by (e),
\[\|Ey\|^{p}=\sum_{i=1}^{j_{0}}\|E_{i}y\|^{p}=\sum_{i=1}^{j_{0}}\|E_{i}(P_{m_{i+1}}-P_{m_{i}})y\|^{p}\leq
\sum_{i=1}^{j_{0}}\|(P_{m_{i+1}}-P_{m_{i}})y\|^{p}\leq\|y\|^{p}.\]
So $\|E\|\leq 1$. Hence we obtain (i).

Let $x=(x_{1},\ldots,x_{j_{0}},0,0,\ldots)\in(\mathcal{F}_{1}\oplus\ldots\oplus\mathcal{F}_{j_{0}})_{l^{p}}$. Since $E_{i_{1}}V_{i_{2}}=0$ when $i_{1}\neq i_{2}$,
\[EVx=(E_{1}V_{1}x_{1},\ldots,E_{j_{0}}V_{j_{0}}x_{j_{0}},0,0\ldots).\]
So by (b),
\[\|EVx-x\|\leq\sum_{i=1}^{j_{0}}\|E_{i}V_{i}x_{i}-x_{i}\|\leq\sum_{i=1}^{j_{0}}\frac{1}{j_{0}}\epsilon\|x_{i}\|\leq\epsilon\|x\|.\]
This proves (ii).

For $x=(x_{1},\ldots,x_{j_{0}},0,0,\ldots)\in(\mathcal{F}_{1}\oplus\ldots\oplus\mathcal{F}_{j_{0}})_{l^{p}}$,
\[V(\psi_{1}(\alpha)\oplus\psi_{2}(\alpha)\oplus\ldots)x-\rho(\alpha)Vx=\sum_{i=1}^{j_{0}}(V_{i}\psi_{i}(\alpha)x_{i}-\rho(\alpha)V_{i}x_{i}),\]
so by (c),
\[\|V(\psi_{1}(\alpha)\oplus\psi_{2}(\alpha)\oplus\ldots)x-\rho(\alpha)Vx\|\leq\sum_{i=1}^{j_{0}}\frac{1}{j_{0}}\epsilon\|x_{i}\|\leq\epsilon\|x\|.\]
This proves (iii). For $y\in l^{p}$,
\begin{align*}
&E\rho(\alpha)y-(\psi_{1}(\alpha)\oplus\psi_{2}(\alpha)\oplus\ldots)Ey\\=&
(E_{1}\rho(\alpha)y-\psi_{1}(\alpha)E_{1}y,\,\ldots,\,E_{j_{0}}\rho(\alpha)y-\psi_{j_{0}}(\alpha)E_{j_{0}}y,0,0,\ldots),
\end{align*}
so by (d), for $x^{*}=(x_{1}^{*},\ldots,x_{j_{0}}^{*},0,0,\ldots)\in(\mathcal{G}_{1}\oplus\ldots\oplus\mathcal{G}_{j_{0}})_{l^{q}}$,
\begin{align*}
&|E\rho(\alpha)-(\psi_{1}(\alpha)\oplus\psi_{2}(\alpha)\oplus\ldots)E|_{x^{*}}\\=&
\sup_{\|y\|=1}\left|\sum_{i=1}^{j_{0}}x_{i}^{*}(E_{i}\rho(\alpha)y-\psi_{i}(\alpha)E_{i}y)\right|\\\leq&
\sum_{i=1}^{j_{0}}|E_{i}\rho(\alpha)-\psi_{i}(\alpha)E_{i}|_{x_{i}^{*}}\leq\sum_{i=1}^{j_{0}}\frac{1}{j_{0}}\epsilon\|x_{i}^{*}\|
\leq\epsilon\|x^{*}\|.
\end{align*}
This proves (iv). By (e), we have $V_{i}x_{i}\in l^{p}([m_{1},\infty))$ for all $x\in\mathcal{X}_{i}$ and $i\in\mathbb{N}$. So $Vx\in l^{p}([m_{1},\infty))$ for all $x\in(\mathcal{X}_{1}\oplus\mathcal{X}_{2}\oplus\ldots)_{l^{p}}$. This proves (v).
\end{proof}
\section{$l^{p}$ version of Voiculescu's absorption theorem}
\begin{lemma}\label{embed}
Let $\Lambda$ be a countable set. Let $\lambda\geq 1$. Let $\mathcal{X}$ be a Banach space that is either finite dimensional or $\lambda$-isomorphic to $l^{p}$. Let $\psi:\Lambda\to B(\mathcal{X})$ and $\rho:\Lambda\to B(l^{p})$ be such that $\psi\stackrel{\lambda}{\ll}\rho$. Then for all $\epsilon>0$ and finite subset $\Lambda_{0}\subset\Lambda$, there exist operators $L:\mathcal{X}\to l^{p}$ and $R:l^{p}\to\mathcal{X}$ satisfying
\begin{enumerate}[(I)]
\item $\|L\|\leq2\lambda^{2}$ and $\|R\|\leq4\lambda(1+\epsilon)$;
\item $L\psi(\alpha)-\rho(\alpha)L$ is compact for all $\alpha\in\Lambda$ and has norm at most $\epsilon$ for all $\alpha\in\Lambda_{0}$;
\item $R\rho(\alpha)-\psi(\alpha)R$ is compact for all $\alpha\in\Lambda$ and has norm at most $\epsilon$ for all $\alpha\in\Lambda_{0}$; and
\item $RL=I$.
\end{enumerate}
\end{lemma}
\begin{proof}
Fix $\epsilon>0$ and finite subset $\Lambda_{0}\subset\Lambda$. Let $\Omega_{-1}\subset\Omega_{0}\subset\Omega_{1}\subset\ldots$ be finite subsets of $\Lambda$ such that $\cup_{n=-1}^{\infty}\Omega_{n}=\Lambda$ and $\Lambda_{0}\subset\Omega_{-1}$.

By Lemma \ref{qcaum}, there are finite rank operators $A_{1},A_{2},\ldots$ on $\mathcal{X}$ such that
\begin{enumerate}[(i)]
\item $\displaystyle\sup_{n\in\mathbb{N}}\|A_{n}\|\leq\lambda$;
\item $A_{n}\to I$ in SOT, as $n\to\infty$;
\item $\displaystyle\|A_{n}\psi(\alpha)-\psi(\alpha)A_{n}\|\leq\frac{\epsilon}{2^{n}}$ for all $\alpha\in\Omega_{n}$ and $n\geq -1$;
\item $(A_{n+1}-A_{n-2})(A_{n}-A_{n-1})=A_{n}-A_{n-1}$ for all $n\geq -1$;
\item \[\left(\sum_{n=1}^{\infty}\|(A_{n}-A_{n-1})x\|^{p}\right)^{\frac{1}{p}}\leq2\lambda\|x\|,\quad x\in\mathcal{X};\text{ and}\]
\item  if $(x_{n})_{n\in\mathbb{N}}$ is a sequence in $\mathcal{X}$ such that $\displaystyle\sum_{n=1}^{\infty}\|x_{n}\|^{p}<\infty$, then
\[\left(\left\|\sum_{n=1}^{\infty}(A_{n+1}-A_{n-2})x_{n}\right\|^{p}\right)^{\frac{1}{p}}\leq4\lambda\left(\sum_{n=1}^{\infty}\|x_{n}\|^{p}\right)^{
\frac{1}{p}}.\]
\end{enumerate}
By using Lemma \ref{llimprove2} recursively, we obtain integers $1<m_{1}<m_{2}<\ldots$ and operators $V_{n}:\mathcal{X}\to l^{p}$ and $E_{n}:l^{p}\to\mathcal{X}$, for $n\in\mathbb{N}$, such that for every $n\in\mathbb{N}$, we have
\begin{enumerate}[(a)]
\item $\|V_{n}\|\leq\lambda$ and $\|E_{n}\|\leq 1+\epsilon$;
\item $E_{n}V_{n}x=x$ for all $x$ in the range of $A_{n}-A_{n-1}$;
\item $\displaystyle\|V_{n}\psi(\alpha)x-\rho(\alpha)V_{n}x\|\leq\frac{\epsilon}{2^{n}}\|x\|$ for all $\alpha\in\Omega_{n}$ and $x$ in the range of $A_{n}-A_{n-1}$;
\item $\displaystyle|E_{n}\rho(\alpha)-\psi(\alpha)E_{n}|_{x^{*}}\leq\frac{\epsilon}{2^{n}}\|x^{*}\|$ for all $\alpha\in\Omega_{n}$ and $x^{*}$ in the range of $(A_{n+1}-A_{n-2})^{*}$; and
\item $V_{n}x\in l^{p}([m_{n}+1,m_{n+1}])$ for all $x\in\mathcal{X}$ and $E_{n}y=0$ for all $y\in l^{p}([1,m_{n}]\cup(m_{n+1},\infty))$.
\end{enumerate}
Take
\[Lx=\sum_{n=1}^{\infty}V_{n}(A_{n}-A_{n-1})x,\]
for $x\in\mathcal{X}$, and
\[Ry=\sum_{n=1}^{\infty}(A_{n+1}-A_{n-2})E_{n}y,\]
for $y\in l^{p}$, where $A_{0}=A_{-1}=0$. For every $x\in\mathcal{X}$,
\begin{eqnarray*}
\|Lx\|&=&\left\|\sum_{n=1}^{\infty}V_{n}(A_{n}-A_{n-1})x\right\|
\\&=&\left(\sum_{n=1}^{\infty}\|V_{n}(A_{n}-A_{n-1})x\|^{p}\right)^{\frac{1}{p}}\text{ by (e)}
\\&\leq&\lambda\left(\sum_{n=1}^{\infty}\|(A_{n}-A_{n-1})x\|^{p}\right)^{\frac{1}{p}}\text{ by (a)}
\\&\leq&2\lambda^{2}\|x\|\text{ by (v)}.
\end{eqnarray*}
For each $n\in\mathbb{N}$, let $P_{n}$ be the projection from $l^{p}$ onto $l^{p}([1,n])$. For every $y\in l^{p}$,
\begin{eqnarray*}
\|Ry\|&=&\left\|\sum_{n=1}^{\infty}(A_{n+1}-A_{n-2})E_{n}y\right\|
\\&\leq&4\lambda\left(\sum_{n\in\mathbb{N}}\|E_{n}y\|^{p}\right)^{\frac{1}{p}}\text{ by (vi)}
\\&\leq&4\lambda\left(\sum_{n\in\mathbb{N}}\|E_{n}(P_{m_{n+1}}-P_{m_{n}})y\|^{p}\right)^{\frac{1}{p}}\text{ by (e)}
\\&\leq&4\lambda(1+\epsilon)\left(\sum_{n\in\mathbb{N}}\|(P_{m_{n+1}}-P_{m_{n}})y\|^{p}\right)^{\frac{1}{p}}\text{ by (a)}
\\&\leq&4\lambda(1+\epsilon)\|y\|.
\end{eqnarray*}
Thus (I) is proved.

For all $\alpha\in\Omega_{n-1}$ and $n\in\mathbb{N}$, we have
\begin{align*}
&\|V_{n}(A_{n}-A_{n-1})\psi(\alpha)-V_{n}\psi(\alpha)(A_{n}-A_{n-1})\|\\\leq&\lambda\|(A_{n}-A_{n-1})\psi(\alpha)-\psi(\alpha)(A_{n}-A_{n-1})\|\text{ by (a)}\\\leq&
\lambda(\|A_{n}\psi(\alpha)-\psi(\alpha)A_{n}\|+\|A_{n-1}\psi(\alpha)-\psi(\alpha)A_{n-1}\|)\\\leq&
\lambda\left(\frac{\epsilon}{2^{n}}+\frac{\epsilon}{2^{n-1}}\right)\text{ by (iii)}\\=&
\frac{3\lambda\epsilon}{2^{n}},
\end{align*}
and
\begin{eqnarray*}
\|V_{n}\psi(\alpha)(A_{n}-A_{n-1})-\rho(\alpha)V_{n}(A_{n}-A_{n-1})\|&\leq&\frac{\epsilon}{2^{n}}\|A_{n}-A_{n-1}\|\text{ by (c)}\\&\leq&
\frac{2\lambda\epsilon}{2^{n}}\text{ by (i)}.
\end{eqnarray*}
Combining these two inequalities, we obtain
\begin{equation}\label{boundL}
\|V_{n}(A_{n}-A_{n-1})\psi(\alpha)-\rho(\alpha)V_{n}(A_{n}-A_{n-1})\|\leq\frac{5\lambda\epsilon}{2^{n}},
\end{equation}
for all $\alpha\in\Omega_{n-1}$ and $n\in\mathbb{N}$. For all $\alpha\in\Lambda$ and $x\in\mathcal{X}$,
\[L\psi(\alpha)x=\sum_{n=1}^{\infty}V_{n}(A_{n}-A_{n-1})\psi(\alpha)x\quad\text{and}\quad\rho(\alpha)Lx=
\sum_{n=1}^{\infty}\rho(\alpha)V_{n}(A_{n}-A_{n-1})x.\]
Recall the properties of $\Omega_{n}$ at the beginning of the proof. By (\ref{boundL}),
\[\sum_{n=1}^{\infty}\|V_{n}(A_{n}-A_{n-1})\psi(\alpha)-\rho(\alpha)V_{n}(A_{n}-A_{n-1})\|\]
is finite for all $\alpha\in\Lambda$ and is at most $5\lambda\epsilon$ for all $\alpha\in\Lambda_{0}$. Since
\[V_{n}(A_{n}-A_{n-1})\psi(\alpha)-\rho(\alpha)V_{n}(A_{n}-A_{n-1})\]
has finite rank for every $n\in\mathbb{N}$, it follows that $L\psi(\alpha)-\rho(\alpha)L$ is compact for all $\alpha\in\Lambda$ and has norm at most $5\lambda\epsilon$ for all $\alpha\in\Lambda_{0}$. Thus (II) is proved.

Let $x^{*}\in\mathcal{X}^{*}$. Let $n\in\mathbb{N}$. Let $x_{1}^{*}=(A_{n+1}-A_{n-2})^{*}x^{*}$. For all $\alpha\in\Omega_{n-2}$,
\begin{eqnarray*}
|(A_{n+1}-A_{n-2})E_{n}\rho(\alpha)-(A_{n+1}-A_{n-2})\psi(\alpha)E_{n}|_{x^{*}}&=&|E_{n}\rho(\alpha)-\psi(\alpha)E_{n}|_{x_{1}^{*}}\\&\leq&
\frac{\epsilon}{2^{n}}\|x_{1}^{*}\|\text{ by (d)}\\&\leq&\frac{2\lambda\epsilon}{2^{n}}\|x^{*}\|\text{ by (i)},
\end{eqnarray*}
and
\begin{align*}
&\|(A_{n+1}-A_{n-2})\psi(\alpha)E_{n}-\psi(\alpha)(A_{n+1}-A_{n-2})E_{n}\|\\\leq&
(1+\epsilon)\|(A_{n+1}-A_{n-2})\psi(\alpha)-\psi(\alpha)(A_{n+1}-A_{n-2})\|\text{ by (a)}\\\leq&
(1+\epsilon)(\|A_{n+1}\psi(\alpha)-\psi(\alpha)A_{n+1}\|+\|A_{n-2}\psi(\alpha)-\psi(\alpha)A_{n-2}\|)\\\leq&
(1+\epsilon)\left(\frac{\epsilon}{2^{n+1}}+\frac{\epsilon}{2^{n-2}}\right)\text{ by (iii)}\\\leq&
\frac{5\lambda(1+\epsilon)\epsilon}{2^{n}}.
\end{align*}
Combining these two inequalities, we obtain
\[|(A_{n+1}-A_{n-2})E_{n}\rho(\alpha)-\psi(\alpha)(A_{n+1}-A_{n-2})E_{n}|_{x^{*}}\leq\frac{7\lambda(1+\epsilon)\epsilon}{2^{n}}\|x^{*}\|,\]
for all $\alpha\in\Omega_{n-2}$, functional $x^{*}\in\mathcal{X}^{*}$ and $n\in\mathbb{N}$. Thus,
\begin{equation}\label{boundR}
\|(A_{n+1}-A_{n-2})E_{n}\rho(\alpha)-\psi(\alpha)(A_{n+1}-A_{n-2})E_{n}\|\leq\frac{7\lambda(1+\epsilon)\epsilon}{2^{n}},
\end{equation}
for all $\alpha\in\Omega_{n-2}$ and $n\in\mathbb{N}$. For all $y\in l^{p}$ and $\alpha\in\Lambda$,
\[R\rho(\alpha)y=\sum_{n=1}^{\infty}(A_{n+1}-A_{n-2})E_{n}\rho(\alpha)y\quad\text{and}\quad
\psi(\alpha)Ry=\sum_{n=1}^{\infty}\psi(\alpha)(A_{n+1}-A_{n-2})E_{n}y.\]
By (\ref{boundR}),
\[\sum_{n=1}^{\infty}\|(A_{n}-A_{n-1})E_{n}\rho(\alpha)-\psi(\alpha)E_{n}(A_{n}-A_{n-1})\|\]
is finite for all $\alpha\in\Lambda$ and has norm at most $7\lambda(1+\epsilon)\epsilon$ for all $\alpha\in\Lambda_{0}$. Since
\[(A_{n}-A_{n-1})E_{n}\rho(\alpha)-\psi(\alpha)E_{n}(A_{n}-A_{n-1})\]
has finite rank for every $n\in\mathbb{N}$, it follows that $R\rho(\alpha)-\psi(\alpha)R$ is compact for all $\alpha\in\Lambda$ and has norm at most $7\lambda(1+\epsilon)\epsilon$ for all $\alpha\in\Lambda_{0}$. Thus (III) is proved.

By (e), we have $E_{n_{1}}V_{n_{2}}=0$ when $n_{1}\neq n_{2}$. So for every $x\in\mathcal{X}$,
\begin{eqnarray*}
RLx&=&\sum_{n=1}^{\infty}(A_{n+1}-A_{n-2})E_{n}V_{n}(A_{n}-A_{n-1})x\\&=&
\sum_{n=1}^{\infty}(A_{n+1}-A_{n-2})(A_{n}-A_{n-1})x\text{ by (b)}\\&=&
\sum_{n=1}^{\infty}(A_{n}-A_{n-1})x\text{ by (iv)}\\&=&
x\text{ by (ii)}.
\end{eqnarray*}
Thus, (IV) is proved.
\end{proof}
\begin{lemma}\label{embed2}
Let $\Lambda$ be a countable set. Let $\lambda\geq 1$. Let $1<p<\infty$. Let $\mathcal{X}$ be a Banach space that is either finite dimensional or $\lambda$-isomorphic to $l^{p}$. Let $\psi:\Lambda\to B(\mathcal{X})$ and $\rho:\Lambda\to B(l^{p})$ be such that $\psi\stackrel{\lambda}{\ll}\rho$. Then for all $\epsilon>0$ and finite subset $\Lambda_{0}\subset\Lambda$, there exist a Banach space $\mathcal{Y}$, an invertible operator $S:\mathcal{X}\oplus\mathcal{Y}\to l^{p}$ and a map $\rho_{2}:\Lambda\to B(\mathcal{Y})$ such that
\begin{enumerate}[(i)]
\item $\|S\|\leq2\lambda^{2}+1$ and $\|S^{-1}\|\leq(4\lambda+8\lambda^{3}+1)(1+\epsilon)$; and
\item $\rho(\alpha)-S(\psi(\alpha)\oplus\rho_{2}(\alpha))S^{-1}$ is compact for all $\alpha\in\Lambda$ and has norm at most $\epsilon$ for all $\alpha\in\Lambda_{0}$.
\end{enumerate}
\end{lemma}
\begin{proof}
Let $L:\mathcal{X}\to l^{p}$ and $R:l^{p}\to\mathcal{X}$ be as in Lemma \ref{embed}. Since $RL=I$, the operator $LR$ is an idempotent on $l^{p}$. Take $\mathcal{Y}$ to be the range of $I-LR$. Take $S(x,y)=Lx+y$ for $x\in\mathcal{X}$ and $y\in\mathcal{Y}$. It is easy to see that $S^{-1}z=(Rz,(I-LR)z)$ for $z\in l^{p}$. Thus, by (I) in Lemma \ref{embed}, we have $\|S\|\leq2\lambda^{2}+1$ and $\|S^{-1}\|\leq(4\lambda+8\lambda^{3}+1)(1+\epsilon)$.

Take $\rho_{2}(\alpha)y=(I-LR)\rho(\alpha)y$ for $y\in\mathcal{Y}$ and $\alpha\in\Lambda$. Then
\[(\psi(\alpha)\oplus\rho_{2}(\alpha))S^{-1}z=(\psi(\alpha)Rz,(I-LR)\rho(\alpha)(I-LR)z),\]
for all $\alpha\in\Lambda$ and $z\in l^{p}$. Hence,
\[S(\psi(\alpha)\oplus\rho_{2}(\alpha))S^{-1}=L\psi(\alpha)R+(I-LR)\rho(\alpha)(I-LR),\]
for all $\alpha\in\Lambda$. By (II) and (III) in Lemma \ref{embed}, it follows that $\rho(\alpha)-S(\psi(\alpha)\oplus\rho_{2}(\alpha))S^{-1}$ is compact for all $\alpha\in\Lambda$ and has norm at most $\epsilon$ for all $\alpha\in\Lambda_{0}$.
\end{proof}
\begin{lemma}[\cite{Pelczynski}]\label{prime}
If $P$ is an idempotent on $l^{p}$ then the range of $P$ is either finite dimensional or isomorphic to $l^{p}$.
\end{lemma}
If $\mathcal{X}$ is a Banach space and $\psi:\Lambda\to B(\mathcal{X})$ and $\rho:\Lambda\to B(l^{p})$ are maps such that $\psi\ll\rho$, then there exist operators $V:\mathcal{X}\to l^{p}$ and $E:l^{p}\to\mathcal{X}$ such that $\|EV-I\|<1$. Thus, $EV\in B(\mathcal{X})$ is invertible and $V(EV)^{-1}E$ is an idempotent on $l^{p}$. Moreover, $V$ defines an invertible from $\mathcal{X}$ onto the range of this idempotent. Therefore, by Lemma \ref{prime}, the space $\mathcal{X}$ is either finite dimensional or isomorphic to $l^{p}$.

The following result is an $l^{p}$ version of Voiculescu's absorption theorem where the notion of $\stackrel{\lambda}{\ll}$ is defined at the beginning of Section 6.
\begin{theorem}\label{lpVoic}
Let $1<p<\infty$.  Let $\mathcal{X}_{1},\mathcal{X}_{2},\ldots$ be Banach spaces. Let $\Lambda$ be a countable set. Let $\rho:\Lambda\to B(l^{p})$. Suppose that $\lambda\geq 1$ and for each $i\in\mathbb{N}$, we have $\psi_{i}:\Lambda\to B(\mathcal{X}_{i})$ such that $\psi_{i}\stackrel{\lambda}{\ll}\rho$. Then $\rho$ is approximately similar to $\rho\oplus\psi_{1}\oplus\psi_{2}\oplus\ldots$.
\end{theorem}
\begin{proof}
By Lemma \ref{dsl}, we have $\psi_{1}\oplus\psi_{2}\oplus\ldots\ll\rho$. So it suffices to show that $\rho$ is approximately similar to $\rho\oplus\psi$ for all $\psi:\Lambda\to B(\mathcal{X})$ such that $\psi\ll\rho$ where $\mathcal{X}$ is any Banach space. As remarked above, $\psi\ll\rho$ automatically implies that $\mathcal{X}$ is either finite dimensional or isomorphic to $l^{p}$.

By Lemma \ref{dsl}, if $\psi\ll\rho$ then $\psi\oplus\psi\oplus\ldots\ll\rho$. Let $\epsilon>0$. Let $\Lambda_{0}$ be a finite subset of $\Lambda$. By Lemma \ref{embed2}, there exist a Banach space $\mathcal{Y}$ and maps $\rho_{2}:\Lambda\to B(\mathcal{Y})$ and $\kappa:\Lambda\to K(l^{p})$ such that $\|\kappa(\alpha)\|\leq\epsilon$, for all $\alpha\in\Lambda_{0}$, and $\rho+\kappa$ is similar to $(\psi\oplus\psi\oplus\ldots)\oplus\rho_{2}$, where $(\rho+\kappa)(\alpha)=\rho(\alpha)+\kappa(\alpha)$ for $\alpha\in\Lambda$.

Note that $(\psi\oplus\psi\oplus\ldots)\oplus\rho_{2}$ is similar to $(\psi\oplus\psi\oplus\ldots)\oplus\rho_{2}\oplus\psi$, which is similar to $(\rho+\kappa)\oplus\psi$. Therefore, $\rho+\kappa$ is similar to $(\rho+\kappa)\oplus\psi$. So $\rho$ is approximately similar to $\rho\oplus\psi$.
\end{proof}
Let $T_{1},T_{2}\in B(l^{p})$. Let $\Lambda$ be a singleton and define $\psi_{1}:\Lambda\to B(l^{p})$ and $\psi_{2}:\Lambda\to B(l^{p})$ by $\psi_{1}(\alpha)=T_{1}$ and $\psi_{2}(\alpha)=T_{2}$. We write $T_{1}\ll T_{2}$ if $\psi_{1}\ll\psi_{2}$. The operators $T_{1},T_{2}$ are {\it approximately similar} if $\psi_{1}$ and $\psi_{2}$ are approximately similar. We can also define $\ll$ and the notion of approximate similarity for tuples of operators. The following result is the single operator version of Theorem \ref{lpVoic}.
\begin{theorem}\label{lpVoicsingleoperator}
Let $T_{1}$ and $T_{2}$ be operators on $l^{p}$. If $T_{1}\ll T_{2}$ then $T_{2}$ is approximately similar to $T_{2}\oplus T_{1}\oplus T_{1}\oplus\ldots$.
\end{theorem}
\begin{corollary}\label{UB}
Let $U\in B(l^{p}(\mathbb{N}))$ and $B\in B(l^{p}(\mathbb{Z}))$ be the unilateral and bilateral shifts, respectively, i.e.,
\[Ue_{n}=e_{n+1},\]
for $n\in\mathbb{N}$, and
\[Be_{n}=e_{n+1},\]
$n\in\mathbb{Z}$. Then $U$ is approximately similar to $U\oplus B\oplus B\oplus\ldots$.
\end{corollary}
\begin{proof}
Let $P$ be the projection from $l^{p}(\mathbb{Z})$ onto $l^{p}(\mathbb{N})$. For $n\in\mathbb{N}$, define operators $V_{n}:l^{p}(\mathbb{Z})\to l^{p}(\mathbb{N})$ and $E_{n}:l^{p}(\mathbb{N})\to l^{p}(\mathbb{Z})$ by
\[V_{n}x=PB^{n}x\quad\text{and}\quad E_{n}y=B^{-n}y,\]
for $x\in l^{p}(\mathbb{Z})$ and $y\in l^{p}(\mathbb{N})$. It is easy that (i) $\|V_{n}\|=\|E_{n}\|=1$ for all $n\in\mathbb{N}$, (ii) $E_{n}V_{n}\to I$ in SOT, as $n\to\infty$, (iii) $V_{n}B-UV_{n}\to 0$ in SOT, as $n\to\infty$, (iv) $E_{n}U=BE_{n}$ for all $n\in\mathbb{N}$ and (v) $V_{n}\to 0$ in WOT, as $n\to\infty$. So $B\ll U$. By Theorem \ref{lpVoicsingleoperator}, the result follows.
\end{proof}
Let $\mathcal{X}$ be a Banach space. Two operators $T_{1},T_{2}\in B(\mathcal{X})$ are {\it norm approximately similar} if there are invertible  operators $S_{n}:l^{p}\to l^{p}$, for $n\in\mathbb{N}$, such that
\[\sup_{n\in\mathbb{N}}\|S_{n}\|\|S_{n}^{-1}\|<\infty\]
and
\[\lim_{n\to\infty}\|T_{1}-S_{n}T_{2}S_{n}^{-1}\|=0.\]
\begin{corollary}\label{boundedsimorbit}
Let $T_{1},T_{2}\in B(l^{p})$. Suppose that $T_{1}$ and $T_{2}$ are norm approximately similar and that $T_{1}$ is norm approximately similar to $T_{1}\oplus T_{1}\oplus\ldots$. Then $T_{1}$ and $T_{2}$ are approximately similar.
\end{corollary}
\begin{proof}
The operators $T_{2}$ and $T_{1}\oplus T_{1}\oplus\ldots$ are norm approximately similar. So there are invertible operators $W_{n}:l^{p}\to(l^{p}\oplus l^{p}\oplus\ldots)_{l^{p}}$ such that $\displaystyle\sup_{n\in\mathbb{N}}\|W_{n}\|$ and $\displaystyle\sup_{n\in\mathbb{N}}\|W_{n}^{-1}\|$ are finite and
\[\lim_{n\to\infty}\|T_{1}\oplus T_{1}\oplus\ldots-W_{n}T_{2}W_{n}^{-1}\|=0.\]
For each $n\in\mathbb{N}$, let $J_{n}:l^{p}\to(l^{p}\oplus l^{p}\oplus\ldots)_{l^{p}}$ be the canonical embedding that maps $l^{p}$ onto the $n$th component of $(l^{p}\oplus l^{p}\oplus\ldots)_{l^{p}}$ and let $Q_{n}:(l^{p}\oplus l^{p}\oplus\ldots)_{l^{p}}\to l^{p}$ be the canonical projection onto the $n$th component of $(l^{p}\oplus l^{p}\oplus\ldots)_{l^{p}}$. Note that $J_{n}x\to 0$ weakly, as $n\to\infty$, for every $x\in l^{p}$. So there are $k_{1}<k_{2}<\ldots$ such that $W_{n}^{-1}J_{k_{n}}x\to 0$ weakly, as $n\to\infty$, for every $x\in l^{p}$. For $n\in\mathbb{N}$, define operators $V_{n}:l^{p}\to l^{p}$ and $E_{n}:l^{p}\to l^{p}$ by
\[V_{n}x=W_{n}^{-1}J_{k_{n}}x\quad\text{and}\quad E_{n}y=Q_{k_{n}}W_{n}y,\]
for $x\in l^{p}$ and $y\in l^{p}$. It is easy to check that (i) $\sup_{n\in\mathbb{N}}\|V_{n}\|$ and $\sup_{n\in\mathbb{N}}\|E_{n}\|$ are finite, (ii) $E_{n}V_{n}=I$ for all $n\in\mathbb{N}$, (iii) \[\|V_{n}T_{1}-T_{2}V_{n}\|=\|(W_{n}^{-1}(T_{1}\oplus T_{1}\oplus\ldots)-T_{2}W_{n}^{-1})J_{k_{n}}\|\to 0,\]
as $n\to\infty$, (iv)
\[\|E_{n}T_{2}-T_{1}E_{n}\|=\|Q_{k_{n}}(W_{n}T_{2}-(T_{1}\oplus T_{1}\oplus\ldots)W_{n})\|\to 0,\]
as $n\to\infty$, and (v) $V_{n}\to 0$ in WOT as $n\to\infty$. Thus, $T_{1}\ll T_{2}$. By Theorem \ref{lpVoicsingleoperator}, we have that $T_{2}$ is approximately similar to $T_{2}\oplus T_{1}$.

Since $T_{1}$ and $T_{2}$ are norm approximately similar and $T_{1}$ is norm approximately similar to $T_{1}\oplus T_{1}\oplus\ldots$, the operator $T_{2}$ is norm approximately similar to $T_{2}\oplus T_{2}\oplus\ldots$. Thus, interchanging the roles of $T_{1}$ and $T_{2}$ and applying  the conclusion of the previous paragraph, we obtain that $T_{1}$ is approximately similar to $T_{1}\oplus T_{2}$. It follows that $T_{1}$ and $T_{2}$ are approximately similar.
\end{proof}
\begin{lemma}\label{diagspectrum}
Let $n\in\mathbb{N}$. For $0\leq r\leq n$ and $k\in\mathbb{Z}$, let $u_{k,r}\in\mathbb{C}$. Assume that $\sup_{k,r}|u_{k,r}|<\infty$. For each $0\leq r\leq n$, let $D_{r}$ be the diagonal operator on $l^{p}(\mathbb{Z})$ defined by $D_{r}e_{k}=u_{r,k}e_{k}$, for $k\in\mathbb{Z}$. Let $\lambda_{1},\ldots,\lambda_{n}\in\mathbb{C}$. Then\[\inf_{x\in l^{p},\,\|x\|=1}(\|(D_{0}-\lambda_{0})x\|^{p}+\ldots+\|(D_{n}-\lambda_{n})x\|^{p})=0\]if and only if $(\lambda_{0},\ldots,\lambda_{n})$ is in the closure of $\{(u_{0,k},\ldots,u_{n,k}):k\in\mathbb{Z}\}$.
\end{lemma}
\begin{proof}
Let $x=(x_{k})_{k\in\mathbb{Z}}\in l^{p}(\mathbb{Z})$. We have
\begin{eqnarray*}
\|(D_{0}-\lambda_{0})x\|^{p}+\ldots+\|(D_{n}-\lambda_{n})x\|^{p}&=&\sum_{k\in\mathbb{Z}}|(u_{0,k}-\lambda_{0})x_{k}|^{p}+\ldots+
\sum_{k\in\mathbb{Z}}|(u_{n,k}-\lambda_{n})x_{k}|^{p}\\&=&\sum_{k\in\mathbb{Z}}(|u_{0,k}-\lambda_{0}|^{p}+\ldots+|u_{n,k}-\lambda_{n}|^{p})|x_{k}|^{p}.
\end{eqnarray*}
Thus,
\[\inf_{x\in l^{p},\,\|x\|=1}(\|(D_{0}-\lambda_{0})x\|^{p}+\ldots+\|(D_{n}-\lambda_{n})x\|^{p})=\inf_{k\in\mathbb{Z}}(|u_{0,k}-\lambda_{0}|^{p}+\ldots+|u_{n,k}-
\lambda_{n}|^{p}).\]
So the result follows.
\end{proof}
\begin{corollary}\label{DB1}
Let $(u_{k})_{k\in\mathbb{Z}}$ and $(v_{k})_{k\in\mathbb{Z}}$ be bounded two sided sequences in $\mathbb{C}$. Let $D_{u}$ and $D_{v}$ be their corresponding diagonal operators on $l^{p}(\mathbb{Z})$. Let $B$ be the bilateral shift on $l^{p}(\mathbb{Z})$. Then the following statements are equivalent.
\begin{enumerate}[(1)]
\item $(D_{\mathbf{u}},B)$ is approximately similar to $(D_{\mathbf{v}},B)$
\item The closures of $\{(u_{k},\ldots,u_{k+n}):k\in\mathbb{Z}\}$ and $\{(v_{k},\ldots,v_{k+n}):k\in\mathbb{Z}\}$ coincide for every $n\in\mathbb{N}$.
\end{enumerate}
\end{corollary}
\begin{proof}
Suppose that $(D_{\mathbf{u}},B)$ is approximately similar to $(D_{\mathbf{v}},B)$. Let $n\in\mathbb{N}$. Then the $(n+1)$-tuples of operators
\[(D_{\mathbf{u}},B^{-1}D_{\mathbf{u}}B,\ldots,B^{-n}D_{\mathbf{u}}B^{n})\quad\text{and}
\quad(D_{\mathbf{v}},B^{-1}D_{\mathbf{v}}B,\ldots,B^{-n}D_{\mathbf{v}}B^{n})\]
are approximately similar. Note that
\[(B^{-r}D_{\mathbf{u}}B^{r})e_{k}=u_{k+r}e_{k}\quad\text{and}\quad(B^{-r}D_{\mathbf{v}}B^{r})e_{k}=v_{k+r}e_{k}\]
for all $0\leq r\leq n$ and $k\in\mathbb{Z}$. Observe that for $\lambda_{0},\ldots,\lambda_{n}\in\mathbb{C}$, the set of all $(n+1)$-tuples $(D_{0},\ldots,D_{n})$ of diagonal operators on $l^{p}(\mathbb{Z})$ satisfying
\[\inf_{x\in l^{p},\,\|x\|=1}(\|(D_{0}-\lambda_{0})x\|^{p}+\ldots+\|(D_{n}-\lambda_{n})x\|^{p})=0\]
is invariant under approximate similarity. So using Lemma \ref{diagspectrum} with $D_{r}=B^{-r}D_{\mathbf{u}}B^{-r}$, for $0\leq r\leq n$, and using Lemma \ref{diagspectrum} with $D_{r}=B^{-r}D_{\mathbf{v}}B^{-r}$, for $0\leq r\leq n$, we obtain that for $\lambda_{0},\ldots,\lambda_{n}\in\mathbb{C}$, the tuple $(\lambda_{0},\ldots,\lambda_{n})$ is in the closure of $\{(u_{n},\ldots,u_{n+k}):k\in\mathbb{Z}\}$ if and only if $(\lambda_{0},\ldots,\lambda_{n})$ is in the closure of $\{(v_{k},\ldots,v_{k+n}):k\in\mathbb{Z}\}$. This means that the closures of $\{(u_{k},\ldots,u_{k+n}):k\in\mathbb{Z}\}$ and $\{(v_{k},\ldots,v_{k+n}):k\in\mathbb{Z}\}$ coincide.

Conversely, suppose that the closures of $\{(u_{k},\ldots,u_{k+n}):k\in\mathbb{Z}\}$ and $\{(v_{k},\ldots,v_{k+n}):k\in\mathbb{Z}\}$ coincide for every $n\in\mathbb{N}$. Let $m\in\mathbb{N}$. Then $(u_{-m},\ldots,u_{m})$ is in the closure of $\{(v_{k},\ldots,v_{k+2m}):k\in\mathbb{Z}\}$. So there exists $k_{m}\in\mathbb{Z}$ such that
\[\|(u_{-m},\ldots,u_{m})-(v_{k_{m}},\ldots,v_{k_{m}+2m})\|_{\infty}\leq\frac{1}{m}.\]
Thus, $|u_{j}-v_{k_{m}+j+m}|\leq\frac{1}{m}$ for all $-m\leq j\leq m$. So there exist $t_{1},t_{2},\ldots\in\mathbb{Z}$ such that $|u_{j}-v_{j+t_{m}}|\leq\frac{1}{m}$ for all $-m\leq j\leq m$ and $m\in\mathbb{N}$ (take $t_{m}=k_{m}+m$). So
\begin{equation}\label{aapprb}
u_{k}=\lim_{m\to\infty}v_{k+t_{m}},
\end{equation}
for all $k\in\mathbb{Z}$.

Since $t_{1},t_{2},\ldots\in\mathbb{Z}$, passing to a subsequence, we have that either $t_{m}=t$ for all $m\in\mathbb{N}$ or $|t_{m}|\to\infty$ as $m\to\infty$. In the first case, $u_{k}=v_{k+t}$ for $k\in\mathbb{Z}$. We have $D_{\mathbf{u}}=B^{-t}D_{\mathbf{v}}B^{t}$ and so $(D_{\mathbf{u}},B)$ and $(D_{\mathbf{v}},B)$ are approximately similar. In the second case, for $n\in\mathbb{N}$, consider operators $V_{n}:l^{p}(\mathbb{Z})\to l^{p}(\mathbb{Z})$ and  $E_{n}:l^{p}(\mathbb{Z})\to l^{p}(\mathbb{Z})$ defined by $V_{n}=B^{t_{n}}$ and $E_{n}=B^{-t_{n}}$. It is easy to check that (i) $\|V_{n}\|=\|E_{n}\|=1$ for all $n\in\mathbb{N}$, (ii) $E_{n}V_{n}=I$ for all $n\in\mathbb{N}$, (iii) $V_{n}B=BV_{n}$ for all $n\in\mathbb{N}$, (iv) $E_{n}B=BE_{n}$ for all $n\in\mathbb{N}$ and (v) $V_{n}\to 0$ in WOT, as $n\to\infty$. This only gives that $B\ll B$. But we want to show that $(D_{\mathbf{u}},B)\ll (D_{\mathbf{v}},B)$. To obtain this, it suffices to show that $V_{n}D_{\mathbf{u}}-D_{\mathbf{v}}V_{n}\to 0$ in SOT and $|E_{n}D_{\mathbf{v}}-D_{\mathbf{u}}E_{n}|_{y^{*}}\to 0$, as $n\to\infty$, for every $y^{*}\in l^{p}(\mathbb{Z})^{*}$.

For all $k\in\mathbb{Z}$ and $n\in\mathbb{N}$, we have $(V_{n}D_{\mathbf{u}}-D_{\mathbf{v}}V_{n})e_{k}=(u_{k}-v_{k+t_{n}})e_{k+t_{n}}$ so by (\ref{aapprb}), we have $\|(V_{n}D_{\mathbf{u}}-D_{\mathbf{v}}V_{n})e_{k}\|\to 0$, as $n\to\infty$. Hence, $V_{n}D_{\mathbf{u}}-D_{\mathbf{v}}V_{n}\to 0$ in SOT, as $n\to\infty$.

Since
\[E_{n}D_{\mathbf{v}}-D_{\mathbf{u}}E_{n}=
B^{-t_{n}}D_{\mathbf{v}}-D_{\mathbf{u}}B^{-t_{n}}=(B^{-t_{n}}D_{\mathbf{v}}B^{t_{n}}-D_{\mathbf{u}})B^{-t_{n}}\]
and $B^{-t_{n}}D_{\mathbf{v}}B^{t_{n}}-D_{\mathbf{u}}$ is a diagonal operator whose $k$th diagonal entry is $v_{k+t_{n}}-u_{k}$, it follows that
\[e_{k}^{*}(E_{n}D_{\mathbf{v}}y-D_{\mathbf{u}}E_{n}y)=e_{k}^{*}((B^{-t_{n}}D_{\mathbf{v}}B^{t_{n}}-D_{\mathbf{u}})B^{-t_{n}}y)=
(v_{k+t_{n}}-u_{k})e_{k}^{*}(B^{-t_{n}}y),\]
and so
\[|e_{k}^{*}(E_{n}D_{\mathbf{v}}y-D_{\mathbf{u}}E_{n}y)|\leq|v_{k+t_{n}}-u_{k}|\|y\|,\]
for all numbers $k\in\mathbb{Z}$ and $n\in\mathbb{N}$ and $y\in l^{p}(\mathbb{Z})$. By (\ref{aapprb}), we deduce that $|E_{n}D_{\mathbf{v}}x-D_{\mathbf{u}}E_{n}|_{e_{k}^{*}}\to 0$, as $n\to\infty$, and so $|E_{n}D_{\mathbf{v}}-D_{\mathbf{u}}E_{n}|_{y^{*}}\to 0$, as $n\to\infty$, for all $y^{*}\in l^{p}(\mathbb{Z})^{*}$.

Therefore, $(D_{\mathbf{u}},B)\ll (D_{\mathbf{v}},B)$. By Theorem \ref{lpVoic}, we have that $(D_{\mathbf{v}},B)$ is approximately similar to $(D_{\mathbf{v}}\oplus D_{\mathbf{u}},B\oplus B)$. Interchanging the roles of $(u_{n})_{n\in\mathbb{Z}}$ and $(v_{n})_{n\in\mathbb{Z}}$, we have that $(D_{\mathbf{u}},B)$ is approximately similar to $(D_{\mathbf{u}}\oplus D_{\mathbf{v}},B\oplus B)$. But $(D_{\mathbf{u}}\oplus D_{\mathbf{v}},B\oplus B)$ is similar to $(D_{\mathbf{v}}\oplus D_{\mathbf{u}},B\oplus B)$. Therefore, $(D_{\mathbf{u}},B)$ is approximately similar to $(D_{\mathbf{v}},B)$.
\end{proof}
The following result follows immediately from Corollary \ref{DB1}. The case $p=2$ was obtained in \cite{ODonovan} (see also \cite{Marcoux}).
\begin{corollary}\label{DB2}
Let $(u_{k})_{k\in\mathbb{Z}}$ and $(v_{k})_{k\in\mathbb{Z}}$ be bounded two sided sequences in $\mathbb{C}$. Assume that the closures of $\{(u_{k},\ldots,u_{k+n}):k\in\mathbb{Z}\}$ and $\{(v_{k},\ldots,v_{k+n}):k\in\mathbb{Z}\}$ coincide for every $n\in\mathbb{Z}$. Define operators $T_{\mathbf{u}},T_{\mathbf{v}}\in B(l^{p}(\mathbb{Z}))$ by
\[T_{\mathbf{u}}e_{k}=u_{k}e_{k+1}\quad\text{and}\quad T_{\mathbf{v}}e_{k}=v_{n}e_{k+1},\]
for $k\in\mathbb{Z}$. Then $T_{\mathbf{u}}$ and $T_{\mathbf{v}}$ are approximately similar.
\end{corollary}
\begin{remark}
If all $u_{n}$ and $v_{n}$ are chosen independently from the same distribution $\mu$ on a compact subset of $\mathbb{C}$, then the assumption of Corollary \ref{DB2} is satisfied almost surely.
\end{remark}
\begin{corollary}\label{TXR}
Let $T\in B(l^{p})$. For each $k\in\mathbb{N}$, let $P_{k}$ be the projection from $l^{p}$ onto $l^{p}([1,k])$. Consider the operator $X$ on $(\oplus_{k=1}^{\infty}l^{p}([1,k]))_{l^{p}}$,
\[X=\bigoplus_{k=1}^{\infty}P_{k}T|_{l^{p}([1,k])}.\]
Then $X$ is approximately similar to $X\oplus T\oplus T\oplus\ldots$.
\end{corollary}
\begin{proof}
For each $n\in\mathbb{N}$, let $J_{n}:l^{p}([1,n])\to(\oplus_{k=1}^{\infty}l^{p}([1,k]))_{l^{p}}$ and $Q_{n}:(\oplus_{k=1}^{\infty}l^{p}([1,k]))_{l^{p}}\to l^{p}([1,n])$ be the canonical embedding and projection, respectively. Let $\mathcal{Y}=(\oplus_{k=1}^{\infty}l^{p}([1,k]))_{l^{p}}$. For $n\in\mathbb{N}$, define operators $V_{n}:l^{p}\to\mathcal{Y}$ and $E_{n}:\mathcal{Y}\to l^{p}$ by
\[V_{n}x=J_{n}P_{n}x\quad\text{and}\quad E_{n}y=Q_{n}y,\]
for $x\in l^{p}$ and $y\in\mathcal{Y}$. It is easy to see that
\begin{enumerate}[(i)]
\item $\|V_{n}\|=\|E_{n}\|=1$ for all $n\in\mathbb{N}$;
\item $E_{n}V_{n}=P_{n}\to I$ in SOT, as $n\to\infty$;
\item\[V_{n}Tx-XV_{n}x=J_{n}P_{n}Tx-XJ_{n}P_{n}x=J_{n}P_{n}Tx-J_{n}P_{n}TP_{n}x=J_{n}P_{n}T(I-P_{n})x,\]
for $n\in\mathbb{N}$ and $x\in l^{p}(\mathbb{Z})$, and so $V_{n}T-XV_{n}\to 0$ in SOT as $n\to\infty$;
\item\[E_{n}Xy-TE_{n}y=Q_{n}Xy-TQ_{n}y=P_{n}TQ_{n}y-TQ_{n}y=(I-P_{n})TQ_{n}y,\]
and so $|E_{n}X-TE_{n}|_{y^{*}}\to 0$, as $n\to\infty$, for every $y^{*}\in l^{p}(\mathbb{Z})^{*}$; and
\item $V_{n}\to 0$ in WOT as $n\to\infty$.
\end{enumerate}
Therefore, $T\ll X$. By Theorem \ref{lpVoicsingleoperator}, the result follows.
\end{proof}
\begin{lemma}\label{amenablesubembed}
Let $G$ be a countable infinite group. Let $H$ be an amenable subgroup of $G$. Consider the homomorphisms $\psi:G\to B(l^{p}(G))$ and $\psi_{H}:G\to B(l^{p}(G/H))$ defined by
\[\psi(s)e_{g}=e_{sg}\quad\text{and}\quad\psi_{H}(s)e_{tH}=e_{stH},\]
for $s,g\in G$. Then $\psi_{H}\stackrel{1}{\ll}\psi$.
\end{lemma}
\begin{proof}
Case 1: $H$ is infinite.

Since $H$ is amenable, there is a F{\o}nler sequence $F_{1},F_{2},\ldots$ for $H$. Let $\mathcal{T}$ be a collection of elements of $G$ such that $t_{1}H\neq t_{2}H$, for all $t_{1}\neq t_{2}$ in $\mathcal{T}$, and $\cup_{t\in\mathcal{T}}tH=G$ (i.e., we pick a representative from each coset). For $n\in\mathbb{N}$, define an operator $V_{n}:l^{p}(G/H)\to l^{p}(G)$ by
\begin{equation}\label{ve1}
V_{n}e_{tH}=\frac{1}{|F_{n}|^{\frac{1}{p}}}\sum_{r\in tF_{n}}e_{r},
\end{equation}
for $t\in\mathcal{T}$. Note that $V_{n}e_{t_{1}H}$ and $V_{n}e_{t_{2}H}$ have disjoint supports for all $t_{1}\neq t_{2}$ in $\mathcal{T}$. So $\|V_{n}\|=1$ for all $n\in\mathbb{N}$.

For each finite subset $F$ of $G$, let $y_{F}^{*}\in l^{p}(G)^{*}$,
\[y_{F}^{*}(x)=\frac{1}{|F|^{1-\frac{1}{p}}}\sum_{s\in F}e_{s}^{*}(x),\]
for $x\in l^{p}(G)$, and let $P_{F}$ be the projection from $l^{p}(G)$ onto $l^{p}(F)$. By H\"older's inequality,
\begin{equation}\label{boundxjast}
|y_{F}^{*}(x)|\leq\|P_{F}x\|,
\end{equation}
for $x\in l^{p}(G)$. Also if $F$ and $F'$ are finite subsets of $G$ with $|F|=|F'|$, then
\begin{equation}\label{xastperturb}
\|y_{F}^{*}-y_{F'}^{*}\|=\frac{|F\Delta F'|^{1-\frac{1}{p}}}{|F|^{1-\frac{1}{p}}}.
\end{equation}
For $n\in\mathbb{N}$, define an operator $E_{n}:l^{p}(G)\to l^{p}(G/H)$ by
\[E_{n}y=\sum_{t\in\mathcal{T}}y_{tF_{n}}^{*}(y)e_{tH},\]
for $y\in l^{p}(G)$. By (\ref{boundxjast}),
\[\|E_{n}y\|^{p}=\sum_{t\in\mathcal{T}}\left|y_{tF_{n}}^{*}(y)\right|^{p}\leq\sum_{t\in\mathcal{T}}\|P_{tH}y\|^{p}=\|y\|^{p},\]
for $y\in l^{p}(G)$. So $\|E_{n}\|\leq 1$ for all $n\in\mathbb{N}$. For every $t\in\mathcal{T}$,
\begin{eqnarray*}
E_{n}V_{n}e_{tH}&=&\sum_{g\in\mathcal{T}}\left(y_{gF_{n}}^{*}\left(\frac{1}{|F_{n}|^{\frac{1}{p}}}\sum_{r\in tF_{n}}e_{r}\right)\right)e_{gH}\\&=&
y_{tF_{n}}^{*}\left(\frac{1}{|F_{n}|^{\frac{1}{p}}}\sum_{r\in tF_{n}}e_{r}\right)e_{tH}+\sum_{g\neq t}\left(y_{gF_{n}}^{*}\left(\frac{1}{|F_{n}|^{\frac{1}{p}}}\sum_{r\in tF_{n}}e_{r}\right)\right)e_{gH}\\&=&e_{tH}+0=e_{tH}.
\end{eqnarray*}
So $E_{n}V_{n}=I$ for all $n\in\mathbb{N}$.

Fix $s\in G$. There exists a permutation $\sigma$ on $\mathcal{T}$ such that $stH=\sigma(t)H$ for all $t\in\mathcal{T}$. Thus, there exists a map $h:\mathcal{T}\to H$ such that $st=\sigma(t)h(t)$ for all $t\in\mathcal{T}$. Thus, for all $t\in\mathcal{T}$,
\begin{eqnarray*}
V_{n}\psi_{H}(s)e_{tH}-\psi(s)V_{n}e_{tH}&=&V_{n}e_{\sigma(t)H}-\frac{1}{|F_{n}|^{\frac{1}{p}}}\sum_{r\in tF_{n}}\psi(s)e_{r}\\&=&
\frac{1}{|F_{n}|^{\frac{1}{p}}}\sum_{r\in \sigma(t)F_{n}}e_{r}-\frac{1}{|F_{n}|^{\frac{1}{p}}}\sum_{r\in stF_{n}}e_{r}\\&=&
\frac{1}{|F_{n}|^{\frac{1}{p}}}\sum_{r\in \sigma(t)F_{n}}e_{r}-\frac{1}{|F_{n}|^{\frac{1}{p}}}\sum_{r\in \sigma(t)h(t)F_{n}}e_{r}.
\end{eqnarray*}
Hence,
\[\|V_{n}\psi_{H}(s)e_{tH}-\psi(s)V_{n}e_{tH}\|=\frac{|\sigma(t)F_{n}\Delta\sigma(t)h(t)F_{n}|^{\frac{1}{p}}}{|F_{n}|^{\frac{1}{p}}}=
\frac{|F_{n}\Delta h(t)F_{n}|^{\frac{1}{p}}}{|F_{n}|^{\frac{1}{p}}}\to 0,\]
as $n\to\infty$. So $V_{n}\psi_{H}(s)-\psi(s)V_{n}\to 0$ in SOT, as $n\to\infty$.

For all $t\in\mathcal{T}$ and $y\in l^{p}(G)$,
\begin{eqnarray*}
e_{tH}^{*}(E_{n}\psi(s^{-1})y)-e_{tH}^{*}(\psi_{H}(s^{-1})E_{n}y)&=&
y_{tF_{n}}^{*}(\psi(s^{-1})y)-e_{stH}^{*}(E_{n}y)\\&=&
y_{stF_{n}}^{*}(y)-e_{\sigma(t)H}^{*}(E_{n}y)\\&=&
y_{\sigma(t)h(t)F_{n}}^{*}(y)-y_{\sigma(t)F_{n}}^{*}(y).
\end{eqnarray*}
By (\ref{xastperturb}), it follows that
\begin{eqnarray*}
|e_{tH}^{*}(E_{n}\psi(s^{-1})y)-e_{tH}^{*}(\psi_{H}(s^{-1})E_{n}y)|&\leq&\|y_{\sigma(t)h(t)F_{n}}^{*}-y_{\sigma(t)F_{n}}^{*}\|\|y\|
\\&=&\frac{|h(t)F_{n}\Delta F_{n}|^{1-\frac{1}{p}}}{|F_{n}|^{1-\frac{1}{p}}}\|y\|,
\end{eqnarray*}
for $t\in\mathcal{T}$ and $y\in l^{p}(G)$. Thus, $|E_{n}\psi(s^{-1})-\psi_{H}(s^{-1})E_{n}|_{e_{tH}^{*}}\to 0$, as $n\to\infty$, for every $t\in\mathcal{T}$. Thus, $|E_{n}\psi(s^{-1})-\psi_{H}(s^{-1})E_{n}|_{y^{*}}\to 0$, as $n\to\infty$, for every $y^{*}\in l^{p}(G)^{*}$. Finally, it is easy to see that $V_{n}\to 0$ in WOT, as $n\to\infty$. Therefore, $\psi_{H}\stackrel{1}{\ll}\psi$.

Case 2: $H$ is finite.

Let $s_{1},s_{2},\ldots$ be distinct elements of $G$. For $n\in\mathbb{N}$, define an operator $V_{n}:l^{p}(G/H)\to l^{p}(G)$ by
\begin{equation}\label{ve2}
V_{n}e_{gH}=\frac{1}{|H|^{\frac{1}{p}}}\sum_{r\in gHs_{n}}e_{r},
\end{equation}
for $g\in G$. For every finite subset $F$ of $G$, let $y_{F}^{*}\in l^{p}(G)^{*}$,
\[y_{F}^{*}(x)=\frac{1}{|F|^{1-\frac{1}{p}}}\sum_{s\in F}e_{s}^{*}(x),\]
for $x\in l^{p}(G)$. For $n\in\mathbb{N}$, define an operator $E_{n}:l^{p}(G)\to l^{p}(G/H)$ by
\[E_{n}y=\sum_{gH\in G/H}y_{gHs_{n}}^{*}(y)e_{gH},\]
for $y\in l^{p}(G)$. It is easy to check that (i) $\|V_{n}\|\leq 1$ and $\|E_{n}\|\leq 1$ for all $n\in\mathbb{N}$, (ii) $E_{n}V_{n}=I$ for all $n\in\mathbb{N}$, (iii) $V\psi_{H}(s)=\psi(s)V_{n}$ for all $n\in\mathbb{N}$, (iv) $E_{n}\psi(s)=\psi_{H}(s)E_{n}$ for all $n\in\mathbb{N}$ and (v) $V_{n}\to 0$ in WOT, as $n\to\infty$. Therefore, $\psi_{H}\stackrel{1}{\ll}\psi$.
\end{proof}
\begin{remark}
In Case 1, the fact that $H$ is infinite implies that $|F_{n}|\to\infty$ so that we have $V_{n}\to 0$ in WOT as $n\to\infty$. In Case 2, we put $s_{n}$ in the definition of $V_{n}$ in (\ref{ve2}) so that $V_{n}\to 0$ in WOT as $n\to\infty$.
\end{remark}
\begin{corollary}\label{grouprepdecomp}
Let $G$ be a countable infinite group. Consider the homomorphism $\psi:G\to B(l^{p}(G))$ defined by $\psi(s)e_{g}=e_{sg}$ for $s,g\in G$. Let $H_{1},H_{2},\ldots$ be amenable subgroups of $G$ such that $\cap_{i=1}^{\infty}\cup_{n=i}^{\infty}H_{n}$ is trivial. For $k\in\mathbb{N}$, let $G/H_{k}$ be the set of all left cosets of $H_{k}$ in $G$ and let $\psi_{k}:G\to B(l^{p}(G/H_{k}))$ be defined by $\psi_{k}(s)e_{gH_{k}}=e_{sgH_{k}}$ for $s,g\in G$. Then $\psi$ and $\oplus_{k=1}^{\infty}\psi_{k}$ are approximately similar.
\end{corollary}
\begin{proof}
For each $n\in\mathbb{N}$, let $J_{n}:l^{p}(G/H_{n})\to\oplus_{k=1}^{\infty}l^{p}(G/H_{k})$ and $Q_{n}:\oplus_{k=1}^{\infty}l^{p}(G/H_{k})\to l^{p}(G/H_{n})$ be the canonical embedding and projection, respectively. Let $\mathcal{Y}=\oplus_{k=1}^{\infty}l^{p}(G/H_{k})$.

Since $\cap_{i=1}^{\infty}\cup_{n=i}^{\infty}H_{n}$ is trivial, there exist $\mathcal{S}_{1}\subset\mathcal{S}_{2}\subset\ldots$ in $G$ such that $\cup_{n=1}^{\infty}\mathcal{S}_{n}=G$ and $g_{1}H_{n}\neq g_{2}H_{n}$ for all $g_{1}\neq g_{2}$ in $\mathcal{S}_{n}$. For $n\in\mathbb{N}$, define operators $V_{n}:l^{p}(G)\to\mathcal{Y}$ and $E_{n}:\mathcal{Y}\to l^{p}(G)$ by
\[V_{n}e_{g}=\begin{cases}J_{n}e_{gH_{n}},&\quad g\in\mathcal{S}_{n}\\0,&\text{Otherwise}\end{cases}\quad\text{and}\quad E_{n}y=\sum_{g\in\mathcal{S}_{n}}e_{gH_{n}}^{*}(Q_{n}y)e_{g},\]
for $g\in G$ and $y\in\mathcal{Y}$.
\begin{enumerate}[(i)]
\item We have $\|V_{n}\|=1$, for all $n\in\mathbb{N}$, and
\[\|E_{n}y\|^{p}=\sum_{g\in\mathcal{S}_{n}}|e_{gH_{n}}^{*}(Q_{n}y)|^{p}\leq\|Q_{n}y\|^{p},\]
for $n\in\mathbb{N}$ and $y\in\mathcal{Y}$. Thus, $\|E_{n}\|\leq1$ for all $n\in\mathbb{N}$.
\item Since $E_{n}V_{n}e_{g}=e_{g}$ for all $g\in\mathcal{S}_{n}$, we have $E_{n}V_{n}\to I$ in SOT as $n\to\infty$.
\item It is easy to see that $V_{n}\psi(s)e_{g}=(\oplus_{k=1}^{\infty}\psi_{k}(s))V_{n}e_{g}$ for all group elements $s\in G$ and $g\in\mathcal{S}_{n}\cap s^{-1}\mathcal{S}_{n}$ and $n\in\mathbb{N}$. Hence, $V_{n}\psi(s)-(\oplus_{k=1}^{\infty}\psi_{k}(s))V_{n}\to 0$ in SOT, as $n\to\infty$, for all $s\in G$.
\item For all group elements $s\in G$ and $g\in\mathcal{S}_{n}\cap s\mathcal{S}_{n}$, point $y\in\mathcal{Y}$ and $n\in\mathbb{N}$,
\begin{eqnarray*}
e_{g}^{*}(E_{n}(\oplus_{k=1}^{\infty}\psi_{k}(s))y)&=&e_{gH_{n}}^{*}(Q_{n}(\oplus_{k=1}^{\infty}\psi_{k}(s))y)\\&=&
e_{gH_{n}}^{*}(\psi_{n}(s)Q_{n}y)\\&=&
e_{s^{-1}gH_{n}}^{*}(Q_{n}y)=e_{s^{-1}g}^{*}(E_{n}y)=e_{g}^{*}(\psi(s)E_{n}y).
\end{eqnarray*}
So $|E_{n}(\oplus_{k=1}^{\infty}\psi_{k}(s))-\psi(s)E_{n}|_{e_{g}^{*}}\to 0$, as $n\to\infty$, for all $s,g\in G$. Thus, $|E_{n}(\oplus_{k=1}^{\infty}\psi_{k}(s))-\psi(s)E_{n}|_{y^{*}}\to 0$, as $n\to\infty$, for all $s,g\in G$ and $y^{*}\in\mathcal{Y}^{*}$.
\item Clearly $V_{n}\to 0$ in WOT as $n\to\infty$.
\end{enumerate}
Thus, $\psi\ll\oplus_{k=1}^{\infty}\psi_{k}$. By Theorem \ref{lpVoic}, we have that $\oplus_{k=1}^{\infty}\psi_{k}$ is approximately similar to $(\oplus_{k=1}^{\infty}\psi_{k})\oplus\psi$. But by Lemma \ref{amenablesubembed} and Theorem \ref{lpVoic}, we have that $\psi$ is approximately similar to $\psi\oplus\psi_{1}\oplus\psi_{2}\oplus\ldots$. It follows that $\oplus_{k=1}^{\infty}\psi_{k}$ is approximately similar to $\psi$.
\end{proof}
\begin{remark}
A consequence of Corollary \ref{grouprepdecomp} is that if $G$ is a countable, residually finite, amenable group, then there exist finite rank idempotents $P_{1},P_{2},\ldots$ on $l^{p}(G)$ such that $P_{n}\to I$ in SOT and $\|P_{n}\psi(g)-\psi(g)P_{n}\|\to 0$, as $n\to\infty$, for every $g\in G$. An alternative way to prove this result is to use a technique of Orfanos \cite{Orfanos}. Moreover, using this technique, we can have $\|P_{n}\|=1$.
\end{remark}
The following result follows from Corollary \ref{grouprepdecomp} by taking $G=\mathbb{Z}$ and $H_{n}=n\mathbb{Z}$.
\begin{corollary}\label{bilateralcircular}
Let $B$ be the bilateral shift on $l^{p}(\mathbb{Z})$. For each $n\in\mathbb{N}$, let $B_{n}$ be the circular shift on $l^{p}(\mathbb{Z}/n\mathbb{Z})$, i.e., $B_{n}e_{j}=e_{j+1}$ for $j\in\mathbb{Z}/n\mathbb{Z}$. Then $B$ is approximately similar to $B_{1}\oplus B_{2}\oplus\ldots$.
\end{corollary}
\begin{lemma}[\cite{Fixman}]\label{Fixman}
Let $p\in[1,\infty]\backslash\{2\}$. Let $B$ be the bilateral shift on $l^{p}(\mathbb{Z})$. Then there exist Laurent polynomials $f_{1},f_{2},\ldots$ such that $\sup_{|v|=1}|f_{n}(v)|=1$, for all $n\in\mathbb{N}$, and $\|f_{n}(B)\|\to\infty$ as $n\to\infty$.
\end{lemma}
\begin{lemma}\label{UcalkinB}
Let $U$ be the unilateral shift on $l^{p}(\mathbb{N})$. Let $B$ be the bilateral shift on $l^{p}(\mathbb{Z})$. Let $f$ be a Laurent polynomial. Then
\[\|f(\pi(U))\|=\|f(\pi(B))\|=\|f(B)\|.\]
\end{lemma}
\begin{proof}
From the proof of Corollary \ref{UB}, there exist isometries $V_{n}:l^{p}(\mathbb{Z})\to l^{p}(\mathbb{N})$, for $n\in\mathbb{N}$, such that $V_{n}B-UV_{n}\to 0$ in SOT, as $n\to\infty$, and $V_{n}\to 0$ in WOT as $n\to\infty$. Let $U^{-1}$ be the backward shift on $l^{p}(\mathbb{N})$. Then $U^{-1}U=I$ and so $V_{n}B^{-1}-U^{-1}V_{n}=U^{-1}(UV_{n}-V_{n}B)B^{-1}\to 0$ in SOT as $n\to\infty$. Hence, for all $a_{-m},\ldots,a_{m}\in\mathbb{C}$,
\[V_{n}\left(\sum_{j=-m}^{j}a_{j}B^{j}\right)-\left(\sum_{j=-m}^{j}a_{j}U^{j}\right)V_{n}\to 0,\]
in SOT, as $n\to\infty$. Since $V_{n}$ are isometries, for all $x\in l^{p}(\mathbb{Z})$ and $K\in K(l^{p}(\mathbb{Z}))$,
\begin{eqnarray*}
\left\|\sum_{j=-m}^{j}a_{j}B^{j}x\right\|&=&\lim_{n\to\infty}\left\|V_{n}\left(\sum_{j=-m}^{j}a_{j}B^{j}\right)x\right\|\\&=&
\lim_{n\to\infty}\left\|\left(\sum_{j=-m}^{j}a_{j}U^{j}\right)V_{n}x\right\|\\&=&
\lim_{n\to\infty}\left\|\left(\sum_{j=-m}^{j}a_{j}U^{j}+K\right)V_{n}x\right\|,
\end{eqnarray*}
where we used the fact that $V_{n}\to 0$ in WOT, as $n\to\infty$, in the last equality. So
\[\left\|\sum_{j=-m}^{j}a_{j}B^{j}\right\|\leq\left\|\sum_{j=-m}^{j}a_{j}U^{j}+K\right\|,\]
for every $K\in K(l^{p}(\mathbb{Z}))$. Thus, $\|f(B)\|\leq\|f(\pi(U))\|$ for every Laurent polynomial $f$.

Note that $B$ is 1-similar to a rank one perturbation of $U^{-1}\oplus U$. So $\|f(\pi(U))\|\leq\|f(\pi(B))\|$. The inequality $\|f(\pi(B))\|\leq\|f(B)\|$ is trivial.
\end{proof}

\begin{lemma}\label{bilateralnondiag}
Let $p\in(1,\infty)\backslash\{2\}$. The bilateral shift $B$ on $l^{p}(\mathbb{Z})$ is not similar to a compact perturbation of a diagonal operator on $l^{p}$.
\end{lemma}
\begin{proof}
If $B$ is similar to a compact perturbation of a diagonal operator on $l^{p}$, then there exists $C>0$ such that $\|f(\pi(B))\|\leq C\sup_{|v|=1}|f(v)|$ for all Laurent polynomial $f$. By Lemma \ref{UcalkinB}, we have $\|f(\pi(B))\|=\|f(B)\|$. Thus, $\|f(B)\|\leq C\sup_{|v|=1}|f(v)|$. By Lemma \ref{Fixman}, an absurdity follows.
\end{proof}
Let $p\in(1,\infty)\backslash\{2\}$. Let $T$ be an invertible isometry on $l^{p}$. Then $T$ is the composition of a permutation operator and a diagonal operator with entries in the unit circle on $\mathbb{C}$ \cite{Lamperti}. Since every permutation on $\mathbb{N}$ is a product of disjoint finitary cyclic permutations and ``bilateral shift" permutations, $T$ is 1-similar to the direct sum of weighted circular shifts (if any) and weighted bilateral shifts (if any), where the weights are in the unit circle. A weighted bilateral shift is an operator on $l^{p}(\mathbb{Z})$ of the form $e_{j}\mapsto w_{j}e_{j+1}$, for $j\in\mathbb{Z}$. A weighted circular shift is an operator on $l^{p}(\mathbb{Z}/r\mathbb{Z})$ of the form $e_{j}\mapsto w_{j}e_{j+1}$, for $j\in\mathbb{Z}/r\mathbb{Z}$, for some $r\in\mathbb{N}$. Here the $w_{j}$ are called the weights.

Every weighted bilateral shift with weights in $\{v\in\mathbb{C}:|v|=1\}$ is 1-similar to the bilateral shift $B$ (by using conjugations by certain diagonal operators). Every weighted circular shift with weights $\{v\in\mathbb{C}:|v|=1\}$ is 1-similar to $vB_{n}$ for some $v\in\mathbb{C}$ with $|v|=1$ and $n\in\mathbb{N}$, where $B_{n}$ is the circular shift defined in Corollary \ref{bilateralcircular}. Thus every invertible isometry on $l^{p}$ is 1-similar to the direct sum of operators of the form $B$ or $vB_{n}$. For notational convenience, let $B_{\infty}=B$. We have
\begin{lemma}\label{invisomdecomposition}
Let $p\in[1,\infty]\backslash\{2\}$. If $T$ is an invertible isometry on $l^{p}$, then there exist countable collections $(n_{i})_{i\in\mathcal{C}}\subset\mathbb{N}\cup\{\infty\}$ and $(v_{i})_{i\in\mathcal{C}}\subset\{v\in\mathbb{C}:|v|=1\}$ such that $T$ is 1-similar to $\displaystyle\bigoplus_{i\in\mathcal{C}}v_{i}B_{n_{i}}$.
\end{lemma}
\begin{lemma}\label{BT}
Let $p\in(1,\infty)\backslash\{2\}$. Let $T_{0}$ be an invertible isometry on $l^{p}$. Let $B$ be the bilateral shift on $l^{p}(\mathbb{Z})$. Then $B$ is approximately similar to $B\oplus T_{0}$.
\end{lemma}
\begin{proof}
By Lemma \ref{invisomdecomposition}, there exist countable collections $(n_{i})_{i\in\mathcal{C}}\subset\mathbb{N}\cup\{\infty\}$ and $(v_{i})_{i\in\mathcal{C}}\subset\{v\in\mathbb{C}:|v|=1\}$ such that $T_{0}$ is 1-similar to $\displaystyle\bigoplus_{i\in\mathcal{C}}v_{i}B_{n_{i}}$. By Lemma \ref{amenablesubembed}, we have $B_{n}\stackrel{1}{\ll}B$, for all $n\in\mathbb{N}$, and also $B\stackrel{1}{\ll}B$. So $v_{i}B_{n_{i}}\stackrel{1}{\ll}v_{i}B$ for all $i\in\mathcal{C}$. But $vB$ is 1-similar to $B$ for every $v\in\mathbb{C}$ with $|v|=1$. Thus, $v_{i}B_{n_{i}}\stackrel{1}{\ll}B$. By Theorem \ref{lpVoic}, the result follows.
\end{proof}
Since the only invertible isometries on $l^{p}$, for $p\neq 2$, are compositions of permutation operators and diagonal operators, if $T_{1},T_{2}\in B(l^{p})$ are 1-similar and $T_{1}$ is a diagonal operator, then $T_{2}$ is also a diagonal operator. Here it is important that $T_{1},T_{2}$ are 1-similar rather than just similar, e.g., every circular shift $B_{n}$ is similar to a diagonal operator.
\begin{lemma}\label{invisomdiagbilateral}
Let $p\in(1,\infty)\backslash\{2\}$. Let $T$ be an invertible isometry on $l^{p}$. Let $B$ be the bilateral shift on $l^{p}(\mathbb{Z})$. If there exists $r\in\mathbb{N}$ such that $T^{r}$ is a diagonal operator, then $T$ is similar to a diagonal operator. If $T^{r}$ is not a diagonal operator for any $r\in\mathbb{N}$, then $T$ is approximately similar to $B$.
\end{lemma}
\begin{proof}
By Lemma \ref{invisomdecomposition}, there exist countable collections $(n_{i})_{i\in\mathcal{C}}\subset\mathbb{N}\cup\{\infty\}$ and $(v_{i})_{i\in\mathcal{C}}\subset\{v\in\mathbb{C}:|v|=1\}$ such that $T$ is 1-similar to $\displaystyle\bigoplus_{i\in\mathcal{C}}v_{i}B_{n_{i}}$. If there exists $r\in\mathbb{N}$ such that $T^{r}$ is a diagonal operator, then $\sup_{i\in\mathcal{C}}n_{i}<\infty$. Since $B_{n}$ is similar to a diagonal operator for every $n\in\mathbb{N}$, it follows that $T$ is similar to a diagonal operator.

Suppose that $T^{r}$ is not a diagonal operator for any $r\in\mathbb{N}$. Then $\sup_{i\in\mathcal{C}}n_{i}=\infty$. Thus, either (1) $n_{i_{0}}=\infty$ for some $i_{0}\in\mathcal{C}$ or (2) $n_{i}<\infty$ for all $i\in\mathcal{C}$ but $\sup_{i\in\mathcal{C}}n_{i}=\infty$.

In Case (1), $T$ is 1-similar to $B\oplus(\bigoplus_{i\neq i_{0}}v_{i}B_{n_{i}})$. Thus, taking $T_{0}=\bigoplus_{i\neq i_{0}}v_{i}B_{n_{i}}$in Lemma \ref{BT}, we have that $B$ is approximately similar to $B\oplus T_{0}$, which is 1-similar to $T$.

In Case (2), there exists an infinite subset $\mathcal{C}_{0}\subset\mathcal{C}$ such that the $n_{i}$, for $i\in\mathcal{C}_{0}$, are finite and distinct. So $T$ is 1-similar to $(\oplus_{i\in\mathcal{C}_{0}}v_{i}B_{n_{i}})\oplus(\oplus_{i\in\mathcal{C}\backslash\mathcal{C}_{0}}v_{i}B_{n_{i}})$. Replacing $\mathcal{C}_{0}$ by a smaller subset, we may assume that $v=\lim_{i\in\mathcal{C}_{0}}v_{i}$ exists, i.e., $v=\lim_{i\to\infty}v_{g(i)}$ for any bijection $g:\mathbb{N}\to\mathcal{C}_{0}$.

From the proof of Corollary \ref{grouprepdecomp}, we have $B\ll\oplus_{i\in\mathcal{C}_{0}}B_{i}$. So $vB\ll\oplus_{i\in\mathcal{C}_{0}}vB_{i}$. Note that $\oplus_{i\in\mathcal{C}_{0}}vB_{i}$ is a compact perturbation of $\oplus_{i\in\mathcal{C}_{0}}v_{i}B_{i}$. Thus, $vB\ll\oplus_{i\in\mathcal{C}_{0}}v_{i}B_{i}$. By Theorem \ref{lpVoicsingleoperator}, we have that $\oplus_{i\in\mathcal{C}_{0}}v_{i}B_{n_{i}}$ is approximately similar $(\oplus_{i\in\mathcal{C}_{0}}v_{i}B_{n_{i}})\oplus vB$. Therefore, $T$ is approximately similar to $T\oplus vB$, which is similar to $T\oplus B$. So by Lemma \ref{BT}, we have that $T$ is approximately similar to $B$.
\end{proof}
\begin{corollary}\label{characinvisom}
Let $p\in(1,\infty)\backslash\{2\}$. Let $T_{1}$ and $T_{2}$ be invertible isometries on $l^{p}$. Then $T_{1}$ and $T_{2}$ are approximately similar if and only if either
\begin{enumerate}[(1)]
\item $T_{1}^{r}$ and $T_{2}^{r}$ are not diagonal operators for any $r\in\mathbb{N}$, or
\item $T_{1}$ and $T_{2}$ have the same spectrum, $\text{dim ker}(T_{1}-\lambda)=\text{dim ker}(T_{2}-\lambda)$ for all isolated point $\lambda$ in the spectrum, and there exists $r\in\mathbb{N}$ such that $T_{1}^{r}$ and $T_{2}^{r}$ are diagonal operators.
\end{enumerate}
\end{corollary}
\begin{proof}
Suppose that $T_{1}$ and $T_{2}$ are approximately similar. By Lemma \ref{invisomdiagbilateral}, if there exists $r\in\mathbb{N}$ such that $T_{1}^{r}$ is a diagonal operator, then $T_{1}$ is similar to a diagonal operator. By Lemma \ref{bilateralnondiag}, we have that $T_{2}$ is not approximately similar to the bilateral shift $B$. By Lemma \ref{invisomdiagbilateral}, there exists $r\in\mathbb{N}$ such that $T_{2}^{r}$ is a diagonal operator. Thus, we have either (1) or (2). (In (2), if the $r$ in $T_{1}^{r}$ and the $r$ in $T_{2}^{r}$ are different, then we can replace them with their product.)

Conversely, if (1) is true then $T_{1}$ and $T_{2}$ are approximately similar (to $B$) by Lemma \ref{invisomdiagbilateral}. If (2) is true then by Lemma \ref{invisomdiagbilateral}, we have that $T_{1}$ is similar to a diagonal operator $D_{1}$ and $T_{2}$ is similar to a diagonal operator $D_{2}$. The operators $D_{1}$ and $D_{2}$ have the same spectrum and $\text{dim ker}(D_{1}-\lambda I)=\text{dim ker}(D_{2}-\lambda I)$ for all isolated point $\lambda$ in the spectrum. So $D_{1}$ and $D_{2}$ are approximately similar via conjugations by certain permutation operators. Thus the result follows.
\end{proof}
\section{Examples of extensions of $K(l^{p})$}
Let $\mathcal{A}$ be a separable unital Banach algebra. An {\it isomorphic extension} of $K(l^{p})$ by $\mathcal{A}$ is a unital homomorphism $\phi:\mathcal{A}\to B(l^{p})/K(l^{p})$ such that there exists $C\geq 1$ satisfying
\[\frac{1}{C}\|a\|\leq\|\phi(a)\|\leq C\|a\|,\]
for all $a\in\mathcal{A}$. Two isomorphic extensions $\phi_{1}:\mathcal{A}\to B(l^{p})/K(l^{p})$ and $\phi_{2}:\mathcal{A}\to B(l^{p})/K(l^{p})$ are {\it (strongly) equivalent} if there is an invertible operator $S$ on $l^{p}$ such that $\phi_{2}(a)=\pi(S)\phi_{1}(a)\pi(S)^{-1}$ for all $a\in\mathcal{A}$. For each isomorphic extension $\phi$ of $K(l^{p})$ by $\mathcal{A}$, let $[\phi]$ be the equivalence class of isomorphic extensions of $K(l^{p})$ by $\mathcal{A}$ containing $\phi$.

If $\phi_{1}:\mathcal{A}\to B(l^{p})/K(l^{p})$ and $\phi_{2}:\mathcal{A}\to B(l^{p})/K(l^{p})$ are isomorphic extensions, define an isomorphic extension $\phi_{1}\oplus\phi_{2}:\mathcal{A}\to B(l^{p}\oplus l^{p})/K(l^{p}\oplus l^{p})$ by
\[(\phi_{1}\oplus\phi_{2})(a)=\phi_{1}(a)\oplus\phi_{2}(a),\]
for $a\in\mathcal{A}$, where we define $\pi(T_{1})\oplus\pi(T_{2})=\pi(T_{1}\oplus T_{2})$ for $T_{1},T_{2}\in B(l^{p})$. Thus, we can define an operation $+$ on the set of all equivalence classes $[\phi]$ by $[\phi_{1}]+[\phi_{2}]=[\phi_{1}\oplus\phi_{2}]$. The set of all equivalence classes of isomorphic extensions of $K(l^{p})$ by $\mathcal{A}$ equipped with $+$ forms a commutative semigroup $\mathcal{EXT}_{\sim,s}(\mathcal{A},K(l^{p}))$.

An isomorphic extension $\phi:\mathcal{A}\to B(l^{p})/K(l^{p})$ is {\it trivial} if there is a unital homomorphism $\rho:\mathcal{A}\to B(l^{p})$ such that $\phi=\pi\circ\rho$. The set of all equivalence classes $[\phi]$ for trivial isomorphic extensions $\phi:\mathcal{A}\to B(l^{p})/K(l^{p})$ is a subsemigroup of $\mathcal{EXT}_{\sim,s}(\mathcal{A},K(l^{p}))$. The quotient of $\mathcal{EXT}_{\sim,s}(\mathcal{A},K(l^{p}))$ by this subsemigroup is denoted by $\mathrm{Ext}_{\sim,s}(\mathcal{A},K(l^{p}))$.
\begin{example}
Let $U$ and $B$ be the unilateral and bilateral shifts on $l^{p}$, respectively. Let $\mathcal{A}$ be the subalgebra of $B(l^{p})$ generated by $B$ and $B^{-1}$. By Lemma \ref{UcalkinB}, there is a unital homomorphism $\phi:\mathcal{A}\to B(l^{p})/K(l^{p})$ such that $\phi(B)=\pi(U)$. Moreover, $\|\phi(T)\|=\|T\|$ for all $T\in\mathcal{A}$. Thus, $\phi$ is an isomorphic extension. Since $\pi(U)$ has nontrivial Fredholm index, $\phi$ is a nontrivial isomorphic extension. Thus, the semigroup $\mathrm{Ext}_{\sim,s}(\mathcal{A},K(l^{p}))$ is nontrivial.
\end{example}
In this section, we show that
\begin{enumerate}[1.]
\item for every separable closed unital subalgebra $\mathcal{A}$ of $B(l^{2})/K(l^{2})$, there is an isomorphic extension $\phi:\mathcal{A}\to B(l^{p})/K(l^{p})$ such that $\phi(a)$ and $a$ have the same index for all invertible $a\in\mathcal{A}$ (Theorem \ref{isompreserving});
\item as a consequence, for all nonempty compact subset $M$ of $\mathbb{C}$ and numbers $n_{i}$ associated with each hole $O_{i}$ of $M$, there is an isomorphic extension $\phi:\mathcal{A}\to B(l^{p})/K(l^{p})$ such that $\phi(z-\lambda)$ has index $n_{i}$ for all $\lambda\in O_{i}$ where $z\in C(M)$ is the identity function (Corollary \ref{indexany}); and
\item if $p\neq 2$, there are two trivial isomorphic extensions $\phi_{1}:C[0,1]\to B(l^{p})/K(l^{p})$ and $\phi_{2}:C[0,1]\to B(l^{p})/K(l^{p})$ that are not equivalent (Theorem \ref{notequivtrivial}).
\end{enumerate}
Let $\mathcal{X}_{1},\mathcal{X}_{2},\ldots$ be Banach spaces. Let $\mathcal{V}$ be the vector space of all formal infinite matrix $(T_{i,j})_{i,j\in\mathbb{N}}$ such that
\begin{enumerate}[(1)]
\item $T_{i,j}$ is an operator from $\mathcal{X}_{j}$ into $\mathcal{X}_{i}$ for all $i,j\in\mathbb{N}$,
\item $\sup_{i,j\in\mathbb{N}}\|T_{i,j}\|<\infty$,
\item there exists $r\geq 0$ such that $T_{i,j}=0$ for all $|i-j|>r$.
\end{enumerate}
Let $(T_{i,j}^{(1)})_{i,j\in\mathbb{N}},(T_{i,j}^{(2)})_{i,j\in\mathbb{N}}\in\mathcal{V}$. Note that for all $i,k\in\mathbb{N}$, the infinite summation $\sum_{j\in\mathbb{N}}T_{i,j}^{(1)}T_{j,k}^{(2)}$ contains only finitely many nonzero terms. Moreover, the infinite matrix $(\sum_{j\in\mathbb{N}}T_{i,j}^{(1)}T_{j,k}^{(2)})_{i,k\in\mathbb{N}}$ is in $\mathcal{V}$. So matrix multiplication is a well defined operation on $\mathcal{V}$. Thus $\mathcal{V}$ becomes an algebra.

For each integer $r\geq 0$, let $\mathcal{V}_{r}$ be the set of all $(T_{i,j})_{i,j\in\mathbb{N}}$ in $\mathcal{V}$ such that $T_{i,j}=0$ for all $|i-j|>r$. Note that $\mathcal{V}=\cup_{r\geq 0}\mathcal{V}_{r}$.
\begin{lemma}\label{Vop}
Let $\mathcal{X}_{1},\mathcal{X}_{2},\ldots$ be finite dimensional Banach spaces. Let $\mathcal{X}^{(p)}=(\oplus_{n\in\mathbb{N}}\mathcal{X}_{n})_{l^{p}}$. Let $\pi_{p}:B(\mathcal{X}^{(p)})\to B(\mathcal{X}^{(p)})/K(\mathcal{X}^{(p)})$ be the quotient map. For each $k\in\mathbb{N}$, let $J_{k}^{(p)}:\mathcal{X}_{k}\to(\oplus_{n\in\mathbb{N}}\mathcal{X}_{n})_{l^{p}}$ and $Q_{k}^{(p)}:(\oplus_{n\in\mathbb{N}}\mathcal{X}_{n})_{l^{p}}\to\mathcal{X}_{k}$ be the canonical embedding and projection, respectively. Then the map $\Phi_{p}:\mathcal{V}\to B(\mathcal{X}^{(p)})$, defined by
\[\Phi_{p}[(T_{i,j})_{i,j\in\mathbb{N}}]=\sum_{i,j\in\mathbb{N}}J_{i}^{(p)}T_{i,j}Q_{j}^{(p)},\]
for $(T_{i,j})_{i,j\in\mathbb{N}}\in\mathcal{V}$, is a unital (algebra) homomorphism such that
\begin{enumerate}[(i)]
\item for all integer $r\in\mathbb{N}$ and $(T_{i,j})_{i,j\in\mathbb{N}}\in\mathcal{V}_{r}$,
    \[\limsup_{j\to\infty}\sup_{i\in\mathbb{N}}\|T_{i,j}\|\leq
    \|\pi_{p}(\Phi_{p}[(T_{i,j})_{i,j\in\mathbb{N}}])\|\leq(2r+1)\limsup_{j\to\infty}\sup_{i\in\mathbb{N}}\|T_{i,j}\|,\]
\item for every integer $r\in\mathbb{N}$, the set $\Phi_{p}(\mathcal{V}_{r})$ consists of all operators $T\in B(\mathcal{X}^{(p)})$ such that $Q_{i}^{(p)}TJ_{j}^{(p)}=0$ for all $|i-j|>r$.
\end{enumerate}
\end{lemma}
\begin{proof}
That the sum $\sum_{i,j\in\mathbb{N}}J_{i}^{(p)}T_{i,j}Q_{j}^{(p)}$ converges in SOT unconditionally and that (i) is satisfied follow from Lemma \ref{Multidiagest}. For $(T_{i,j}^{(1)})_{i,j\in\mathbb{N}},(T_{i,j}^{(2)})_{i,j\in\mathbb{N}}\in\mathcal{V}$ and $r,s\in\mathbb{N}$, we have
\begin{align*}
&J_{r}^{(p)}\Phi_{p}[(T_{i,j}^{(1)})_{i,j\in\mathbb{N}}]\Phi_{p}[(T_{i,j}^{(2)})_{i,j\in\mathbb{N}}]Q_{s}^{(p)}\\=&
\left(\sum_{j\in\mathbb{N}}J_{r}^{(p)}T_{r,j}^{(1)}Q_{j}^{(p)}\right)\left(\sum_{i\in\mathbb{N}}J_{i}^{(p)}T_{i,s}^{(2)}Q_{s}^{(p)}\right)\\=&
\sum_{j\in\mathbb{N}}J_{r}^{(p)}T_{r,j}^{(1)}T_{j,s}^{(2)}Q_{s}^{(p)}\\=&
J_{r}^{(p)}\Phi_{p}[(T_{i,j}^{(1)})_{i,j\in\mathbb{N}}(T_{i,j}^{(2)})_{i,j\in\mathbb{N}}]Q_{s}^{(p)}.
\end{align*}
Thus, $\Phi_{p}$ is a unital (algebra) homomorphism. Finally to prove (ii), note that if $T\in\Phi_{p}(\mathcal{V}_{r})$ then $Q_{i}^{(p)}TJ_{j}^{(p)}=0$ for all $|i-j|>r$. Conversely, if $Q_{i}^{(p)}TJ_{j}^{(p)}=0$ for all $|i-j|>r$, then setting $T_{i,j}=Q_{i}^{(p)}TJ_{j}^{(p)}$, for $i,j\in\mathbb{N}$, we have $T=\phi_{p}[(T_{i,j})_{i,j\in\mathbb{N}}]$. Thus, the result follows.
\end{proof}
For convenience, we repeat Lemma \ref{Vop} for $p=2$.
\begin{lemma}\label{Vop2}
Let $\mathcal{X}_{1},\mathcal{X}_{2},\ldots$ be finite dimensional Banach spaces. Let $\mathcal{X}^{(2)}=(\oplus_{n\in\mathbb{N}}\mathcal{X}_{n})_{l^{2}}$. Let $\pi_{2}:B(\mathcal{X}^{(2)})\to B(\mathcal{X}^{(2)})/K(\mathcal{X}^{(2)})$ be the quotient map. For each $k\in\mathbb{N}$, let $J_{k}^{(2)}:\mathcal{X}_{k}\to(\oplus_{n\in\mathbb{N}}\mathcal{X}_{n})_{l^{2}}$ and $Q_{k}^{(2)}:(\oplus_{n\in\mathbb{N}}\mathcal{X}_{n})_{l^{2}}\to\mathcal{X}_{k}$ be the canonical embedding and projection, respectively. Then the map $\Phi_{2}:\mathcal{V}\to B(\mathcal{X}^{(2)})$, defined by
\[\Phi_{2}[(T_{i,j})_{i,j\in\mathbb{N}}]=\sum_{i,j\in\mathbb{N}}J_{i}^{(2)}T_{i,j}Q_{j}^{(2)},\]
for $(T_{i,j})_{i,j\in\mathbb{N}}\in\mathcal{V}$, is a unital (algebra) homomorphism such that
\begin{enumerate}[(i)]
\item for all integer $r\in\mathbb{N}$ and $(T_{i,j})_{i,j\in\mathbb{N}}\in\mathcal{V}_{r}$,
    \[\limsup_{j\to\infty}\sup_{i\in\mathbb{N}}\|T_{i,j}\|\leq
    \|\pi_{2}(\Phi_{2}[(T_{i,j})_{i,j\in\mathbb{N}}])\|\leq(2r+1)\limsup_{j\to\infty}\sup_{i\in\mathbb{N}}\|T_{i,j}\|,\]
\item for every integer $r\in\mathbb{N}$, the set $\Phi_{2}(\mathcal{V}_{r})$ consists of all operators $T\in B(\mathcal{X}^{(2)})$ such that $Q_{i}^{(2)}TJ_{j}^{(2)}=0$ for all $|i-j|>r$.
\end{enumerate}
\end{lemma}

Note that the algebra $\mathcal{V}$ depends only on $\mathcal{X}_{1},\mathcal{X}_{2},\ldots$ but not on $p$ so the algebra $\mathcal{V}$ and the set $\mathcal{V}_{r}$ in Lemma \ref{Vop} and Lemma \ref{Vop2} are the same.

\begin{lemma}[\cite{Voiculescu}, Proof of Lemma 1.2]\label{triquasidiagonalp2}
Let $\mathcal{B}$ be a separable subalgebra of $B(l^{2})$. Then there are $0=m_{1}<m_{2}<\ldots$ such that $T-\sum_{k=1}^{\infty}(Q_{k-1}+Q_{k}+Q_{k+1})TQ_{k}$ is compact for every $T\in\mathcal{B}$, where $Q_{0}=0$ and $Q_{k}$ is the canonical projection from $l^{2}$ onto $l^{2}([m_{k}+1,m_{k+1}])$.
\end{lemma}
Let $0=m_{1}<m_{2}<\ldots$. Note that in Lemma \ref{Vop2}, if we take $\mathcal{X}_{k}=l^{2}([m_{k}+1,m_{k+1}])$ for $k\in\mathbb{N}$, then $\mathcal{X}^{(2)}=(\oplus_{n\in\mathbb{N}}l^{2}([m_{n}+1,m_{n+1}]))_{l^{2}}=l^{2}(\mathbb{N})$ and the projection $Q_{k}^{(2)}$ coincides with the projection $Q_{k}$ in Lemma \ref{triquasidiagonalp2}.
\begin{lemma}\label{separablesubalgebra}
Let $\mathcal{X}^{(p)}=(\oplus_{n\in\mathbb{N}}l^{2}([m_{n}+1,m_{n+1}]))_{l^{p}}$. Let $U$ be the unilateral shift on $l^{2}$. Let $\pi_{2}:B(l^{2})\to B(l^{2})/K(l^{2})$ be the quotient map. Let $\mathcal{A}$ be a separable closed unital subalgebra of $B(l^{2})/K(l^{2})$ containing $\pi_{2}(U)$. Then there exist $0=m_{1}<m_{2}<\ldots$ and a unital homomorphism $\phi:\mathcal{A}\to B(\mathcal{X}^{(p)})/K(\mathcal{X}^{(p)})$ such that
\begin{equation}\label{separablesubalgebraeq1}
\frac{1}{3}\|a\|\leq\|\phi(a)\|\leq 3\|a\|,
\end{equation}
for all $a\in\mathcal{A}$, and $\phi(\pi_{2}(U))$ has Fredholm index $-1$.
\end{lemma}
\begin{proof}
Let $\mathcal{B}=\pi_{2}^{-1}(\mathcal{A})$. By Lemma \ref{triquasidiagonalp2}, there are $0=m_{1}<m_{2}<\ldots$ such that $T-\sum_{k=1}^{\infty}(Q_{k-1}^{(2)}+Q_{k}^{(2)}+Q_{k+1}^{(2)})TQ_{k}^{(2)}$ is compact for every $T\in\mathcal{B}$, where $Q_{k}^{(2)}$ is the canonical projection from $l^{2}(\mathbb{N})$ onto $l^{2}([m_{k}+1,m_{k+1}])$ and $Q_{0}^{(2)}=0$. In Lemma \ref{Vop} and Lemma \ref{Vop2}, take $\mathcal{X}_{k}=l^{2}([m_{k}+1,m_{k+1}])$ for $k\in\mathbb{N}$. We have that the unital homomorphisms $\Phi_{2}:\mathcal{V}\to B(\mathcal{X}^{(2)})$ and $\Phi_{p}:\mathcal{V}\to B(\mathcal{X}^{(p)})$, defined by
\[\Phi_{p}[(T_{i,j})_{i,j\in\mathbb{N}}]=\sum_{i,j\in\mathbb{N}}J_{i}^{(p)}T_{i,j}Q_{j}^{(p)}\quad\text{and}\quad
\Phi_{2}[(T_{i,j})_{i,j\in\mathbb{N}}]=\sum_{i,j\in\mathbb{N}}J_{i}^{(2)}T_{i,j}Q_{j}^{(2)},\]
for $(T_{i,j})_{i,j\in\mathbb{N}}\in\mathcal{V}$, satisfy
\begin{enumerate}[(i)]
\item for all integer $r\in\mathbb{N}$ and $(T_{i,j})_{i,j\in\mathbb{N}}\in\mathcal{V}_{r}$,
\[\limsup_{j\to\infty}\sup_{i\in\mathbb{N}}\|T_{i,j}\|\leq\|\pi_{p}(\Phi_{p}[(T_{i,j})_{i,j\in\mathbb{N}}])\|\leq
(2r+1)\limsup_{j\to\infty}\sup_{i\in\mathbb{N}}\|T_{i,j}\|,\]
and
\[\limsup_{j\to\infty}\sup_{i\in\mathbb{N}}\|T_{i,j}\|\leq\|\pi_{2}(\Phi_{2}[(T_{i,j})_{i,j\in\mathbb{N}}])\|\leq
(2r+1)\limsup_{j\to\infty}\sup_{i\in\mathbb{N}}\|T_{i,j}\|,\]
\item $\Phi_{2}(\mathcal{V}_{1})$ consists of all operators $T\in B(\mathcal{X}^{(2)})$ such that $Q_{i}^{(2)}TJ_{j}^{(2)}=0$ for all $|i-j|>1$.
\end{enumerate}
So
\begin{equation}\label{boundp2}
\frac{1}{2r+1}\|\pi_{2}(\Phi_{2}[(T_{i,j})_{i,j\in\mathbb{N}}])\|\leq\|\pi_{p}(\Phi_{p}[(T_{i,j})_{i,j\in\mathbb{N}}])\|\leq
(2r+1)\|\pi_{2}(\Phi_{2}[(T_{i,j})_{i,j\in\mathbb{N}}])\|
\end{equation}
for all $(T_{i,j})_{i,j\in\mathbb{N}}\in\mathcal{V}_{r}$ and integer $r\geq 0$. Thus, the map $\phi_{0}:\pi_{2}(\Phi_{2}(\mathcal{V}))\to B(\mathcal{X}^{(p)})/K(\mathcal{X}^{(p)})$,
\[\phi_{0}(\pi_{2}(\Phi_{2}[(T_{i,j})_{i,j\in\mathbb{N}}])))=\pi_{p}(\Phi_{p}[(T_{i,j})_{i,j\in\mathbb{N}}]),\]
for $(T_{i,j})_{i,j\in\mathbb{N}}\in\mathcal{V}$, is a well defined unital homomorphism that is bounded on $\pi_{2}(\Phi_{2}(\mathcal{V}_{r}))$ for each $r\geq 0$.

For every $T\in B(l^{2})$,
\[Q_{i}^{(2)}\left(\sum_{k=1}^{\infty}(Q_{k-1}^{(2)}+Q_{k}^{(2)}+Q_{k+1}^{(2)})TQ_{k}^{(2)}\right)J_{j}^{(2)}=0\]
for all $|i-j|>1$. So by (ii), the operator $\sum_{k=1}^{\infty}(Q_{k-1}^{(2)}+Q_{k}^{(2)}+Q_{k+1}^{(2)})TQ_{k}^{(2)}$ is in $\Phi_{2}(\mathcal{V}_{1})$. Thus, from the beginning from this proof, every operator $T\in\mathcal{B}$ is the sum of a compact operator and an operator in $\Phi_{2}(\mathcal{V}_{1})$. Hence $\mathcal{A}\subset\pi_{2}(\Phi_{2}(\mathcal{V}_{1}))$.

Take $\phi$ to be the restriction of $\phi_{0}$ to $\mathcal{A}$. By (\ref{boundp2}), we have $\frac{1}{3}\|a\|\leq\|\phi(a)\|\leq 3\|a\|$ for all $a\in\mathcal{A}$. Thus, (\ref{separablesubalgebraeq1}) is proved. It remains to show that $\phi(\pi_{2}(U)))$ has Fredholm index $-1$.

Let $(e_{s})_{s\in\mathbb{N}}$ be the canonical basis for $\mathcal{X}^{(2)}=(\oplus_{n\in\mathbb{N}}l^{2}([m_{n}+1,m_{n+1}]))_{l^{2}}=l^{2}(\mathbb{N})$. Let $(x_{s})_{s\in\mathbb{N}}$ be the canonical basis for $\mathcal{X}^{(p)}=(\oplus_{n\in\mathbb{N}}l^{2}([m_{n}+1,m_{n+1}]))_{l^{p}}$, i.e., for each $n\in\mathbb{N}$, we have that $(x_{s})_{m_{n}+1\leq s\leq m_{n+1}}$ is the canonical basis for $l^{2}([m_{n}+1,m_{n+1}])$. Since $Ue_{s}=e_{s+1}$ for every $s\in\mathbb{N}$, we have that $Q_{i}^{(2)}UJ_{j}^{(2)}=0$ for all $|i-j|>1$. So by (ii), we have that $U\in\Phi_{2}(\mathcal{V}_{1})$. Thus there exists $(T_{i,j}^{(0)})_{i,j\in\mathbb{N}}\in\mathcal{V}_{1}$ such that $\Phi_{2}[(T_{i,j}^{(0)})_{i,j\in\mathbb{N}}]=U$. It is easy to see that if $U^{(p)}=\Phi_{p}[(T_{i,j}^{(0)})_{i,j\in\mathbb{N}}]$ then $U^{(p)}x_{s}=x_{s+1}$ for every $s\in\mathbb{N}$. Thus $\phi(\pi_{2}(U))=\phi_{0}(\pi_{2}(U))=\pi_{p}(\Phi_{p}[(T_{i,j}^{(0)})_{i,j\in\mathbb{N}}])=\pi_{p}(U^{(p)})$ has Fredholm index $-1$.
\end{proof}
\begin{lemma}[\cite{Conway}]\label{indexconnect}
For every $k\in\mathbb{Z}$, the set of all invertible $T\in B(l^{2})/K(l^{2})$ with Fredholm index $k$ is path connected.
\end{lemma}
\begin{lemma}[\cite{Murphy}]\label{indexopen}
Let $\mathcal{X}$ be a Banach space. For every $k\in\mathbb{Z}$, the set of all invertible $T\in B(\mathcal{X})/K(\mathcal{X})$ with Fredholm index $k$ is open.
\end{lemma}
\begin{lemma}\label{p2isom}
Let $\mathcal{X}^{(p)}=(\oplus_{n\in\mathbb{N}}l^{2}([m_{n}+1,m_{n+1}]))_{l^{p}}$. Let $\mathcal{A}$ be a separable closed unital subalgebra of $B(l^{2})/K(l^{2})$. Then there exist $0=m_{1}<m_{2}<\ldots$ and a unital homomorphism $\phi:\mathcal{A}\to B(\mathcal{X}^{(p)})/K(\mathcal{X}^{(p)})$such that
\begin{equation}\label{p2isomeq1}
\frac{1}{3}\|a\|\leq\|\phi(a)\|\leq 3\|a\|,
\end{equation}
for all $a\in\mathcal{A}$, and $\phi(a)$ and $a$ have the same Fredholm index for every $a\in\mathcal{A}$ that is invertible in $B(l^{2})/K(l^{2})$.
\end{lemma}
\begin{proof}
Let $\mathcal{C}$ be a countable dense subset of the set of all $a\in\mathcal{A}$ that is invertible in $B(l^{2})/K(l^{2})$. Let $U$ be the unilateral shift on $l^{2}$. Let $\pi_{2}:B(l^{2})\to B(l^{2})/K(l^{2})$ be the quotient map. By Lemma \ref{indexconnect}, for every $a\in\mathcal{C}$, there is a path $f_{a}:[0,1]\to B(l^{2})/K(l^{2})$ such that
\begin{enumerate}[(1)]
\item $f_{a}(0)=a$;
\item $f_{a}(1)=\pi_{2}(U^{-\text{ind }a})$; and
\item $f_{a}(t)$ is invertible in $B(l^{2})/K(l^{2})$ for all $t\in[0,1]$,
\end{enumerate}
where $\text{ind }a$ is the Fredholm index of $a$. Since $f_{a}$ is continuous, $\{f_{a}(t):t\in[0,1]\}$ is separable.

Let $\mathcal{A}_{1}$ be the closed subalgebra of $B(l^{2})/K(l^{2})$ generated by $\mathcal{A}$, $\pi_{2}(U)$, $\pi_{2}(U)^{-1}$ and $f_{a}(t)$ for $a\in\mathcal{C}$ and $t\in[0,1]$. Note that $\mathcal{A}_{1}$ is separable. By Lemma \ref{separablesubalgebra}, there exist $0=m_{1}<m_{2}<\ldots$ and a unital homomorphism $\phi_{1}$ from $\mathcal{A}_{1}$ into $B(\mathcal{X}^{(p)})/K(\mathcal{X}^{(p)})$ such that (\ref{p2isomeq1}) is satisfied and $\phi_{1}(\pi_{2}(U))$ has Fredholm index $-1$.

Let $a\in\mathcal{C}$. Then $\phi_{1}\circ f_{a}$ is a path in the set of all invertible elements of $B(\mathcal{X}^{(p)})/K(\mathcal{X}^{(p)})$. By Lemma \ref{indexopen}, the subset $\{t\in[0,1]:\text{ind }\phi_{1}(f_{a}(t))=\text{ind }\phi_{1}(f_{a}(0))\}$ of $[0,1]$ is closed and open in $[0,1]$. Thus $\phi_{1}(f_{a}(1))$ and $\phi_{1}(f_{a}(0))$ have the same Fredholm index. Since $\phi_{1}(f_{a}(0))=\phi_{1}(a)$ and
\[\phi_{1}(f_{a}(1))=\phi_{1}(\pi_{2}(U^{-\text{ind }a}))=\phi_{1}(\pi_{2}(U))^{-\text{ind }a},\]
it follows that $\phi_{1}(a)$ and $a$ have the same Fredholm index. Thus the Fredholm indices of $\phi_{1}(a)$ and $a$ coincide for all $a\in\mathcal{C}$. Since $\mathcal{C}$ is dense in the set of all $a\in\mathcal{A}$ that is invertible in $B(l^{2})/K(l^{2})$, the Fredholm indices of $\phi_{1}(a)$ and $a$ coincide for all $a\in\mathcal{A}$ that is invertible in $B(l^{2})/K(l^{2})$. The result follows by taking $\phi$ to be the restriction of $\phi_{1}$ to $\mathcal{A}$.
\end{proof}
Combining Lemmas \ref{p2isom} and \ref{isomorphicspace}, we obtain the following result.
\begin{theorem}\label{isompreserving}
Let $\mathcal{A}$ be a separable closed unital subalgebra of $B(l^{2})/K(l^{2})$. Then there exist a unital homomorphism $\phi:\mathcal{A}\to B(l^{p})/K(l^{p})$ and $C\geq 1$ such that
\[\frac{1}{C}\|a\|\leq\|\phi(a)\|\leq C\|a\|,\]
for all $a\in\mathcal{A}$, and $\phi(a)$ and $a$ have the same Fredholm index for every $a\in\mathcal{A}$ that is invertible in $B(l^{2})/K(l^{2})$.
\end{theorem}
\begin{corollary}
Let $\mathcal{A}$ be a separable closed unital subalgebra of $B(l^{2})/K(l^{2})$. Then there exists an isomorphic extension of $K(l^{p})$ by $\mathcal{A}$.
\end{corollary}
\begin{corollary}\label{indexany}
Let $M$ be a nonemepty compact subset of $\mathbb{C}$. Let $O_{1},O_{2},\ldots$ be the bounded connected components of $\mathbb{C}\backslash M$. For each $i\in\mathbb{N}$, let $n_{i}\in\mathbb{Z}$. Let $z\in C(M)$ be the identity function on $M$. Then there exists an isomorphic extension $\phi:C(M)\to B(l^{p})/K(l^{p})$ such that $\phi(z-\lambda)$ has Fredholm index $n_{i}$ for all $\lambda\in O_{i}$ and $i\in\mathbb{N}$.
\end{corollary}
\begin{proof}
Brown, Douglas and Fillmore \cite{Brown1} showed that there exists an isomorphic extension $\phi_{1}:C(M)\to B(l^{2})/K(l^{2})$ such that $\phi_{1}(z-\lambda)$ has Fredholm index $n_{i}$ for all $\lambda\in O_{i}$ and $i\in\mathbb{N}$. Let $\mathcal{A}$ be the range of $\phi_{1}$. By Theorem \ref{isompreserving}, there is an isomorphic extension $\phi_{2}:\mathcal{A}\to B(l^{p})/K(l^{p})$ such that $\phi_{2}(a)$ and $a$ have the same Fredholm index for every $a\in\mathcal{A}$ that is invertible in $B(l^{2})/K(l^{2})$. The result follows by taking $\phi=\phi_{2}\circ\phi_{1}$.
\end{proof}
\begin{corollary}\label{indexscalaress}
There exist a Fredholm operator $T\in B(l^{p})$ and $C\geq 1$ such that the Fredholm index of $T$ is $-1$ and
\[\|f(\pi(T))\|\leq C\sup_{|v|=1}|f(v)|\]
for every Laurent polynomial $f$.
\end{corollary}
\begin{proof}
By Corollary \ref{indexany}, there exists a unital homomorphism $\phi:C(S^{1})\to B(l^{p})/K(l^{p})$ such that $\phi(z)$ has Fredholm index $-1$. Note that
\[\|f(\phi(z))\|=\|\phi(f)\|\leq\|\phi\|\sup_{|v|=1}|f(v)|\]
for every Laurent polynomial $f$. The result follows by taking $T\in B(l^{p})$ such that $\pi(T)=\phi(z)$.
\end{proof}
\begin{remark}
An explicit operator satisfying the conclusion of Corollary \ref{indexscalaress} can be constructed as follows. Let $0=m_{1}<m_{2}<\ldots$ be such that $m_{s+1}-m_{s}\to\infty$ as $s\to\infty$. Let $(x_{s})_{s\in\mathbb{N}}$ be the canonical basis for $\mathcal{X}^{(p)}=(\oplus_{n\in\mathbb{N}}l^{2}([m_{n}+1,m_{n+1}]))_{l^{p}}$. From the proof of Lemma \ref{separablesubalgebra}, one can check that the operator $T_{0}$ on $\mathcal{X}^{(p)}$ defined by $T_{0}x_{s}=x_{s+1}$, for $s\in\mathbb{N}$, has Fredholm index $-1$ and
\[\|f(\pi(T_{0}))\|\leq3\sup_{|v|=1}|f(v)|\]
for every Laurent polynomial $f$. Let $S:\mathcal{X}^{(p)}\to l^{p}$ be an invertible operator which exists by Lemma \ref{isomorphicspace}. Then $ST_{0}S^{-1}$ satisfies the conclusion of Corollary \ref{indexscalaress}.
\end{remark}
Next we show that if $p\in(1,\infty)\backslash\{2\}$ then there are two trivial isomorphic extensions of $K(l^{p})$ by $C[0,1]$ that are not equivalent.
\begin{lemma}\label{Diagonalp21}
Let $T$ be a diagonal operator on $l^{2}([1,n])$ with distinct diagonal entries $u_{1},\ldots,u_{n}$. Let $D$ be a diagonal operator on $l^{p}$. Let $\beta\geq 1$. Let $L:l^{2}([1,n])\to l^{p}$ be an operator such that $\frac{1}{\beta}\|x\|\leq\|Lx\|\leq\beta\|x\|$ for all $x\in l^{2}([1,n])$. Let $\epsilon=\min_{i\neq j}|u_{i}-u_{j}|$. Then
\begin{equation}\label{TD1}
\frac{1}{\beta}\sqrt{n}\leq\frac{2}{\epsilon}\|DL-LT\|n+\beta n^{\frac{1}{p}}.
\end{equation}
and
\begin{equation}\label{TD2}
\frac{1}{\beta}n^{\frac{1}{p}}\leq\frac{4}{\epsilon}\|DL-LT\|n+\beta\sqrt{n}.
\end{equation}
\end{lemma}
\begin{proof}
Let $(e_{i}^{(2)})_{1\leq i\leq n}$ be the canonical basis for $l^{2}([1,n])$. Let $(e_{j}^{(p)})_{j\in\mathbb{N}}$ be the canonical basis for $l^{p}$. Let $v_{1},v_{2},\ldots$ be the diagonal entries of $D$. We have $Te_{i}^{(2)}=u_{i}e_{i}^{(2)}$, for all $1\leq i\leq n$, and $De_{j}^{(p)}=v_{j}e_{j}^{(p)}$ for all $j\in\mathbb{N}$. For each $1\leq i\leq n$, let $\mathcal{J}_{i}=\{j\in\mathbb{N}:|v_{j}-u_{i}|<\frac{\epsilon}{2}\}$ (which could be empty). Note that $\mathcal{J}_{1},\ldots,\mathcal{J}_{n}$ are disjoint. For each $1\leq i\leq n$, let $P_{i}$ be the canonical projection from $l^{p}$ onto $l^{p}(\mathcal{J}_{i})$. For all $1\leq i\leq n$ and $x=(x_{1},x_{2},\ldots)\in(I-P_{i})l^{p}=l^{p}(\mathbb{N}\backslash\mathcal{J}_{i})$,
\begin{eqnarray*}
\|(D-u_{i}I)x\|&=&\left(\sum_{j=1}^{\infty}|v_{j}-u_{i}|^{p}|x_{j}|^{p}\right)^{\frac{1}{p}}\\&\geq&
\left(\sum_{j\in\mathbb{N}\backslash\mathcal{J}_{i}}|v_{j}-u_{i}|^{p}|x_{j}|^{p}\right)^{\frac{1}{p}}\\&\geq&
\frac{\epsilon}{2}\left(\sum_{j\in\mathbb{N}\backslash\mathcal{J}_{i}}|x_{j}|^{p}\right)^{\frac{1}{p}}=\frac{\epsilon}{2}\|x\|.
\end{eqnarray*}
Hence $\|(D-u_{i}I)(I-P_{i})x\|\geq\frac{\epsilon}{2}\|(I-P_{i})x\|$ for all $1\leq i\leq n$ and $x\in l^{p}$. Thus,
\begin{eqnarray*}
\frac{\epsilon}{2}\|(I-P_{i})Le_{i}^{(2)}\|&\leq&\|(D-u_{i}I)(I-P_{i})Le_{i}^{(2)}\|\\&=&
\|(I-P_{i})(D-u_{i}I)Le_{i}^{(2)}\|\\&\leq&\|(D-u_{i}I)Le_{i}^{(2)}\|=\|(DL-LT)e_{i}^{(2)}\|\leq\|DL-LT\|,
\end{eqnarray*}
for all $1\leq i\leq n$, where the first equality follows from the fact that $D$ and $P_{i}$ commute. So
\begin{equation}\label{Lelocal}
\|(I-P_{i})Le_{i}^{(2)}\|\leq\frac{2}{\epsilon}\|DL-LT\|,
\end{equation}
for all $1\leq i\leq n$. Since $\mathcal{J}_{1},\ldots,\mathcal{J}_{n}$ are disjoint,
\begin{equation}\label{sumPLe}
\left\|\sum_{i=1}^{n}P_{i}Le_{i}^{(2)}\right\|=\left(\sum_{i=1}^{n}\|P_{i}Le_{i}^{(2)}\|^{p}\right)^{\frac{1}{p}}.
\end{equation}
So by (\ref{Lelocal}),
\begin{eqnarray*}
\left\|\sum_{i=1}^{n}Le_{i}^{(2)}\right\|&\leq&\sum_{i=1}^{n}\|(I-P_{i})Le_{i}^{(2)}\|+\left\|\sum_{i=1}^{n}P_{i}Le_{i}^{(2)}\right\|\\&\leq&
\frac{2}{\epsilon}\|DL-LT\|n+\left(\sum_{i=1}^{n}\|P_{i}Le_{i}^{(2)}\|^{p}\right)^{\frac{1}{p}}\\&\leq&
\frac{2}{\epsilon}\|DL-LT\|n+\left(\sum_{i=1}^{n}\|Le_{i}^{(2)}\|^{p}\right)^{\frac{1}{p}}\leq\frac{2}{\epsilon}\|DL-LT\|n+\beta n^{\frac{1}{p}}.
\end{eqnarray*}
Since
\[\left\|\sum_{i=1}^{n}Le_{i}^{(2)}\right\|\geq\frac{1}{\beta}\left\|\sum_{i=1}^{n}e_{i}^{(2)}\right\|=\frac{1}{\beta}\sqrt{n},\]
(\ref{TD1}) is proved.

By (\ref{Lelocal}),
\[\|P_{i}Le_{i}^{(2)}\|\geq\|Le_{i}^{(2)}\|-\frac{2}{\epsilon}\|DL-LT\|\geq\frac{1}{\beta}-\frac{2}{\epsilon}\|DL-LT\|,\]
for all $1\leq i\leq n$. So by (\ref{sumPLe}),
\[\left\|\sum_{i=1}^{n}P_{i}Le_{i}^{(2)}\right\|\geq\left(\frac{1}{\beta}-\frac{2}{\epsilon}\|DL-LT\|\right)n^{\frac{1}{p}}.\]
But by (\ref{Lelocal}),
\begin{eqnarray*}
\left\|\sum_{i=1}^{n}P_{i}Le_{i}^{(2)}\right\|&\leq&\sum_{i=1}^{n}\|(I-P_{i})Le_{i}^{(2)}\|+\left\|\sum_{i=1}^{n}Le_{i}^{(2)}\right\|\\&\leq&
\frac{2}{\epsilon}\|DL-LT\|n+\beta\left\|\sum_{i=1}^{n}e_{i}^{(2)}\right\|\\&=&\frac{2}{\epsilon}\|DL-LT\|n+\beta\sqrt{n}.
\end{eqnarray*}
Hence
\[\left(\frac{1}{\beta}-\frac{2}{\epsilon}\|DL-LT\|\right)n^{\frac{1}{p}}\leq\frac{2}{\epsilon}\|DL-LT\|n+\beta\sqrt{n}.\]
So
\[\frac{1}{\beta}n^{\frac{1}{p}}\leq\frac{2}{\epsilon}\|DL-LT\|(n+n^{\frac{1}{p}})+\beta\sqrt{n}.\]
Thus (\ref{TD2}) is proved.
\end{proof}
\begin{lemma}\label{Diagonalp22}
Let $p\in(1,\infty)\backslash\{2\}$. For each $n\in\mathbb{N}$, let $T_{n}$ be the diagonal operator on $l^{2}([1,n])$ with diagonal entries $\frac{1}{n},\ldots,\frac{n}{n}$. Let $(n_{k})_{k\in\mathbb{N}}$ be a sequence in $\mathbb{N}$ such that every natural number appears in the sequence $n_{1},n_{2},\ldots$ infinitely many times. Let $\mathcal{X}=(\oplus_{k\in\mathbb{N}}l^{2}([1,n_{k}]))_{l^{p}}$. Consider the operator $T=\oplus_{k\in\mathbb{N}}T_{n_{k}}$ on $\mathcal{X}$. Let $D$ be a diagonal operator on $l^{p}$. Then there is no operator $L:\mathcal{X}\to l^{p}$ such that $\inf\{\|Lx\|:\|x\|=1\}>0$ and $DL-LT$ is compact.
\end{lemma}
\begin{proof}
Suppose, for contradiction, that there is an operator $L:\mathcal{X}\to l^{p}$ such that $\inf\{\|Lx\|:\|x\|=1\}>0$ and $DL-LT$ is compact. Let $\beta\geq 1$ be such that $\frac{1}{\beta}\|x\|\leq\|Lx\|\leq\beta\|x\|$ for all $x\in\mathcal{X}$. For each $i\in\mathbb{N}$, let $J_{i}:l^{2}([1,n_{i}])\to(\oplus_{k\in\mathbb{N}}l^{2}([1,n_{k}]))_{l^{p}}$ be the canonical embedding. We have $\frac{1}{\beta}\|x\|\leq\|(LJ_{i})x\|\leq\beta\|x\|$ for all $x\in l^{2}([1,n_{i}])$. Since $DL-LT$ is compact, $\|D(LJ_{i})-(LJ_{i})T_{n_{i}}\|=\|(DL-LT)J_{i}\|\to 0$ as $i\to\infty$.

By Lemma \ref{Diagonalp21},
\[\frac{1}{\beta}\sqrt{n_{i}}\leq2\|D(LJ_{i})-(LJ_{i})T_{n_{i}}\|n_{i}^{2}+\beta n_{i}^{\frac{1}{p}},\]
and
\[\frac{1}{\beta}n_{i}^{\frac{1}{p}}\leq 4\|D(LJ_{i})-(LJ_{i})T_{n_{i}}\|n_{i}^{2}+\beta\sqrt{n_{i}},\]
for all $i\in\mathbb{N}$. Since $\|D(LJ_{i})-(LJ_{i})T_{n_{i}}\|\to 0$ as $i\to\infty$ and every natural number appears in $n_{1},n_{2},\ldots$ infinitely many times, it follows that
\[\frac{1}{\beta}\sqrt{n}\leq\beta n^{\frac{1}{p}},\]
and
\[\frac{1}{\beta}n^{\frac{1}{p}}\leq\beta\sqrt{n},\]
for all $n\in\mathbb{N}$. An absurdity follows since $p\neq 2$.
\end{proof}
If $T$ is an operator on a Banach space $\mathcal{X}_{0}$ and there is a constant $C\geq 1$ such that
\[\|f(T)\|\leq C\sup_{0\leq v\leq 1}|f(v)|,\]
for every polynomial $f$, then we can define a unital homomorphism $\rho:C[0,1]\to B(\mathcal{X}_{0})$ by setting $\rho(z)=T$ where $z\in C[0,1]$ is the identity function on $[0,1]$.
\begin{theorem}\label{notequivtrivial}
Let $p\in(1,\infty)\backslash\{2\}$. There exist trivial isomorphic extensions $\phi_{1},\phi_{2}$ of $K(l^{p})$ by $C[0,1]$ that are not equivalent. Moreover, $[\phi_{1}]\neq[\phi_{2}]+[\phi_{3}]$ for every isomorphic extension $\phi_{3}$ of $K(l^{p})$ by $C[0,1]$.
\end{theorem}
\begin{proof}
Let $T$ and $\mathcal{X}$ be as in Lemma \ref{Diagonalp22}. Let $D$ be a diagonal operator on $l^{p}$ whose entries are dense in $[0,1]$. Each of $T$ and $D$ has spectrum $[0,1]$ and
\[\|f(T)\|=\|f(T)\|_{e}=\|f(D)\|=\|f(D)\|_{e}=\sup_{0\leq v\leq 1}|f(v)|,\]
for every polynomial $f$. Let $S:\mathcal{X}\to l^{p}$ be an invertible operator which exists by Lemma \ref{isomorphicspace}. Define unital homomorphisms $\rho_{1},\rho_{2}$ from $C[0,1]$ into $B(l^{p})$ by setting $\rho_{1}(z)=STS^{-1}$ and $\rho_{2}(z)=D$. Take $\phi_{1}=\pi\circ\rho_{1}$ and $\phi_{2}=\pi\circ\rho_{2}$. By the conclusion of Lemma \ref{Diagonalp22}, we have that $T\oplus T_{0}$ is not similar to a compact perturbation of $D$ for any operator $T_{0}$ on any Banach space $\mathcal{X}_{0}$. So $\rho_{1}(z)\oplus T_{0}$ is not similar to a compact perturbation of $\rho_{2}(z)$ for any operator $T_{0}$ on any Banach space $\mathcal{X}_{0}$. By taking $\mathcal{X}_{0}=\{0\}$, we have that $\phi_{1}$ and $\phi_{2}$ are not equivalent. By taking $\mathcal{X}_{0}=l^{p}$ and $\pi(T_{0})=\phi_{3}(z)$, we have that $[\phi_{1}]\neq[\phi_{2}]+[\phi_{3}]$.
\end{proof}
\section{$\mathrm{Ext}_{\sim,s}(\mathcal{A},K(l^{p}))$ is a group for certain $\mathcal{A}$}
Let $\Lambda$ be a set. Let $\mathcal{X}$ and $\mathcal{Y}_{n}$, for $n\in\mathbb{N}$, be Banach spaces. Let $\psi:\Lambda\to B(\mathcal{X})$ and $\eta_{n}:\Lambda\to B(\mathcal{Y}_{n})$, for $n\in\mathbb{N}$, be maps. Let $\lambda\geq 1$. We write $\psi\stackrel{\lambda}\prec(\eta_{n})_{n\in\mathbb{N}}$ if there exist operators $V_{n}:l^{p}\to\mathcal{Y}_{n}$ and $E_{n}:\mathcal{Y}_{n}\to l^{p}$, for $n\in\mathbb{N}$, such that
\begin{enumerate}[(i)]
\item $\|V_{n}\|_{e}\leq\lambda$ for all $n\in\mathbb{N}$;
\item $\|E_{n}\|_{e}\leq1$ for all $n\in\mathbb{N}$;
\item $\lim_{n\to\infty}\|V_{n}\psi(\alpha)-\rho(\alpha)V_{n}\|_{e}=0$ for every $\alpha\in\Lambda$;
\item $\lim_{n\to\infty}\|E_{n}\rho(\alpha)-\psi(\alpha)E_{n}\|_{e}=0$ for every $\alpha\in\Lambda$; and
\item $\lim_{n\to\infty}\|E_{n}V_{n}-I\|_{e}=0$.
\end{enumerate}
Note that in (i) and (ii), we can have $\|V_{n}\|\leq\lambda$ and $\|E_{n}\|\leq 1$ by replacing $V_{n}$ by $(1-\frac{1}{n})V_{n}(I-P_{n})$ and $E_{n}$ by $(1-\frac{1}{n})(I-P_{n})E_{n}$ for some finite rank operator $P_{n}$ on $l^{p}$.
\begin{lemma}\label{psietaembed}
Let $\Lambda$ be a countable set. Let $\lambda\geq1$. Let $\psi:\Lambda\to B(l^{p})$. For $n\in\mathbb{N}$, let $\mathcal{Y}_{n}$ be a Banach space and let $\eta_{n}:\Lambda\to B(\mathcal{Y}_{n})$. Suppose that $\psi\stackrel{\lambda}\prec(\eta_{n})_{n\in\mathbb{N}}$. Then there exist operators $L:l^{p}\to(\mathcal{Y}_{1}\oplus\mathcal{Y}_{2}\oplus\ldots)_{l^{p}}$ and $R:(\mathcal{Y}_{1}\oplus\mathcal{Y}_{2}\oplus\ldots)_{l^{p}}\to l^{p}$ such that
\begin{enumerate}[(I)]
\item $\|L\|\leq2\lambda$ and $\|R\|\leq4$;
\item $L\psi(\alpha)-(\oplus_{n=1}^{\infty}\eta_{n}(\alpha))L$ is compact for all $\alpha\in\Lambda$;
\item $R(\oplus_{n=1}^{\infty}\eta_{n}(\alpha))-\psi(\alpha)R$ is compact for all $\alpha\in\Lambda$; and
\item $RL-I$ is compact.
\end{enumerate}
\end{lemma}
\begin{proof}
There exist operators $V_{n}:l^{p}\to\mathcal{Y}_{n}$ and $E_{n}:\mathcal{Y}_{n}\to l^{p}$, for $n\in\mathbb{N}$, such that $\|V_{n}\|\leq\lambda$ and $\|E_{n}\|\leq 1$, for all $n\in\mathbb{N}$, and (iii)-(v) are satisfied. Let $\Omega_{1}\subset\Omega_{2}\subset\ldots$ be finite subsets of $\Lambda$ such that $\cup_{n=1}^{\infty}\Omega_{n}=\Lambda$. By Lemma \ref{qcau}, there is a refined diagonal approximate identity $(A_{n})_{n\in\mathbb{N}}$ on $l^{p}$ such that
\begin{equation}\label{psietaembedeq1}
\|A_{n}\psi(\alpha)-\psi(\alpha)A_{n}\|\leq\frac{1}{2^{n}},
\end{equation}
for all $\alpha\in\Omega_{n}$ and $n\in\mathbb{N}$. Since $A_{n}\to 0$ in SOT, as $n\to\infty$, passing to a subsequence of $(A_{n})_{n\in\mathbb{N}}$, we have
\begin{enumerate}[(a)]
\item $\|V_{n}\|\leq\lambda$;
\item $\|E_{n}\|\leq 1$;
\item $\|(V_{n}\psi(\alpha)-\eta_{n}(\alpha)V_{n})(A_{n}-A_{n-1})\|\leq\frac{1}{2^{n}}$ for every $\alpha\in\Omega_{n}$;
\item $\|(A_{n+1}-A_{n-2})(E_{n}\eta_{n}(\alpha)-\psi(\alpha)E_{n})\|\leq\frac{1}{2^{n}}$ for every $\alpha\in\Omega_{n}$; and
\item $\|(E_{n}V_{n}-I)(A_{n}-A_{n-1})\|\leq\frac{1}{2^{n}}$,
\end{enumerate}
for all $n\geq 3$. Since $(A_{n})_{n\in\mathbb{N}}$ is a refined diagonal approximate identity, by Lemma \ref{qcausum} and (\ref{psietaembedeq1}),
\begin{enumerate}[(1)]
\item $\displaystyle\|A_{n}\psi(\alpha)-\psi(\alpha)A_{n}\|\leq\frac{1}{2^{n}}$ for all $\alpha\in\Omega_{n}$ and $n\in\mathbb{N}$;
\item \[\left(\sum_{n=1}^{\infty}\|(A_{n}-A_{n-1})x\|^{p}\right)^{\frac{1}{p}}\leq2\|x\|,\]
for all $x\in l^{p}$; and
\item  if $(x_{n})_{n\in\mathbb{N}}$ is a sequence in $l^{p}$ such that $\displaystyle\sum_{n=1}^{\infty}\|x_{n}\|^{p}<\infty$, then
\[\left(\left\|\sum_{n=1}^{\infty}(A_{n+1}-A_{n-2})x_{n}\right\|^{p}\right)^{\frac{1}{p}}\leq
4\left(\sum_{n=1}^{\infty}\|x_{n}\|^{p}\right)^{\frac{1}{p}}.\]
\end{enumerate}
For each $n\in\mathbb{N}$, let $J_{n}:\mathcal{Y}_{n}\to(\mathcal{Y}_{1}\oplus\mathcal{Y}_{2}\oplus\ldots)_{l^{p}}$ and $Q_{n}:(\mathcal{Y}_{1}\oplus\mathcal{Y}_{2}\oplus\ldots)_{l^{p}}\to\mathcal{Y}_{n}$ be the canonical embedding and projection, respectively. Take
\[Lx=\sum_{n=3}^{\infty}J_{n}V_{n}(A_{n}-A_{n-1})x,\]
for $x\in l^{p}$, and
\[Ry=\sum_{n=3}^{\infty}(A_{n+1}-A_{n-2})E_{n}Q_{n}y,\]
for $y\in(\mathcal{Y}_{1}\oplus\mathcal{Y}_{2}\oplus\ldots)_{l^{p}}$. For $x\in l^{p}$,
\begin{eqnarray*}
\|Lx\|&=&\left\|\sum_{n=3}^{\infty}J_{n}V_{n}(A_{n}-A_{n-1})x\right\|\\&=&
\left(\sum_{n=3}^{\infty}\|V_{n}(A_{n}-A_{n-1})x\|^{p}\right)^{\frac{1}{p}}\\&\leq&
\lambda\left(\sum_{n=3}^{\infty}\|(A_{n}-A_{n-1})x\|^{p}\right)^{\frac{1}{p}}\text{ by (a)}\\&\leq&
2\lambda\|x\|\text{ by (2)}.
\end{eqnarray*}
For $y\in(\mathcal{Y}_{1}\oplus\mathcal{Y}_{2}\oplus\ldots)_{l^{p}}$,
\begin{eqnarray*}
\|Ry\|&=&\left\|\sum_{n=3}^{\infty}(A_{n+1}-A_{n-2})E_{n}Q_{n}y\right\|\\&\leq&
4\left(\sum_{n=3}^{\infty}\|E_{n}Q_{n}y\|^{p}\right)^{\frac{1}{p}}\text{ by (3)}\\&\leq&
4\left(\sum_{n=3}^{\infty}\|Q_{n}y\|^{p}\right)^{\frac{1}{p}}\text{ by (b)}\\&\leq&
4\|y\|.
\end{eqnarray*}
Thus, (I) is proved.

Let $\alpha\in\Lambda$. Since $\Omega_{1}\subset\Omega_{2}\subset\ldots$ and $\cup_{n=1}^{\infty}\Omega_{n}=\Lambda$, there exists $n_{0}\geq 2$ such that $\alpha\in\Omega_{n}$ for all $n\geq n_{0}$. For $n\geq n_{0}+1$,
\begin{align*}
&\|V_{n}(A_{n}-A_{n-1})\psi(\alpha)-V_{n}\psi(\alpha)(A_{n}-A_{n-1})\|\\
\leq&\lambda\|(A_{n}-A_{n-1})\psi(\alpha)-\psi(\alpha)(A_{n}-A_{n-1})\|\text{ by (a)}\\
\leq&\lambda(\|A_{n}\psi(\alpha)-\psi(\alpha)A_{n}\|+\|A_{n-1}\psi(\alpha)-\psi(\alpha)A_{n-1}\|)\\
\leq&\lambda\left(\frac{1}{2^{n}}+\frac{1}{2^{n-1}}\right)\text{ by (1)}.
\end{align*}
By (c), we have $\|(V_{n}\psi(\alpha)-\eta_{n}(\alpha)V_{n})(A_{n}-A_{n-1})\|\leq\frac{1}{2^{n}}$ for $n\geq n_{0}$. Therefore,
\begin{align*}
&\|V_{n}(A_{n}-A_{n-1})\psi(\alpha)-\eta_{n}(\alpha)V_{n}(A_{n}-A_{n-1})\|\\\leq&
\lambda\left(\frac{1}{2^{n}}+\frac{1}{2^{n-1}}\right)+\frac{1}{2^{n}},
\end{align*}
for all $n\geq n_{0}+1$. Thus,
\begin{equation}\label{boundforL}
\sum_{n=3}^{\infty}\|V_{n}(A_{n}-A_{n-1})\psi(\alpha)-\eta_{n}(\alpha)V_{n}(A_{n}-A_{n-1})\|<\infty,
\end{equation}
for every $\alpha\in\Lambda$. For $x\in l^{p}$ and $\alpha\in\Lambda$,
\[L\psi(\alpha)x=\sum_{n=3}^{\infty}J_{n}V_{n}(A_{n}-A_{n-1})\psi(\alpha)x\]
and
\[(\oplus_{n=1}^{\infty}\eta_{n}(\alpha))Lx=\sum_{n=3}^{\infty}J_{n}\eta_{n}(\alpha)V_{n}(A_{n}-A_{n-1})x.\]
Since $V_{n}(A_{n}-A_{n-1})\psi(\alpha)-\eta_{n}(\alpha)V_{n}(A_{n}-A_{n-1})$ has finite rank for $n\geq 3$, by (\ref{boundforL}), $L\psi(\alpha)-(\oplus_{n=1}^{\infty}\eta_{n}(\alpha))L$ is compact for all $\alpha\in\Lambda$. So (II) is proved.

Let $\alpha\in\Lambda$. There exists $n_{0}\in\mathbb{N}$ such that $\alpha\in\Omega_{n}$ for all $n\geq n_{0}$. For $n\geq n_{0}+2$,
\begin{align*}
&\|(A_{n+1}-A_{n-2})\psi(\alpha)E_{n}-\psi(\alpha)(A_{n+1}-A_{n-2})E_{n}\|\\\leq&
\lambda\|(A_{n+1}-A_{n-2})\psi(\alpha)-\psi(\alpha)(A_{n+1}-A_{n-2})\|\text{ by (b)}\\\leq&
\lambda\left(\frac{1}{2^{n+1}}+\frac{1}{2^{n-2}}\right)\text{ by (1)}.
\end{align*}
By (d), we have $\|(A_{n+1}-A_{n-2})(E_{n}\eta_{n}(\alpha)-\psi(\alpha)E_{n})\|\leq\frac{1}{2^{n}}$ for all $n\geq n_{0}$. Therefore,
\begin{equation}\label{boundforR}
\sum_{n=3}^{\infty}\|(A_{n+1}-A_{n-2})E_{n}\eta_{n}(\alpha)-\psi(\alpha)(A_{n+1}-A_{n-2})E_{n}\|<\infty
\end{equation}
for every $\alpha\in\Lambda$. For $y\in(\mathcal{Y}_{1}\oplus\mathcal{Y}_{2}\oplus\ldots)_{l^{p}}$ and $\alpha\in\Lambda$,
\[R(\oplus_{n=1}^{\infty}\eta_{n}(\alpha))y=\sum_{n=3}^{\infty}(A_{n+1}-A_{n-2})E_{n}\eta_{n}(\alpha)Q_{n}y\]
and
\[\psi(\alpha)Ry=\sum_{n=3}^{\infty}\psi(\alpha)(A_{n+1}-A_{n-2})E_{n}Q_{n}y.\]
Therefore, $R\eta(\alpha)-\psi(\alpha)R$ is compact for all $\alpha\in\Lambda$. So (III) is proved.

For $x\in l^{p}$,
\[RLx=\sum_{n=3}^{\infty}(A_{n+1}-A_{n-2})E_{n}V_{n}(A_{n}-A_{n-1})x\]
and since $(A_{n})_{n\in\mathbb{N}}$ is a refined diagonal approximate unit on $l^{p}$,
\[x=\sum_{n=1}^{\infty}(A_{n}-A_{n-1})x=\sum_{n=1}^{\infty}(A_{n+1}-A_{n-2})(A_{n}-A_{n-1})x.\]
By (e), we have $\|(E_{n}V_{n}-I)(A_{n}-A_{n-1})\|\leq\frac{1}{2^{n}}$ for $n\geq 3$. Therefore, $RL-I$ is compact. So (IV) is proved.
\end{proof}
\begin{lemma}\label{psietaembed2}
Let $\Lambda$ be a countable set. Let $\psi:\Lambda\to B(l^{p})$ and $\eta:\Lambda\to B(l^{p})$. Let $\lambda\geq 1$. Suppose that there exist operators $L,R\in B(l^{p})$ such that
\begin{enumerate}[(I)]
\item $\|L\|\leq\lambda$ and $\|R\|\leq\lambda$;
\item $L\psi(\alpha)-\eta(\alpha)L$ is compact for all $\alpha\in\Lambda$;
\item $R\eta(\alpha)-\psi(\alpha)R$ is compact for all $\alpha\in\Lambda$; and
\item $RL-I$ is compact.
\end{enumerate}
Then there exist a Banach space $\mathcal{Y}$, an invertible operator $S:l^{p}\oplus\mathcal{Y}\to l^{p}\oplus l^{p}$ and a map $\zeta:\Lambda\to B(\mathcal{Y})$ such that
\begin{enumerate}[(i)]
\item $\|S\|\leq\lambda+2$ and $\|S^{-1}\|\leq2\lambda^{2}+6\lambda+5$; and
\item $\eta(\alpha)\oplus\eta(\alpha)-S(\psi(\alpha)\oplus\zeta(\alpha))S^{-1}$ is compact for all $\alpha\in\Lambda$.
\end{enumerate}
\end{lemma}
\begin{proof}
There exists $m\in\mathbb{N}$ such that $\|(RL-I)(I-P)\|\leq\frac{1}{2}$ where $P$ is the projection from $l^{p}$ onto $l^{p}([1,m])$. Define $L_{1}:l^{p}\to l^{p}\oplus l^{p}$ and $R_{1}:l^{p}\oplus l^{p}\to l^{p}$ by
\[L_{1}x=(L(I-P)x,Px)\quad\text{and}\quad R_{1}(y_{1},y_{2})=Ry_{1}+Py_{2},\]
for $x\in l^{p}$ and $y_{1},y_{2}\in l^{p}$. We have
\begin{enumerate}[(a)]
\item $\|L_{1}\|\leq\lambda+1$ and $\|R_{1}\|\leq\lambda+1$;
\item $L_{1}\psi(\alpha)-(\eta(\alpha)\oplus\eta(\alpha))L_{1}$ is compact for all $\alpha\in\Lambda$;
\item $R_{1}(\eta(\alpha)\oplus\eta(\alpha))-\psi(\alpha)R_{1}$ is compact for all $\alpha\in\Lambda$; and
\item $R_{1}L_{1}-I$ is compact and has norm at most $\frac{1}{2}$.
\end{enumerate}
Since $\|R_{1}L_{1}-I\|\leq\frac{1}{2}$, the operator $R_{1}L_{1}$ is invertible and $\|(R_{1}L_{1})^{-1}\|\leq 2$. Define $L_{2}:l^{p}\to l^{p}\oplus l^{p}$ and $R_{2}:l^{p}\oplus l^{p}\to l^{p}$ by $L_{2}=L_{1}$ and $R_{2}=(R_{1}L_{1})^{-1}R_{1}$. Then
\begin{enumerate}[(1)]
\item $\|L_{2}\|\leq\lambda+1$ and $\|R_{2}\|\leq2\lambda+2$;
\item $L_{2}\psi(\alpha)-(\eta(\alpha)\oplus\eta(\alpha))L_{2}$ is compact for all $\alpha\in\Lambda$;
\item $R_{2}(\eta(\alpha)\oplus\eta(\alpha))-\psi(\alpha)R_{2}$ is compact for all $\alpha\in\Lambda$; and
\item $R_{2}L_{2}=I$.
\end{enumerate}
The rest of this proof proceeds in the same way as the proof of Lemma \ref{embed2}. The operator $L_{2}R_{2}$ is an idempotent on $l^{p}\oplus l^{p}$. Take $\mathcal{Y}$ to be the range of $I-L_{2}R_{2}$. Take $S(x,y)=L_{2}x+y$ for $x\in l^{p}$ and $y\in\mathcal{Y}$. Then $S^{-1}z=(R_{2}z,(I-L_{2}R_{2})z)$ for $z\in l^{p}\oplus l^{p}$. We have $\|S\|\leq\lambda+2$ and $\|S^{-1}\|\leq(2\lambda+2)+(\lambda+1)(2\lambda+2)+1$. Take $\zeta(\alpha)=(I-L_{2}R_{2})(\eta(\alpha)\oplus\eta(\alpha))|_{\mathcal{Y}}$ for $\alpha\in\Lambda$. It is easy to check that $\eta(\alpha)\oplus\eta(\alpha)-S(\psi(\alpha)\oplus\zeta(\alpha))S^{-1}$ is compact for all $\alpha\in\Lambda$.
\end{proof}
\begin{lemma}\label{psietaembed3}
Let $\Lambda$ be a countable set. Let $\lambda\geq1$. Let $\psi:\Lambda\to B(l^{p})$. For $n\in\mathbb{N}$, let $\mathcal{Y}_{n}$ be a Banach space and let $\eta_{n}:\Lambda\to B(\mathcal{Y}_{n})$. Suppose that $\psi\stackrel{\lambda}{\prec}(\eta_{n})_{n\in\mathbb{N}}$ and $(\oplus_{n=1}^{\infty}\mathcal{Y}_{n})_{l^{p}}$ is $\lambda$-isomorphic to $l^{p}$. Then there exist a Banach space $\mathcal{Y}$ and a map $\zeta:\Lambda\to B(\mathcal{Y})$ such that $\oplus_{n=1}^{\infty}(\eta_{n}\oplus\eta_{n})$ is $(2500\lambda^{4})$-similar to $\psi\oplus\zeta$ modulo compact operators.
\end{lemma}
\begin{proof}
Let $W:(\oplus_{n=1}^{\infty}\mathcal{Y}_{n})_{l^{p}}\to l^{p}$ be an invertible operator such that $\|W\|\leq1$ and $\|W^{-1}\|\leq\lambda$. Let $L:l^{p}\to(\oplus_{n=1}^{\infty}\mathcal{Y}_{n})_{l^{p}}$ and $R:(\oplus_{n=1}^{\infty}\mathcal{Y}_{n})_{l^{p}}\to l^{p}$ be obtained by applying Lemma \ref{psietaembed}. Then
\begin{enumerate}[(I)]
\item $\|WL\|\leq2\lambda$ and $\|RW^{-1}\|\leq4\lambda$;
\item $WL\psi(\alpha)-(W(\oplus_{n=1}^{\infty}\eta_{n}(\alpha))W^{-1})WL$ is compact for all $\alpha\in\Lambda$;
\item $RW^{-1}(W(\oplus_{n=1}^{\infty}\eta_{n}(\alpha))W^{-1})-\psi(\alpha)RW^{-1}$ is compact for all $\alpha\in\Lambda$; and
\item $(RW^{-1})(WL)-I$ is compact.
\end{enumerate}
Taking $\eta(\alpha)=W(\oplus_{n=1}^{\infty}\eta_{n}(\alpha))W^{-1}$, for $\alpha\in\Lambda$, in Lemma \ref{psietaembed2}, we have that there exist a Banach space $\mathcal{Y}$ and a map $\zeta:\Lambda\to B(\mathcal{Y})$ such that $\eta\oplus\eta$ is $(2500\lambda^{3})$-similar to $\psi\oplus\zeta$ modulo compact operators. But $\eta\oplus\eta$ is $\lambda$-similar to $\oplus_{n=1}^{\infty}(\eta_{n}\oplus\eta_{n})$. Thus, the result follows.
\end{proof}
\begin{lemma}\label{psietalocalembed}
Let $d\in\mathbb{N}$. Let $M$ be a nonempty compact subset of $\mathbb{R}^{d}$. Let $\psi:C(M)\to B(l^{p})$ be a map such that $\pi\circ\psi:C(M)\to B(l^{p})/K(l^{p})$ is a unital homomorphism. Let $\Omega$ be a finite subset of $C(M)$. Let $M_{0}$ be a dense subset of $M$. Let $\epsilon>0$. There exist $k\in\mathbb{N}$, points $w_{1},\ldots,w_{k}\in M_{0}$, operators $V:l^{p}\to(l^{p})_{u}^{\oplus k}$ and $E:(l^{p})_{u}^{\oplus k}\to l^{p}$ such that
\begin{enumerate}[(i)]
\item $\|V\|_{e}\leq\|\pi\circ\psi\|$ and $\|E\|_{e}\leq 2^{d}\|\pi\circ\psi\|$;
\item $EV-I$ is compact;
\item $\|V\psi(h)-\eta(h)V\|_{e}\leq\epsilon$ for all $h\in\Omega$; and
\item $\|E\eta(h)-\psi(h)E\|_{e}\leq\epsilon$ for all $h\in\Omega$,
\end{enumerate}
where $\eta(h)=(h(w_{1})I\oplus\ldots\oplus h(w_{k})I)_{u}$ for $h\in C(M)$.
\end{lemma}
\begin{proof}
Let $\gamma>0$ be such that
\begin{equation}\label{unifcontin}
|h(v)-h(w)|\leq\frac{\epsilon}{2^{d}},
\end{equation}
for all $v,w\in M$ with distance at most $\gamma$ and $h\in\Omega$. By Lemma \ref{partitionunity}, there exist a partition of unity $(f_{i})_{1\leq i\leq k}$ on $M$ and continuous functions $g_{i}:M\to[0,1]$, for $1\leq i\leq k$, such that
\begin{enumerate}[(1)]
\item the diameter of the support of $g_{i}$ is at most $\gamma$ for every $1\leq i\leq k$,
\item every $v\in M$ is contained in at most $2^{d}$ of $\mathrm{supp}\,g_{1},\ldots,\mathrm{supp}\,g_{k}$,
\item $g_{i}=1$ on the support of $f_{i}$ for every $1\leq i\leq k$.
\end{enumerate}
Without loss of generality, we may assume that each $f_{i}$ is not a zero function. Thus, $\mathrm{supp}\,g_{i}$ has nonempty interior. Since $M_{0}$ is dense, there exists $w_{i}\in M_{0}\cap\mathrm{supp}\,g_{i}$.

Take
\[Vx=(\psi(f_{1})x,\ldots,\psi(f_{k})x),\]
for all $x\in l^{p}$, and
\[E(y_{1},\ldots,y_{k})=\sum_{i=1}^{k}\psi(g_{i})y_{i},\]
for all $(y_{1},\ldots,y_{k})\in(l^{p})_{u}^{\oplus k}$.

Since $(f_{i})_{1\leq i\leq k}$ is a partition of unity on $M$, we have $\|\sum_{i=1}^{k}\delta_{i}f_{i}\|\leq 1$ for all $\delta\in\{-1,1\}^{k}$. So
\[\left\|\sum_{i=1}^{k}\delta_{i}\psi(f_{i})\right\|_{e}=\left\|\sum_{i=1}^{k}\delta_{i}\pi\circ\psi(f_{i})\right\|\leq\|\pi\circ\psi\|,\]
for all $\delta\in\{-1,1\}^{k}$. So by Lemma \ref{unconditional3}, we have $\|V\|_{e}\leq\|\pi\circ\psi\|$.

By (2), we have $\|\sum_{i=1}^{k}\delta_{i}g_{i}\|\leq 2^{d}$ for all $\delta\in\{-1,1\}^{k}$. Thus,
\[\left\|\sum_{i=1}^{k}\delta_{i}\psi(g_{i})\right\|_{e}=\left\|\sum_{i=1}^{k}\delta_{i}\pi\circ\psi(g_{i})\right\|\leq 2^{d}\|\pi\circ\psi\|,\]
for all $\delta\in\{-1,1\}^{k}$. So by Lemma \ref{unconditional3}, we have $\|E\|_{e}\leq 2^{d}\|\pi\circ\psi\|$. Thus, (i) is proved.

By (3), we have $g_{i}f_{i}=f_{i}$ for all $1\leq i\leq k$. Since $\displaystyle EV=\sum_{i=1}^{k}\psi(g_{i})\psi(f_{i})$, it follows that
\[\pi(EV)=\sum_{i=1}^{k}\pi\circ\psi(g_{i}f_{i})=\sum_{i=1}^{k}\pi\circ\psi(f_{i})=1.\]
Hence, (ii) is proved.

Since $w_{i}$ is in the support of $g_{i}$, by (1) and (\ref{unifcontin}), we have $|h(v)-h(w_{i})|\leq\frac{\epsilon}{2^{d}}$ for all $v$ in $\mathrm{supp}\,g_{i}$ and $h\in\Omega$. By (2), we have $\sum_{i=1}^{k}g_{i}\leq 2^{d}$ and so
\[\sum_{i=1}^{k}|h(v)-h(w_{i})|g_{i}(v)\leq\sum_{i=1}^{k}\frac{\epsilon}{2^{d}}g_{i}(v)\leq\epsilon.\]
Since $0\leq f_{i}\leq g_{i}$, we also have
\[\sum_{i=1}^{k}|h(v)-h(w_{i})|f_{i}(v)\leq\epsilon.\]
Therefore,
\[\left|\sum_{i=1}^{k}\delta_{i}(h(v)-h(w_{i}))g_{i}(v)\right|\leq\epsilon\quad\text{and}\quad
\left|\sum_{i=1}^{k}\delta_{i}(h(v)-h(w_{i}))f_{i}(v)\right|\leq\epsilon\]
for all $v\in M$. So
\begin{equation}\label{psietalocalembedeq1}
\left\|\sum_{i=1}^{k}\delta_{i}(h-h(w_{i}))g_{i}\right\|\leq\epsilon\quad\text{and}\quad
\left\|\sum_{i=1}^{k}\delta_{i}(h-h(w_{i}))f_{i}\right\|\leq\epsilon,
\end{equation}
for all $h\in\Omega$. Note that
\[V\psi(h)x-\eta(h)Vx=(\psi(f_{1})(\psi(h)x-h(w_{1}))x,\ldots,\psi(f_{k})(\psi(h)x-h(w_{k}))x),\]
for all $x\in l^{p}$, and
\[(E\eta(h)-\psi(h)E)(y_{1},\ldots,y_{k})=\sum_{i=1}^{k}(h(w_{i})-\psi(h))\psi(g_{i})y_{i},\]
for all $y_{1},\ldots,y_{k}\in l^{p}$. By (\ref{psietalocalembedeq1}) and Lemma \ref{unconditional3}, we obtain (iii) and (iv).
\end{proof}
Modulo some technicalities, the following result follows by combining Lemmas \ref{psietaembed3} and \ref{psietalocalembed}.
\begin{theorem}\label{main21}
Let $d\in\mathbb{N}$. Let $M$ be a nonempty compact subset of $\mathbb{R}^{d}$. Then $\mathrm{Ext}_{\sim,s}(C(M),K(l^{p}))$ is a group.
\end{theorem}
\begin{proof}
It suffices to show that for every isomorphic extension $\phi:C(M)\to B(l^{p})/K(l^{p})$, there exists an isomorphic extension $\phi^{(-1)}:C(M)\to B(l^{p})/K(l^{p})$ such that $\phi\oplus\phi^{(-1)}:C(M)\to B(l^{p}\oplus l^{p})/K(l^{p}\oplus l^{p})$ is a trivial isomorphic extension. Let $\psi:C(M)\to B(l^{p})/K(l^{p})$ be any map such that $\pi\circ\psi=\phi$. Let $\Lambda$ be a countable dense subset of $C(M)$. By Lemmas \ref{psietalocalembed} and \ref{normleq2}, there exist $k_{1},k_{2},\ldots\in\mathbb{N}$ and unital homomorphisms $\eta_{n}:C(M)\to B((l^{p})_{u}^{\oplus k_{n}})/K((l^{p})_{u}^{\oplus k_{n}})$, for $n\in\mathbb{N}$, such that $\psi|_{\Lambda}\stackrel{\lambda}{\prec}(\eta_{n}|_{\Lambda})_{n\in\mathbb{N}}$ where $\lambda=2^{d}\|\phi\|^{2}$. By Lemma \ref{isomorphicspaceu}, we have that $(\oplus_{n=1}^{\infty}(l^{p})_{u}^{\oplus k_{n}})_{l^{p}}$ is isomorphic to $l^{p}$. By Lemma \ref{psietaembed3}, there exist a Banach space $\mathcal{Y}$ and a map $\zeta:\Lambda\to B(\mathcal{Y})$ such that $\oplus_{n=1}^{\infty}(\eta_{n}|_{\Lambda}\oplus\eta_{n}|_{\Lambda})$ is similar to $(\psi|_{\Lambda})\oplus\zeta$ modulo compact operators. Since $\oplus_{n=1}^{\infty}(\eta_{n}\oplus\eta_{n})$ is a unital homomorphism and $\Lambda$ is dense in $C(M)$, it follows that we can extend $\zeta$ to $\widetilde{\zeta}:C(M)\to B(\mathcal{Y})$ so that $\oplus_{n=1}^{\infty}(\eta_{n}\oplus\eta_{n})$ is similar to $\psi\oplus\widetilde{\zeta}$ modulo compact operators. In particular, $\pi\circ\widetilde{\zeta}$ is a unital homomorphism.

Let $\zeta_{0}:C(M)\to B(l^{p})$ be defined by sending $h\in C(M)$ to the diagonal operator on $l^{p}$ with entries $h(w_{1}),h(w_{2}),\ldots$ where $(w_{n})_{n\in\mathbb{N}}$ is a fixed sequence that is dense in $M$ and each $w_{n}$ occurs infinitely many times. Then $\oplus_{n=1}^{\infty}(\eta_{n}\oplus\eta_{n})\oplus\zeta_{0}$ is similar to $\psi\oplus\widetilde{\zeta}\oplus\zeta_{0}$ modulo compact operators.

Note that $\oplus_{n=1}^{\infty}(\eta_{n}\oplus\eta_{n})\oplus\zeta_{0}$ is a unital homomorphism and so $\pi\circ(\psi\oplus\widetilde{\zeta}\oplus\zeta_{0})$ is a trivial isomorphic extension. Since $\mathcal{Y}$ is isomorphic to the range of an idempotent on $l^{p}$, by Lemma \ref{prime}, the direct sum $\mathcal{Y}\oplus l^{p}$ is isomorphic to $l^{p}$. Let $S:\mathcal{Y}\oplus l^{p}\to l^{p}$ be any invertible operator. Define $\psi^{(-1)}:C(M)\to B(l^{p})$ by $\psi^{(-1)}(h)=S(\widetilde{\zeta}(h)\oplus\zeta_{0}(h))S^{-1}$ for $h\in C(M)$. Then $\pi\circ\psi^{(-1)}:C(M)\to B(l^{p})/K(l^{p})$ is an isomorphic extension. Since $\pi\circ(\psi\oplus\widetilde{\zeta}\oplus\zeta_{0})$ is a trivial isomorphic extension, $(\pi\circ\psi)\oplus(\pi\circ\psi^{(-1)})$ is a trivial isomorphic extension. The result follows.
\end{proof}
\begin{lemma}\label{varphithetaembed}
Let $G$ be a countable amenable group. Let $\mathcal{X}=(\oplus_{g\in G}l^{p})_{l^{p}}$. Define $\theta:G\to B(\mathcal{X})$ by
\[\theta(s)[(y_{g})_{g\in G}]=(y_{s^{-1}g})_{g\in G},\]
for $s\in S$ and $(y_{g})_{g\in G}\in\mathcal{X}$. Let $\psi:G\to B(l^{p})$ be a map such that $\pi\circ\psi$ is a unital (group) homomorphism and $\sup_{s\in G}\|\psi(s)\|<\infty$. Let $F_{0}$ be a finite subset of $G$. Let $\epsilon>0$. Then there exist operators $V,E\in B(l^{p})$ such that
\begin{enumerate}[(i)]
\item $\|V\|_{e}\leq\sup_{s\in G}\|\psi(s)\|$ and $\|E\|\leq\sup_{s\in G}\|\psi(s)\|$;
\item $EV-I$ is compact;
\item $\|V\psi(s)-\theta(s)V\|_{e}\leq\epsilon$ for all $s\in F_{0}$; and
\item $\|E\theta(s)-\psi(s)E\|_{e}\leq\epsilon$ for all $s\in F_{0}$.
\end{enumerate}
\end{lemma}
\begin{proof}
For each $s\in G$, let $J_{s}:l^{p}\to\mathcal{X}$ be the canonical embedding that maps $l^{p}$ onto component $s$ of $\mathcal{X}=(\oplus_{g\in G}l^{p})_{l^{p}}$ and let $Q_{s}:\mathcal{X}\to l^{p}$ be the canonical projection from $\mathcal{X}=(\oplus_{g\in G}l^{p})_{l^{p}}$ onto its component $s$. Since $G$ is amenable, there is a nonempty finite subset $F$ of $G$ such that
\begin{equation}\label{varphithetaembedeq1}
\frac{|s_{0}F\Delta F|}{|F|}\leq\epsilon,
\end{equation}
for all $s_{0}\in F_{0}$. Let $\frac{1}{p}+\frac{1}{q}=1$. Take
\[Vx=\frac{1}{|F|^{\frac{1}{p}}}\sum_{s\in F}J_{s}\psi(s^{-1})x,\]
for $x\in l^{p}$, and
\[Ey=\frac{1}{|F|^{\frac{1}{q}}}\sum_{s\in F}\psi(s)Q_{s}y,\]
for $y\in\mathcal{X}$. We have
\begin{eqnarray*}
\|Vx\|&=&\left\|\frac{1}{|F|^{\frac{1}{p}}}\sum_{s\in F}J_{s}\psi(s^{-1})x\right\|\\&=&
\frac{1}{|F|^{\frac{1}{p}}}\left(\sum_{s\in F}\|\psi(s^{-1})x\|^{p}\right)^{\frac{1}{p}}\leq\sup_{s\in G}\|\psi(s)\|\|x\|,
\end{eqnarray*}
for $x\in l^{p}$, and
\begin{eqnarray*}
\|Ey\|&=&\left\|\frac{1}{|F|^{\frac{1}{q}}}\sum_{s\in F}\psi(s)Q_{s}y\right\|\\&=&
\sup_{s\in G}\|\psi(s)\|\frac{1}{|F|^{\frac{1}{q}}}\sum_{s\in F}\|Q_{s}y\|\\&\leq&
\sup_{s\in G}\|\psi(s)\|\left(\sum_{s\in F}\|Q_{s}y\|^{p}\right)^{\frac{1}{p}}\leq\sup_{s\in G}\|\psi(s)\|\|y\|,
\end{eqnarray*}
where the first inequality follows from H\"older's inequality. Thus, (i) is proved.

We have
\begin{eqnarray*}
EV&=&\frac{1}{|F|^{\frac{1}{q}}}\sum_{s\in F}\psi(s)Q_{s}\left(\frac{1}{|F|^{\frac{1}{q}}}\sum_{g\in F}J_{g}\psi(g^{-1})\right)\\&=&
\frac{1}{|F|^{\frac{1}{q}}}\sum_{s\in F}\psi(s)Q_{s}\left(\frac{1}{|F|^{\frac{1}{q}}}J_{s}\psi(s^{-1})\right)\\&=&
\frac{1}{|F|}\sum_{s\in F}\psi(s)\psi(s^{-1}).
\end{eqnarray*}
Since $\pi\circ\psi$ is a unital homomorphism, $EV-I$ is compact. So (ii) is proved.

For $s_{0}\in G$, modulo compact operators, we have
\begin{eqnarray*}
V\psi(s_{0})&=&\frac{1}{|F|^{\frac{1}{p}}}\sum_{s\in F}J_{s}\psi(s^{-1})\psi(s_{0})\\&=&
\frac{1}{|F|^{\frac{1}{p}}}\sum_{s\in F}J_{s}\psi(s^{-1}s_{0})\\&=&
\frac{1}{|F|^{\frac{1}{p}}}\sum_{s\in F}J_{s}\psi((s_{0}^{-1}s)^{-1})\\&=&
\frac{1}{|F|^{\frac{1}{p}}}\sum_{s_{0}s\in F}J_{s_{0}s}\psi(s^{-1})\\&=&
\frac{1}{|F|^{\frac{1}{p}}}\sum_{s\in s_{0}^{-1}F}\theta(s_{0})J_{s}\psi(s^{-1}).
\end{eqnarray*}
Thus,
\begin{eqnarray*}
\|V\psi(s_{0})-\theta(s_{0})V\|_{e}&=&\left\|\frac{1}{|F|^{\frac{1}{p}}}\sum_{s\in s_{0}^{-1}F\Delta F}\theta(s_{0})J_{s}\psi(s^{-1})\right\|_{e}\\&=&
\left\|\frac{1}{|F|^{\frac{1}{p}}}\sum_{s\in s_{0}^{-1}F\Delta F}J_{s}\psi(s^{-1})\right\|_{e}.
\end{eqnarray*}
Since
\begin{eqnarray*}
\left\|\frac{1}{|F|^{\frac{1}{p}}}\sum_{s\in s_{0}^{-1}F\Delta F}J_{s}\psi(s^{-1})x\right\|&=&\frac{1}{|F|^{\frac{1}{p}}}
\left(\sum_{s\in s_{0}^{-1}F\Delta F}\|\psi(s^{-1})x\|^{p}\right)^{\frac{1}{p}}\\&\leq&
\left(\frac{|s_{0}^{-1}F\Delta F|}{|F|}\right)^{\frac{1}{p}}\sup_{s\in G}\|\psi(s)\|\|x\|,
\end{eqnarray*}
for all $x\in l^{p}$, it follows that
\[\|V\psi(s_{0})-\theta(s_{0})V\|_{e}\leq\left(\frac{|s_{0}^{-1}F\Delta F|}{|F|}\right)^{\frac{1}{p}}\sup_{s\in G}\|\psi(s)\|.\]
Thus, by (\ref{varphithetaembedeq1}), we have
\[\|V\psi(s_{0})-\theta(s_{0})V\|_{e}\leq\epsilon^{\frac{1}{p}}\sup_{s\in G}\|\psi(s)\|,\]
for all $s_{0}\in F_{0}$. Thus, (iii) is proved.

For $s_{0}\in G$, modulo compact operators, we have
\begin{eqnarray*}
\psi(s_{0})E&=&\frac{1}{|F|^{\frac{1}{q}}}\sum_{s\in F}\psi(s_{0})\psi(s)Q_{s}\\&=&
\frac{1}{|F|^{\frac{1}{q}}}\sum_{s\in F}\psi(s_{0}s)Q_{s}\\&=&
\frac{1}{|F|^{\frac{1}{q}}}\sum_{s_{0}^{-1}s\in F}\psi(s)Q_{s_{0}^{-1}s}\\&=&
\frac{1}{|F|^{\frac{1}{q}}}\sum_{s\in s_{0}F}\psi(s)Q_{s}\theta(s_{0}).
\end{eqnarray*}
Thus,
\begin{eqnarray*}
\|\psi(s_{0})E-E\theta(s_{0})\|_{e}&=&\left\|\frac{1}{|F|^{\frac{1}{q}}}\sum_{s\in s_{0}F\Delta F}\psi(s)Q_{s}\theta(s_{0})\right\|_{e}\\&=&
\left\|\frac{1}{|F|^{\frac{1}{q}}}\sum_{s\in s_{0}F\Delta F}\psi(s)Q_{s}\right\|_{e}.
\end{eqnarray*}
Since
\begin{eqnarray*}
\left\|\frac{1}{|F|^{\frac{1}{q}}}\sum_{s\in s_{0}F\Delta F}\psi(s)Q_{s}y\right\|&\leq&\sup_{s\in G}\|\psi(s)\|\frac{1}{|F|^{\frac{1}{q}}}
\sum_{s\in s_{0}F\Delta F}\|Q_{s}y\|\\&\leq&
\sup_{s\in G}\|\psi(s)\|\left(\frac{|s_{0}F\Delta F|}{|F|}\right)^{\frac{1}{q}}\left(\sum_{s\in s_{0}F\Delta F}\|Q_{s}y\|^{p}\right)^{\frac{1}{p}}\\&\leq&
\sup_{s\in G}\|\psi(s)\|\left(\frac{|s_{0}F\Delta F|}{|F|}\right)^{\frac{1}{q}}\|y\|,
\end{eqnarray*}
for all $y\in\mathcal{X}$, it follows that
\[\|E\theta(s_{0})-\psi(s_{0})E\|_{e}\leq\sup_{s\in G}\|\psi(s)\|\left(\frac{|s_{0}F\Delta F|}{|F|}\right)^{\frac{1}{q}}.\]
So by (\ref{varphithetaembedeq1}), we have
\[\|E\theta(s_{0})-\psi(s_{0})E\|_{e}\leq\sup_{s\in G}\|\psi(s)\|\epsilon^{\frac{1}{q}},\]
for all $s_{0}\in F_{0}$. Thus, (iv) is proved.
\end{proof}
Modulo some technicalities, the following result follows by combining Lemmas \ref{psietaembed3} and \ref{varphithetaembed}.
\begin{theorem}\label{main22}
Let $G$ be a countable amenable group. Let $\omega:G\to B(l^{p}(G))$ be the left regular representation, i.e., $\omega(s)e_{g}=e_{sg}$ for $s,g\in G$, where $(e_{g})_{g\in G}$ is the canonical basis for $l^{p}(G)$. Let $\mathcal{A}$ be the subalgebra of $B(l^{p}(G))$ generated by $\{\omega(s):s\in G\}$. Then $\mathrm{Ext}_{\sim,s}(\mathcal{A},K(l^{p}))$ is a group.
\end{theorem}
\begin{proof}
It suffices to show that for every isomorphic extension $\phi:\mathcal{A}\to B(l^{p})/K(l^{p})$, there exists an isomorphic extension $\phi^{(-1)}:\mathcal{A}\to B(l^{p})/K(l^{p})$ such that $\phi\oplus\phi^{(-1)}:\mathcal{A}\to B(l^{p}\oplus l^{p})/K(l^{p}\oplus l^{p})$ is a trivial isomorphic extension. Let $\psi:\mathcal{A}\to B(l^{p})$ be any map such that $\pi\circ\psi=\phi$ and $\|\psi(a)\|\leq\|\psi(a)\|_{e}+1=\|\phi(a)\|+1$ for all $a\in\mathcal{A}$. Then $\sup_{s\in G}\|\psi\circ\omega(s)\|\leq\|\phi\|+1<\infty$. By Lemma \ref{varphithetaembed}, we have $\psi\circ\omega\stackrel{\lambda}{\prec}(\theta)_{n\in\mathbb{N}}$ where $\lambda=(\|\phi\|+1)^{2}$.

By Lemma \ref{psietaembed3}, there exist a Banach space $\mathcal{Y}$ and a map $\zeta:G\to B(\mathcal{Y})$ such that $\oplus_{n=1}^{\infty}(\theta\oplus\theta)$ is similar to $(\psi\circ\omega)\oplus\zeta$ modulo compact operators. Note that $\theta$ is 1-similar to $\omega\oplus\omega\oplus\ldots$. Thus, there is a unital isometric homomorphism $\eta:\mathcal{A}\to B(l^{p})$ such that $\eta(\omega(s))=\theta(s)$ for all $s\in G$. Therefore, $\oplus_{n=1}^{\infty}(\eta\circ\omega\oplus\eta\circ\omega)$ is similar to $(\psi\circ\omega)\oplus\zeta$ modulo compact operators. Since $\mathcal{A}$ is generated by $\{\omega(s):s\in G\}$ and $\eta$ is a unital homomorphism, it follows that there is a map $\widetilde{\zeta}:\mathcal{A}\to B(\mathcal{Y})$ such that $\oplus_{n=1}^{\infty}(\eta\oplus\eta)$ is similar to $\psi\oplus\widetilde{\zeta}$ modulo compact operators. (In particular, $\pi\circ\widetilde{\zeta}$ is a unital homomorphism.) Thus, $\oplus_{n=1}^{\infty}(\eta\oplus\eta)$ is similar to $\psi\oplus(\widetilde{\zeta}\oplus\eta)$ modulo compact operators.

By Lemma \ref{prime}, the direct sum $\mathcal{Y}\oplus l^{p}$ is isomorphic to $l^{p}$. Let $S:\mathcal{Y}\oplus l^{p}\to l^{p}$ be any invertible operator. Define $\psi^{(-1)}:\mathcal{A}\to B(l^{p})$ by $\psi^{(-1)}(a)=S(\widetilde{\zeta}(a)\oplus\eta(a))S^{-1}$ for $a\in\mathcal{A}$. Then $\pi\circ\psi^{(-1)}:\mathcal{A}\to B(l^{p})/K(l^{p})$ is an isomorphic extension. Therefore, $(\pi\circ\psi)\oplus(\pi\circ\psi^{(-1)})$ is a trivial isomorphic extension. The result follows.
\end{proof}
\begin{corollary}
Suppose that either $\mathcal{A}=C(M)$, for some compact subset $M$ of a Euclidean space, or $\mathcal{A}$ is the subalgebra of $B(l^{p})$ generated by the range of the left regular representation of a countable amenable group $G$. Let $\phi:\mathcal{A}\to B(l^{p})/K(l^{p})$ be a unital homomorphism. Then there exist $L,R\in B(l^{p})$ and a unital homomorphism $\eta:\mathcal{A}\to B(l^{p})$ such that $\phi(a)=\pi(R\eta(a)L)$ and $\eta(a)LR-LR\eta(a)$ is compact for all $a\in\mathcal{A}$.
\end{corollary}
\begin{proof}
For each $a\in\mathcal{A}$, let $\psi(a)\in B(l^{p})$ be such that $\pi(\psi(a))=\phi(a)$. By Theorems \ref{main21} and \ref{main22}, there exist a unital homomorphism $\eta:\mathcal{A}\to B(l^{p})$, a map $\psi^{(-1)}:\mathcal{A}\to B(l^{p})/K(l^{p})$ and an invertible operator $S:l^{p}\oplus l^{p}\to l^{p}$ such that $\eta(a)-S(\psi(a)\oplus\psi^{(-1)}(a))S^{-1}$ modulo compact operators. Let $Q$ be the projection from $l^{p}\oplus l^{p}$ onto $l^{p}\oplus 0$. Take $L=S|_{l^{p}\oplus 0}$ and $R=QS^{-1}$. Then $\psi(a)-R\eta(a)L$ and $\eta(a)LR-LR\eta(a)$ are compact for all $a\in\mathcal{A}$. So the result follows.
\end{proof}
\section{Homotopy invariance of $\mathrm{Ext}_{\sim,s}(\mathcal{A},K(l^{p}))^{-1}$}
We write $(l^{p})^{(\infty)}=(l^{p}\oplus l^{p}\oplus\ldots)_{l^{p}}$ and for $T\in B(l^{p})$, we write $T^{(\infty)}=T\oplus T\oplus\ldots$.
\begin{lemma}\label{localize}
Let $A$ be a diagonal operator on $l^{p}$ with diagonal entries in $[0,1]$. Let $\epsilon>0$. Let $k\geq\frac{1}{\epsilon}$ be an integer. Then there exist operators $V:l^{p}\to\underbrace{l^{p}\oplus\ldots\oplus l^{p}}_{k+1}$ and $E:\underbrace{l^{p}\oplus\ldots\oplus l^{p}}_{k+1}\to l^{p}$ such that
\begin{enumerate}[(i)]
\item $\|V\|\leq2$ and $\|E\|\leq 2$;
\item $EV=I$;
\item $\|VA-(\frac{0}{k}I\oplus\ldots\oplus\frac{k}{k}I)V\|\leq\epsilon$ and $VT-(T\oplus\ldots\oplus T)V$ is compact;
\item $\|E(\frac{0}{k}I\oplus\ldots\oplus\frac{k}{k}I)-AE\|\leq\epsilon$ and $E(T\oplus\ldots\oplus T)-TE$ is compact; and
\item $V=(I\oplus f(A)\oplus\ldots\oplus f(A))V$ and $E=E(I\oplus f(A)\oplus\ldots\oplus f(A))$,
\end{enumerate}
for all $T\in B(l^{p})$ such that $TA-AT$ is compact and for all $f\in C[0,1]$ such that $f(t)=1$ for all $\frac{2}{5k}\leq t\leq 1$.
\end{lemma}
\begin{proof}
For $1\leq i\leq k+1$, let $U_{i}=(\frac{i-1.6}{k},\frac{i-0.4}{k})$. Note that $(U_{i})_{1\leq i\leq k+1}$ is an open cover of $[0,1]$. Let $(f_{i})_{1\leq i\leq k+1}$ be a partition of unity  on $[0,1]$ subordinate to $(U_{i})_{1\leq i\leq k+1}$. For each $1\leq i\leq k+1$, let $g_{i}:[0,1]\to[0,1]$ be a continuous function such that $\mathrm{supp}\,g_{i}\subset U_{i}$ and $g_{i}=1$ on $\mathrm{supp}\,f_{i}$.

Take
\[Vx=(f_{1}(A)x,\ldots,f_{k+1}(A)x),\]
for all $x\in l^{p}$, and
\[E(y_{1},\ldots,y_{k+1})=\sum_{i=1}^{k+1}g_{i}(A)y_{i},\]
for all $y_{1},\ldots,y_{k+1}\in l^{p}$.

Let $a_{1},a_{2},\ldots$ be the diagonal entries of $A$. Every $t\in[0,1]$ is contained in at most $2$ of the sets $U_{1},\ldots,U_{k+1}$. Note that
\[\mathrm{supp}\,g_{i}(A)=\{j\in\mathbb{N}:g_{i}(a_{j})\neq 0\}\subset\{j\in\mathbb{N}:a_{j}\in U_{i}\},\]
and $\mathrm{supp}\,f_{i}(A)\subset\mathrm{supp}\,g_{i}(A)$ for all $1\leq i\leq k+1$. Therefore, every $j\in\mathbb{N}$ is contained in at most $2$ of the sets $\mathrm{supp}\,f_{1}(A),\ldots,\mathrm{supp}\,f_{k+1}(A)$ and every $j\in\mathbb{N}$ is contained in at most $2$ of the sets $\mathrm{supp}\,g_{1}(A),\ldots,\mathrm{supp}\,g_{k+1}(A)$. By Lemma \ref{Embedding}, we have $\|V\|\leq 2$ and $\|E\|\leq 2$. Thus, (i) is proved.

Since $g_{i}f_{i}=f_{i}$ for all $1\leq i\leq k+1$, we have $\displaystyle EV=\sum_{i=1}^{k+1}g_{i}(A)f_{i}(A)=\sum_{i=1}^{k+1}f_{i}(A)=I$. Hence, (ii) is proved.

Since $|t-\frac{i-1}{k}|\leq\frac{1}{k}\leq\epsilon$ for all $t\in U_{i}$ and $\mathrm{supp}\,f_{i}\subset\mathrm{supp}\,g_{i}\subset U_{i}$,
\begin{equation}\label{localizeeq1}
\|(A-\frac{i-1}{k}I)g_{i}(A)\|\leq\epsilon\quad\text{and}\quad\|(A-\frac{i-1}{k}I)f_{i}(A)\|\leq\epsilon.
\end{equation}
Note that
\[VAx-(\frac{0}{k}I\oplus\ldots\oplus\frac{k}{k}I)Vx=((A-\frac{0}{k}I)f_{1}(A)x,\ldots,(A-\frac{k}{k}I)f_{k}(A)x),\]
for all $x\in l^{p}$, and
\[(E(\frac{0}{k}I\oplus\ldots\oplus\frac{0}{k}I)-AE)(y_{1},\ldots,y_{k+1})=\sum_{i=1}^{k+1}(A-\frac{i-1}{k}I)g_{i}(A)y_{i},\]
for all $y_{1},\ldots,y_{k+1}\in l^{p}$. By (\ref{localizeeq1}) and Lemma \ref{Embedding}, we obtain
\[\|VA-(\frac{0}{k}I\oplus\ldots\oplus\frac{k}{k}I)V\|\leq\epsilon\quad\text{and}
\quad\|E(\frac{0}{k}I\oplus\ldots\oplus\frac{k}{k}I)-AE\|\leq\epsilon.\]
If $T$ is an operator on $l^{p}$ such that $TA-AT$ is compact, then $Tf_{i}(A)-f_{i}(A)T$ and $Tg_{i}(A)-g_{i}(A)T$ are compact for all $1\leq i\leq k+1$. Thus, $VT-(T\oplus\ldots\oplus T)V$ and $E(T\oplus\ldots\oplus T)-TE$ are compact. So (iii) and (iv) are proved.

For $2\leq i\leq k+1$, we have $U_{i}\subset(\frac{0.4}{k},\infty)$ and so
\[\mathrm{supp}\,f_{i}(A)\subset\mathrm{supp}\,g_{i}(A)\subset\{j\in\mathbb{N}:a_{j}\in U_{i}\}\subset\{j\in\mathbb{N}:a_{j}>\frac{2}{5k}\}.\]
Thus, $f_{i}(A)=f(A)f_{i}(A)$ and $g_{i}(A)=g_{i}(A)f(A)$ for all $2\leq i\leq k+1$ and $f\in C[0,1]$ such that $f(t)=1$ for all $\frac{2}{5k}\leq t\leq 1$. So (v) is proved.
\end{proof}
\begin{lemma}\label{homotopyembed}
Let $A$ be a diagonal operator on $l^{p}$ with diagonal entries in $[0,1]$. Let $\mathcal{B}$ be the commutant of $\pi(A)$ in $B(l^{p})/K(l^{p})$. Let $\Lambda$ be a set. For $\alpha\in\Lambda$, let $h_{\alpha}:[0,1]\to\mathcal{B}$ be a $\pi(A)$-continuous function. For $t\in[0,1]$, let $\rho_{t}:\Lambda\to B(l^{p})$. Let $W_{1},W_{2}\in B(l^{p})$ be such that $W_{2}W_{1}=I$ and $\|W_{1}\|=\|W_{2}\|=1$. Suppose that
\begin{enumerate}[(a)]
\item $(\pi(W_{1})h_{\alpha}(t)-\pi(\rho_{t}(\alpha)W_{1}))\pi(A)=0$ and $\pi(A)(\pi(W_{2}\rho_{t}(\alpha))-h_{\alpha}(t)\pi(W_{2}))=0$ for all $t\in[0,1]$ and $\alpha\in\Lambda$; and
\item $\pi(W_{1})h_{\alpha}(0)=\pi(\rho_{0}(\alpha)W_{1})$ and $\pi(W_{2}\rho_{0}(\alpha))=h_{\alpha}(0)\pi(W_{2})$ for all $\alpha\in\Lambda$.
\end{enumerate}
Let $\xi:\Lambda\to B(l^{p})$ be a map such that $\pi\circ\xi(\alpha)=h_{\alpha}(\pi(A))$ for $\alpha\in\Lambda$. Then $\xi\stackrel{4}{\prec}(\oplus_{r=0}^{k}\rho_{\frac{r}{k}})_{k\in\mathbb{N}}$.
\end{lemma}
\begin{proof}
By Lemma \ref{localize}, there exist operators $V_{k}:l^{p}\to\underbrace{l^{p}\oplus\ldots\oplus l^{p}}_{k+1}$ and $E_{k}:\underbrace{l^{p}\oplus\ldots\oplus l^{p}}_{k+1}\to l^{p}$, for $k\in\mathbb{N}$, such that
\begin{enumerate}[(i)]
\item $\|V_{k}\|\leq2$ and $\|E_{k}\|\leq 2$ for all $k\in\mathbb{N}$;
\item $E_{k}V_{k}=I$ for all $k\in\mathbb{N}$;
\item $\lim_{k\to\infty}\|V_{k}A-(\frac{0}{k}I\oplus\ldots\oplus\frac{k}{k}I)V_{k}\|=0$ and $V_{k}T-(T\oplus\ldots\oplus T)V_{k}$ is compact for all $k\in\mathbb{N}$;
\item $\lim_{k\to\infty}\|E_{k}(\frac{0}{k}I\oplus\ldots\oplus\frac{k}{k}I)-AE_{k}\|=0$ and $E_{k}(T\oplus\ldots\oplus T)-TE_{k}$ is compact for all $k\in\mathbb{N}$; and
\item $V_{k}=(I\oplus f(A)\oplus\ldots\oplus f(A))V_{k}$ and $E_{k}=E_{k}(I\oplus f(A)\oplus\ldots\oplus f(A))$,
\end{enumerate}
for all $T\in B(l^{p})$ such that $TA-AT$ is compact and for all $f\in C[0,1]$ such that $f(t)=1$ for all $\frac{2}{5k}\leq t\leq 1$. For each $\alpha\in\Lambda$, since $h_{\alpha}$ is $\pi(A)$-continuous, there exists a continuous function $\widetilde{h}_{\alpha}:[0,1]\to\mathcal{B}$ such that $(\widetilde{h}_{\alpha}(t)-h_{\alpha}(t))\pi(A)=0$, for all $t\in[0,1]$, and $\widetilde{h}_{\alpha}(0)=h_{\alpha}(0)$. Since $\widetilde{h}_{\alpha}$ can approximated by $\mathcal{B}$-simple functions, by (iii) and (iv), we have
\begin{equation}\label{homotopyembedeq1}
\lim_{k\to\infty}\|\pi(V_{k})\widetilde{h}_{\alpha}(\pi(A))-(\widetilde{h}_{\alpha}(\frac{0}{k})\oplus\ldots\oplus \widetilde{h}_{\alpha}(\frac{k}{k}))\pi(V_{k})\|=0,
\end{equation}
and
\begin{equation}\label{homotopyembedeq2}
\lim_{k\to\infty}\|\pi(E_{k})(\widetilde{h}_{\alpha}(\frac{0}{k})\oplus\ldots\oplus\widetilde{h}_{\alpha}(\frac{k}{k}))-
\widetilde{h}_{\alpha}(\pi(A))\pi(E_{k})\|=0,
\end{equation}
for all $\alpha\in\Lambda$. For $k\in\mathbb{N}$, fix $f_{k}\in C[0,1]$ such that $f_{k}(0)=0$ and $f_{k}(t)=1$ for all $\frac{2}{5k}\leq t\leq 1$. Since $f_{k}(0)=0$, by (a), we have
\begin{equation}\label{homotopyembedeq3}
(\pi(W_{1})h_{\alpha}(t)-\pi(\rho_{t}(\alpha)W_{1}))\pi(f_{k}(A))=0,
\end{equation}
and
\begin{equation}\label{homotopyembedeq4}
\pi(f_{k}(A))(\pi(W_{2}\rho_{t}(\alpha))-h_{\alpha}(t)\pi(W_{2}))=0,
\end{equation}
for all $t\in[0,1]$ and $\alpha\in\Lambda$. Since $(\widetilde{h}_{\alpha}(t)-h_{\alpha}(t))\pi(A)=0$ and $f_{k}(0)=0$, we have
\begin{equation}\label{homotopyembedeq5}
(\widetilde{h}_{\alpha}(t)-h_{\alpha}(t))\pi(f_{k}(A))=\pi(f_{k}(A))(\widetilde{h}_{\alpha}(t)-h_{\alpha}(t))=0.
\end{equation}
 So
\begin{align*}
&\left[\pi(W_{1})\widetilde{h}_{\alpha}\left(\frac{0}{k}\right)\oplus\pi(W_{1})\widetilde{h}_{\alpha}\left(\frac{1}{k}\right)\oplus\ldots\oplus
\pi(W_{1})\widetilde{h}_{\alpha}\left(\frac{k}{k}\right)\right]\pi(V_{k})\\=&
\left[\pi(W_{1})\widetilde{h}_{\alpha}\left(\frac{0}{k}\right)\oplus\pi(W_{1})\widetilde{h}_{\alpha}\left(\frac{1}{k}\right)\pi(f_{k}(A))\oplus\ldots
\oplus\pi(W_{1})\widetilde{h}_{\alpha}\left(\frac{k}{k}\right)\pi(f_{k}(A))\right]\pi(V_{k})\text{ by (v)}\\=&
\left[\pi(W_{1})h_{\alpha}\left(\frac{0}{k}\right)\oplus\pi(W_{1})h_{\alpha}\left(\frac{1}{k}\right)\pi(f_{k}(A))\oplus\ldots
\oplus\pi(W_{1})h_{\alpha}\left(\frac{k}{k}\right)\pi(f_{k}(A))\right]\pi(V_{k})\\&\text{by }(\ref{homotopyembedeq5})\text{ and }\widetilde{h}_{\alpha}(0)=h_{\alpha}(0)\\=&
\left[\pi(\rho_{\frac{0}{k}}(\alpha)W_{1})\oplus\pi(\rho_{\frac{1}{k}}(\alpha)W_{1}f_{k}(A))\oplus\ldots\oplus
\pi(\rho_{\frac{k}{k}}(\alpha)W_{1}f_{k}(A))\right]\pi(V_{k})\text{ by }(\ref{homotopyembedeq3})\text{ and (b)}\\=&
\left[\pi(\rho_{\frac{0}{k}}(\alpha)W_{1})\oplus\pi(\rho_{\frac{1}{k}}(\alpha)W_{1})\oplus\ldots\oplus\pi(\rho_{\frac{k}{k}}(\alpha)W_{1})
\right]\pi(V_{k})\text{ by (v)}
\end{align*}
and
\begin{align*}
&\pi(E_{k})\left[\widetilde{h}_{\alpha}\left(\frac{0}{k}\right)\pi(W_{2})\oplus\widetilde{h}_{\alpha}\left(\frac{1}{k}\right)\pi(W_{2})\oplus\ldots
\oplus\widetilde{h}_{\alpha}\left(\frac{k}{k}\right)\pi(W_{2})\right]\\=&
\pi(E_{k})\left[\widetilde{h}_{\alpha}\left(\frac{0}{k}\right)\pi(W_{2})\oplus\pi(f_{k}(A))\widetilde{h}_{\alpha}\left(\frac{1}{k}\right)\pi(W_{2})
\oplus\ldots\oplus\pi(f_{k}(A))\widetilde{h}_{\alpha}\left(\frac{k}{k}\right)\pi(W_{2})\right]\text{ by (v)}\\=&
\pi(E_{k})\left[h_{\alpha}\left(\frac{0}{k}\right)\pi(W_{2})\oplus\pi(f_{k}(A))h_{\alpha}\left(\frac{1}{k}\right)\pi(W_{2})
\oplus\ldots\oplus\pi(f_{k}(A))h_{\alpha}\left(\frac{k}{k}\right)\pi(W_{2})\right]\\&\text{by }(\ref{homotopyembedeq5})\text{ and }\widetilde{h}_{\alpha}(0)=h_{\alpha}(0)\\=&
\pi(E_{k})\left[\pi(W_{2}\rho_{\frac{0}{k}}(\alpha))\oplus\pi(f_{k}(A)W_{2}\rho_{\frac{1}{k}}(\alpha))\oplus\ldots\oplus \pi(f_{k}(A)W_{2}\rho_{\frac{k}{k}}(\alpha))\right]\text{ by }(\ref{homotopyembedeq4})\text{ and (b)}\\=&
\pi(E_{k})\left[\pi(W_{2}\rho_{\frac{0}{k}}(\alpha))\oplus\pi(W_{2}\rho_{\frac{1}{k}}(\alpha))\oplus\ldots\oplus \pi(W_{2}\rho_{\frac{k}{k}}(\alpha))\right]\text{ by (v)}
\end{align*}
Therefore, by (\ref{homotopyembedeq1}) and (\ref{homotopyembedeq2}), we have
\begin{align*}
&\lim_{k\to\infty}\|(\pi(W_{1})\oplus\ldots\oplus\pi(W_{1}))\pi(V_{k})\widetilde{h}_{\alpha}(\pi(A))\\&-
(\pi(\rho_{\frac{0}{k}}(\alpha))\oplus\ldots\oplus\pi(\rho_{\frac{k}{k}}(\alpha)))(\pi(W_{1})\oplus\ldots\oplus\pi(W_{1}))\pi(V_{k})\|=0,
\end{align*}
and
\begin{align*}
&\lim_{k\to\infty}\|\pi(E_{k})(\pi(W_{2})\oplus\ldots\oplus\pi(W_{2}))(\pi(\rho_{\frac{0}{k}}(\alpha))\oplus\ldots\oplus \pi(\rho_{\frac{k}{k}}(\alpha)))\\&
-\widetilde{h}_{\alpha}(\pi(A))\pi(E_{k})(\pi(W_{2})\oplus\ldots\oplus\pi(W_{2}))\|=0.
\end{align*}
By definition, $h_{\alpha}(\pi(A))=\widetilde{h}_{\alpha}(\pi(A))$. By (i), we have $\|(\pi(W_{1})\oplus\ldots\oplus\pi(W_{1}))\pi(V_{k})\|\leq 2$ and $\|\pi(E_{k})(\pi(W_{1})\oplus\ldots\oplus\pi(W_{1}))\|\leq 2$ for all $k\in\mathbb{N}$. By (ii),
\[\pi(E_{k})(\pi(W_{2})\oplus\ldots\oplus\pi(W_{2}))(\pi(W_{1})\oplus\ldots\oplus\pi(W_{1}))\pi(V_{k})=\pi(E_{k}V_{k})=1,\]
for all $k\in\mathbb{N}$. Therefore, $\xi\stackrel{4}{\prec}(\oplus_{r=0}^{k}\rho_{\frac{r}{k}})_{k\in\mathbb{N}}$.
\end{proof}
\begin{lemma}\label{homotopyinv1}
Let $W_{1},W_{2}\in B(l^{p})$ be such that $W_{2}W_{1}=I$ and $\|W_{1}\|=\|W_{2}\|=1$. Let $\Lambda$ be a countable set. For $t\in[0,1]$, let $\psi_{t}:\Lambda\to B(l^{p})$ and $\rho_{t}:\Lambda\to B(l^{p})$. Suppose that
\begin{enumerate}[(a)]
\item $t\mapsto\pi\circ\psi_{t}(\alpha)$ is a continuous function from $[0,1]$ to $B(l^{p})/K(l^{p})$ for every $\alpha\in\Lambda$;
\item $W_{1}\psi_{t}(\alpha)-\rho_{t}(\alpha)W_{1}$ and $W_{2}\rho_{t}(\alpha)-\psi_{t}(\alpha)W_{2}$ are compact for all $t\in[0,1]$ and $\alpha\in\Lambda$; and
\item $W_{1}\psi_{0}(\alpha)=\rho_{0}(\alpha)W_{1}$ and $W_{2}\rho_{0}(\alpha)=\psi_{0}(\alpha)W_{2}$.
\end{enumerate}
Then there are rational numbers $u_{1},u_{2},\ldots$ such that $\psi_{1}\oplus(\oplus_{k\in\mathbb{N}}\rho_{u_{k}})$ is $(2500^{2}\cdot 4^{8})$-similar to $\oplus_{k\in\mathbb{N}}\rho_{u_{k}}$ modulo compact operators.
\end{lemma}
\begin{proof}
Since $t\mapsto\pi\circ\psi_{t}(\alpha)$ is continuous for all $\alpha\in\Lambda$, the set $\{\pi(\psi_{t}(\alpha)):t\in[0,1],\;\alpha\in\Lambda\}$ in $B(l^{p})/K(l^{p})$ is separable. By Lemma \ref{qcau}, there are finite rank diagonal operators $A_{1},A_{2},\ldots$ on $l^{p}$ with diagonal entries in $[0,1]$ such that
\begin{enumerate}[(i)]
\item $\|A_{n}\psi_{t}(\alpha)-\psi_{t}(\alpha)A_{n}\|\to 0$, as $n\to\infty$, for all $t\in[0,1]$ and $\alpha\in\Lambda$; and
\item $\mathrm{supp}(I-A_{1})\supset\mathrm{supp}(I-A_{2})\supset$ and $\cap_{n=1}^{\infty}\mathrm{supp}(I-A_{n})=\emptyset$.
\end{enumerate}
Replacing the sequence $A_{1},A_{2},\ldots$ by $A_{1},\frac{1}{2}A_{1}+\frac{1}{2}A_{2},A_{2},\frac{2}{3}A_{2}+\frac{1}{3}A_{3},\frac{1}{3}A_{2}+\frac{2}{3}A_{3},A_{3},\frac{3}{4}A_{3}+
\frac{1}{4}A_{4},\ldots$, we may assume that $\|A_{n}-A_{n+1}\|\to 0$ as $n\to\infty$.

Take $A=(I-A_{1})\oplus(I-A_{2})\oplus\ldots$. By (i), we have that $A(\psi_{t}(\alpha)^{(\infty)})-(\psi_{t}(\alpha)^{(\infty)})A$ is compact for all $t\in[0,1]$ and $\alpha\in\Lambda$. Let $\mathcal{B}$ be the commutant of $\pi(A)$ in $B((l^{p})^{(\infty)})/K((l^{p})^{(\infty)})$. For $\alpha\in\Lambda$, define $h_{\alpha}:[0,1]\to\mathcal{B}$ by $h_{\alpha}(t)=\pi(\psi_{t}(\alpha)^{(\infty)})$ for $t\in[0,1]$. We claim that $h_{\alpha}$ is $\pi(A)$-continuous. For $n\in\mathbb{N}$, let $P_{n}$ be the projection from $l^{p}$ onto $l^{p}(\mathrm{supp}(I-A_{n}))$. Then $P_{n}\to 0$ in SOT as $n\to\infty$. Note that $P=P_{1}\oplus P_{2}\oplus\ldots$ is the projection from $l^{p}\oplus l^{p}\oplus\ldots$ onto $l^{p}(\mathrm{supp}\,A)$. By (a), the map $t\mapsto\pi(\psi_{t}(\alpha)^{(\infty)}P)$ from $[0,1]$ to $B(l^{p})/K(l^{p})$ is continuous for every $\alpha\in\Lambda$. Thus,
\[\widetilde{h}_{\alpha}(t)=\pi(\psi_{0}(\alpha)^{(\infty)}+(\psi_{t}(\alpha)^{(\infty)}-\psi_{0}(\alpha)^{(\infty)})P)\]
defines a continuous function from $[0,1]$ to $B((l^{p})^{(\infty)})/K((l^{p})^{(\infty)})$. As explained above, $\pi(A)$ commutes with $\pi(\psi_{t}(\alpha)^{(\infty)})$ for all $t\in[0,1]$. Since $P$ is the projection onto $l^{p}(\mathrm{supp}\,A)$, we have $AP=A=PA$. So $\pi(A)$ commutes with $\widetilde{h}_{\alpha}(t)$ and
\begin{eqnarray*}
\widetilde{h}_{\alpha}(t)\pi(A)&=&\pi(\psi_{0}(\alpha)^{(\infty)}A+(\psi_{t}(\alpha)^{(\infty)}-\psi_{0}(\alpha)^{(\infty)})PA)\\&=&
\pi(\psi_{0}(\alpha)^{(\infty)}A+(\psi_{t}(\alpha)^{(\infty)}-\psi_{0}(\alpha)^{(\infty)})A)
\\&=&\pi(\psi_{t}(\alpha)^{(\infty)})\pi(A)=h_{\alpha}(t)\pi(A).
\end{eqnarray*}
Also $\widetilde{h}_{\alpha}(0)=\pi(\psi_{0}(\alpha)^{(\infty)})=h_{\alpha}(0)$. Therefore, $h_{\alpha}:[0,1]\to\mathcal{B}$ is a $\pi(A)$-continuous function.

Since $I-A_{n}\to 0$ in SOT, by (b), the operators
$(W_{1}^{(\infty)}\psi_{t}(\alpha)^{(\infty)}-\rho_{t}(\alpha)^{(\infty)}W_{1}^{(\infty)})A$ and $A(W_{2}^{(\infty)}\rho_{t}(\alpha)^{(\infty)}-\psi_{t}(\alpha)^{(\infty)}W_{2}^{(\infty)})$ are compact. So
\[(\pi(W_{1}^{(\infty)})h_{\alpha}(t)-\pi(\rho_{t}(\alpha)^{(\infty)}W_{1}^{(\infty)}))\pi(A)=0,\]
and
\[\pi(A)(\pi(W_{2}^{(\infty)}\rho_{t}(\alpha)^{(\infty)})-h_{\alpha}(t)\pi(W_{2}^{(\infty)}))=0.\]
By (c), we have
\[\pi(W_{1}^{(\infty)})h_{\alpha}(0)=\pi(W_{1}^{(\infty)}\psi_{0}(\alpha)^{(\infty)})=\pi(\rho_{0}(\alpha)^{(\infty)}W_{1}^{(\infty)}),\]
and
\[\pi(W_{2}^{(\infty)}\rho_{0}(\alpha)^{(\infty)})=\pi(\psi_{0}(\alpha)^{(\infty)}W_{2}^{(\infty)})=h_{\alpha}(0)\pi(W_{2}^{(\infty)}).\]
Let $\xi:\Lambda\to B((l^{p})^{(\infty)})$ be a map such that $\pi\circ\xi(\alpha)=h_{\alpha}(\pi(A))$ for $\alpha\in\Lambda$. By Lemma \ref{homotopyembed}, we have $\xi\stackrel{4}{\prec}(\oplus_{r=0}^{k}\rho_{\frac{r}{k}})_{k\in\mathbb{N}}$. By Lemma \ref{psietaembed3}, there are rational numbers $u_{1},u_{2},\ldots$ in $[0,1]$, a Banach space $\mathcal{Y}$ and a map $\zeta:\Lambda\to B(\mathcal{Y})$ such that $\oplus_{k\in\mathbb{N}}\rho_{u_{k}}$ is $(2500\cdot 4^{4})$-similar to $\xi\oplus\zeta$ modulo compact operators.

Since $\|A_{n}-A_{n+1}\|\to 0$ as $n\to\infty$, the operator $A-I\oplus A$ is compact so $\pi(A)=\pi(I)\oplus\pi(A)$ commute. Since $h_{\alpha}(t)=\pi(\psi_{t}(\alpha))\oplus h_{\alpha}(t)$ for all $t\in[0,1]$, by Lemma \ref{calculusdirectsum2}, we have $h_{\alpha}(\pi(A))=h_{\alpha}(\pi(I)\oplus\pi(A))=\pi(\psi_{1}(\alpha))\oplus h_{\alpha}(\pi(A))$. So $\xi$ and $\psi_{1}\oplus\xi$ coincide modulo compact operators. So $\xi\oplus\zeta$ and $\psi_{1}\oplus\xi\oplus\zeta$ coincide modulo compact operators. Therefore, from the conclusion of the previous paragraph, $\oplus_{k\in\mathbb{N}}\rho_{u_{k}}$ is $(2500^{2}\cdot 4^{8})$-similar to $\psi_{1}\oplus(\oplus_{k\in\mathbb{N}}\rho_{u_{k}})$ modulo compact operators.
\end{proof}
\begin{theorem}\label{main3}
Let $\Lambda$ be a countable set. For $t\in[0,1]$, let $\psi_{t}:\Lambda\to B(l^{p})$, $\zeta_{t}:\Lambda\to B(l^{p})$ and $\rho_{t}:\Lambda\to B(l^{p}\oplus l^{p})$. Suppose that
\begin{enumerate}[(a)]
\item $t\mapsto\pi\circ\psi_{t}(\alpha)$ is a continuous function from $[0,1]$ to $B(l^{p})/K(l^{p})$ for every $\alpha\in\Lambda$; and
\item $\rho_{t}(\alpha)-(\psi_{t}(\alpha)\oplus\zeta_{t}(\alpha))$ is compact for all $t\in[0,1]$ and $\alpha\in\Lambda$.
\end{enumerate}
Then there are rational numbers $u_{1},u_{2},\ldots$ such that $\psi_{1}\oplus(\oplus_{k\in\mathbb{N}}\rho_{u_{k}})$ is $(2500^{2}\cdot 4^{8})$-similar to $\psi_{0}\oplus(\oplus_{k\in\mathbb{N}}\rho_{u_{k}})$ modulo compact operators.
\end{theorem}
\begin{proof}
Define $W\in B(l^{p}\oplus l^{p}\oplus l^{p}\oplus l^{p})$ by
\[W(x_{1},x_{2},x_{3},x_{4})=(x_{1},x_{4},x_{3},x_{2}),\]
for $x_{1},x_{2},x_{3},x_{4}\in l^{p}$. For $t\in[0,1]$ and $\alpha\in\Lambda$, let $\kappa_{0}(\alpha)=\rho_{0}(\alpha)-(\psi_{0}(\alpha)\oplus\zeta_{0}(\alpha))\in K(l^{p}\oplus l^{p})$,
\[\widehat{\psi}_{t}(\alpha)=\kappa_{0}(\alpha)+(\psi_{t}(\alpha)\oplus\zeta_{0}(\alpha)),\]
and
\[\widehat{\rho}_{t}(\alpha)=\begin{cases}W(\rho_{t}(\alpha)\oplus\rho_{0}(\alpha))W^{-1},&t>0\\\rho_{0}(\alpha)\oplus\rho_{0}(\alpha),&t=0
\end{cases}.\]
Define $W_{1}:l^{p}\oplus l^{p}\to l^{p}\oplus l^{p}\oplus l^{p}\oplus l^{p}$ and $W_{2}:l^{p}\oplus l^{p}\oplus l^{p}\oplus l^{p}\to l^{p}\oplus l^{p}$ by
\[W_{1}(x_{1},x_{2})=(x_{1},x_{2},0,0)\quad\text{and}\quad W_{2}(x_{1},x_{2},x_{3},x_{4})=(x_{1},x_{2}),\]
for $x_{1},x_{2},x_{3},x_{4}\in l^{p}$. Since $\widehat{\rho}_{t}(\alpha)-(\psi_{t}(\alpha)\oplus\zeta_{0}(\alpha)\oplus\psi_{0}(\alpha)\oplus\zeta_{t}(\alpha))$ is compact, by (b), we have that $W_{1}\widehat{\psi}_{t}(\alpha)-\widehat{\rho}_{t}(\alpha)W_{1}$ and $W_{2}\widehat{\rho}_{t}(\alpha)-\widehat{\psi}_{t}(\alpha)W_{2}$ are compact for all $t\in[0,1]$ and $\alpha\in\Lambda$. Also since $\widehat{\psi}_{0}(\alpha)=\rho_{0}(\alpha)\in B(l^{p}\oplus l^{p})$, we have $W_{1}\widehat{\psi}_{0}(\alpha)=\widehat{\rho}_{0}(\alpha)W_{1}$ and $W_{2}\widehat{\rho}_{0}(\alpha)=\widehat{\psi}_{0}(\alpha)W_{2}$ for all $\alpha\in\Lambda$. Note that $t\mapsto\pi\circ\widehat{\psi}_{t}(\alpha)$ is continuous for every $\alpha\in\Lambda$. By Lemma \ref{homotopyinv1}, it follows that there are rational numbers $u_{1},u_{2},\ldots$ in $[0,1]$ such that $\widehat{\psi}_{1}\oplus(\oplus_{k\in\mathbb{N}}\widehat{\rho}_{u_{k}})$ is $(2500^{2}\cdot 4^{8})$-similar to $\oplus_{k\in\mathbb{N}}\widehat{\rho}_{u_{k}}$ modulo compact operators. Thus, $\psi_{1}\oplus\zeta_{0}\oplus(\oplus_{k\in\mathbb{N}}\rho_{u_{k}})\oplus(\oplus_{k\in\mathbb{N}}\rho_{0})$ is $(2500^{2}\cdot 4^{8})$-similar to $(\oplus_{k\in\mathbb{N}}\rho_{u_{k}})\oplus(\oplus_{k\in\mathbb{N}}\rho_{0})$ modulo compact operators. Since $\psi_{0}\oplus\zeta_{0}$ and $\rho_{0}$ coincide modulo compact operators, it follows that $\psi_{1}\oplus(\oplus_{k\in\mathbb{N}}\rho_{u_{k}})\oplus(\oplus_{k\in\mathbb{N}}\rho_{0})$ is $(2500^{2}\cdot 4^{8})$-similar to $\psi_{0}\oplus(\oplus_{k\in\mathbb{N}}\rho_{u_{k}})\oplus(\oplus_{k\in\mathbb{N}}\rho_{0})$ modulo compact operators.
\end{proof}
\begin{corollary}
Let $U$ and $B$ be the unilateral and bilateral shifts on $l^{p}$, respectively. Then there are rational numbers $u_{1},u_{2},\ldots$ in $[0,1]$ such that $U\oplus(\oplus_{k\in\mathbb{N}}u_{k}B)$ is similar to a compact perturbation of $\oplus_{k\in\mathbb{N}}u_{k}B$.
\end{corollary}
\begin{proof}
Let $U^{-1}$ be the backward shift on $l^{p}$. Then $U\oplus U^{-1}$ is similar to a rank one perturbation of $B$. In Theorem \ref{main3}, take $\Lambda$ to be a singleton, take $\psi_{t}(\alpha)=tU$, take $\zeta_{t}(\alpha)=tU^{-1}$ and take $\rho_{t}(\alpha)=tB$. The result follows.
\end{proof}
With the notation at the beginning of Section 8, for a unital Banach algebra $\mathcal{A}$ that is isomorphic to a subalgebra of $l^{p}$, if $\phi:\mathcal{A}\to B(l^{p})/K(l^{p})$ is an isomorphic extension, we denote the image of $[\phi]$ in $\mathrm{Ext}_{\sim,s}(\mathcal{A},K(l^{p}))$ by $[\![\phi]\!]$.

Let $\mathcal{A}_{1},A_{2}$ be unital Banach algebras that are isomorphic to subalgebras of $B(l^{p})$. If $\tau:\mathcal{A}_{1}\to\mathcal{A}_{2}$ is a unital homomorphism, $\phi:\mathcal{A}_{2}\to B(l^{p})/K(l^{p})$ is an isomorphic extension and $\theta:\mathcal{A}_{1}\to B(l^{p})/K(l^{p})$ is a trivial isomorphic extension, then the map $(\phi\circ\tau)\oplus\theta:\mathcal{A}_{1}\to B(l^{p})/K(l^{p})$, defined as $a\mapsto\phi(\tau(a))\oplus\theta(a)$, is an isomorphic extension. The map $[\![\phi]\!]\to[\![(\phi\circ\tau)\oplus\theta]\!]$ is a well defined homomorphism from the semigroup $\mathrm{Ext}_{\sim,s}(\mathcal{A}_{2},K(l^{p}))$ into the semigroup $\mathrm{Ext}_{\sim,s}(\mathcal{A}_{1},K(l^{p}))$. We denote this homomorphism by $\tau_{*}$
\begin{corollary}\label{homotopyinv2}
Let $\mathcal{A}_{1},\mathcal{A}_{2}$ be separable unital Banach algebra that is isomorphic to subalgebras of $B(l^{p})$. Let $\phi:\mathcal{A}_{2}\to B(l^{p})/K(l^{p})$ be an isomorphic extension such that $[\![\phi]\!]$ is invertible in $\mathrm{Ext}_{\sim,s}(\mathcal{A}_{2},K(l^{p}))$. For $t\in[0,1]$, let $\tau_{t}:\mathcal{A}_{1}\to\mathcal{A}_{2}$ be a unital homomorphism. Suppose that $t\mapsto\tau_{t}(a)$ is a continuous function from $[0,1]$ to $\mathcal{A}_{2}$ for every $a\in\mathcal{A}_{1}$. Then $(\tau_{1})_{*}[\![\phi]\!]=(\tau_{0})_{*}[\![\phi]\!]$.
\end{corollary}
\begin{proof}
Let $\Lambda$ be a countable dense subset of $\mathcal{A}_{1}$. For $t\in[0,1]$, let $\psi_{t}:\Lambda\to B(l^{p})$ be any map such that $\pi\circ\psi_{t}(a)=\phi\circ\tau_{t}(a)$ for all $a\in\Lambda$. Since $[\![\phi]\!]$ is invertible in $\mathrm{Ext}_{\sim,s}(\mathcal{A}_{2},K(l^{p}))$, there exist an isomorphic extension $\phi^{(-1)}:\mathcal{A}_{2}\to B(l^{p})/K(l^{p})$ and a unital homomorphism $\rho:\mathcal{A}_{2}\to B(l^{p}\oplus l^{p})/K(l^{p}\oplus l^{p})$ such that $\phi(a)\oplus\phi^{(-1)}(a)=\pi\circ\rho(a)$ for $a\in\mathcal{A}_{2}$. By Theorem \ref{main3}, there are rational numbers $u_{1},u_{2},\ldots$ in $[0,1]$ such that $\psi_{1}\oplus(\oplus_{k\in\mathbb{N}}\rho\circ\tau_{u_{k}}|_{\Lambda})$ is similar to $\psi_{0}\oplus(\oplus_{k\in\mathbb{N}}\rho\circ\tau_{u_{k}}|_{\Lambda})$ modulo compact operators. Since $\Lambda$ is dense in $\mathcal{A}_{1}$, it follows that $(\phi\circ\tau_{1})\oplus(\oplus_{k\in\mathbb{N}}\rho\circ\tau_{u_{k}})$ is similar to $(\phi\circ\tau_{0})\oplus(\oplus_{k\in\mathbb{N}}\rho\circ\tau_{u_{k}})$ modulo compact operators. Since $\oplus_{k\in\mathbb{N}}\rho\circ\tau_{u_{k}}$ is a unital homomorphism from $\mathcal{A}_{1}$ to $B((l^{p})^{(\infty)})$, we conclude that $(\tau_{1})_{*}[\![\phi]\!]=(\tau_{0})_{*}[\![\phi]\!]$.
\end{proof}
\begin{corollary}\label{homotopyinv3}
Let $M$ be a contractible subset of a Euclidean space. Let $\phi:C(M)\to B(l^{p})/K(l^{p})$ be an isomorphic extension. Then there is a trivial isomorphic extension $\theta:C(M)\to B(l^{p})/K(l^{p})$ such that $[\theta]=[\phi]+[\theta]$.
\end{corollary}
\begin{proof}
By Theorem \ref{main21}, the semigroup $\mathrm{Ext}_{\sim,s}(C(M),K(l^{p}))$ is a group. Since $M$ is contractible, by Corollary \ref{homotopyinv2}, this group is trivial. So the result follows.
\end{proof}
\noindent{\bf Acknowledgements:} The author is grateful to William B. Johnson for useful discussions.

\end{document}